\pdfoutput=1
\newif\ifpersonal
\RequirePackage[l2tabu,orthodox]{nag} 
\documentclass[10pt]{amsart} 
\usepackage{amsmath,amsthm,amssymb,mathrsfs,mathtools,bm,eucal,tensor} 
\usepackage{microtype,lmodern} 
\usepackage[utf8]{inputenc} 
\usepackage[T1]{fontenc} 
\usepackage{enumerate,comment,braket,xspace,tikz-cd,csquotes} 
\usepackage[centering,vscale=0.7,hscale=0.7]{geometry}
\usepackage[hidelinks]{hyperref}
\usepackage[capitalize]{cleveref}

\allowdisplaybreaks

\numberwithin{equation}{section}
\theoremstyle{plain}
\newtheorem{theorem}[equation]{Theorem}
\newtheorem{lemma}[equation]{Lemma}

\newtheorem*{claim*}{Claim}

\newtheorem{proposition}[equation]{Proposition}
\newtheorem{corollary}[equation]{Corollary}
\theoremstyle{definition}
\newtheorem{definition}[equation]{Definition}
\newtheorem{construction}[equation]{Construction}

\newtheorem{example}[equation]{Example}
\newtheorem{remark}[equation]{Remark}

\makeatletter
\@namedef{subjclassname@2020}{%
	\textup{2020} Mathematics Subject Classification}
\makeatother

\ifpersonal
\newcommand{\personal}[1]{\textcolor[rgb]{0,0,1}{(Personal: #1)}}
\newcommand{\discussion}[1]{\textcolor{violet}{(Discussion: #1)}}
\newcommand{\todo}[1]{\textcolor{red}{(Todo: #1)}}
\else
\newcommand{\personal}[1]{\ignorespaces}
\newcommand{\discussion}[1]{\ignorespaces}
\newcommand{\todo}[1]{\ignorespaces}
\fi

\newcommand{\R}{\mathbb R}

\newcommand{\rB}{\mathrm B}

\newcommand{\rR}{\mathrm R}

\newcommand{\fT}{\mathfrak T}

\newcommand{\fV}{\mathfrak V}
\newcommand{\fX}{\mathfrak X}
\newcommand{\fY}{\mathfrak Y}

\newcommand{\ff}{\mathfrak f}

\newcommand{\cD}{\mathcal D}

\newcommand{\cF}{\mathcal F}

\newcommand{\cG}{\mathcal G}
\newcommand{\cI}{\mathcal I}

\newcommand{\cN}{\mathcal N}
\newcommand{\cO}{\mathcal O}

\newcommand{\cS}{\mathcal S}
\newcommand{\cT}{\mathcal T}

\newcommand{\cX}{\mathcal X}
\newcommand{\cY}{\mathcal Y}

\DeclareFontFamily{U}{BOONDOX-calo}{\skewchar\font=45 }
\DeclareFontShape{U}{BOONDOX-calo}{m}{n}{<-> s*[1.05] BOONDOX-r-calo}{}
\DeclareFontShape{U}{BOONDOX-calo}{b}{n}{<-> s*[1.05] BOONDOX-b-calo}{}
\DeclareMathAlphabet{\mathcalboondox}{U}{BOONDOX-calo}{m}{n}

\newcommand{\bbA}{\mathbb A}

\newcommand{\bbG}{\mathbb G}
\newcommand{\bbL}{\mathbb L}
\newcommand{\bbN}{\mathbb N}
\newcommand{\bbP}{\mathbb P}
\newcommand{\bbT}{\mathbb T}

\newcommand{\bA}{\mathbf A}
\newcommand{\bD}{\mathbf D}
\newcommand{\bP}{\mathbf P}

\newcommand{\bW}{\mathbf W}

\newcommand{\bZ}{\mathbf Z}

\makeatletter
\let\save@mathaccent\mathaccent
\newcommand*\if@single[3]{%
	\setbox0\hbox{${\mathaccent"0362{#1}}^H$}%
	\setbox2\hbox{${\mathaccent"0362{\kern0pt#1}}^H$}%
	\ifdim\ht0=\ht2 #3\else #2\fi
}
\newcommand*\rel@kern[1]{\kern#1\dimexpr\macc@kerna}
\newcommand*\widebar[1]{\@ifnextchar^{{\wide@bar{#1}{0}}}{\wide@bar{#1}{1}}}
\newcommand*\wide@bar[2]{\if@single{#1}{\wide@bar@{#1}{#2}{1}}{\wide@bar@{#1}{#2}{2}}}
\newcommand*\wide@bar@[3]{%
	\begingroup
	\def\mathaccent##1##2{%
		\let\mathaccent\save@mathaccent
		\if#32 \let\macc@nucleus\first@char \fi
		\setbox\z@\hbox{$\macc@style{\macc@nucleus}_{}$}%
		\setbox\tw@\hbox{$\macc@style{\macc@nucleus}{}_{}$}%
		\dimen@\wd\tw@
		\advance\dimen@-\wd\z@
		\divide\dimen@ 3
		\@tempdima\wd\tw@
		\advance\@tempdima-\scriptspace
		\divide\@tempdima 10
		\advance\dimen@-\@tempdima
		\ifdim\dimen@>\z@ \dimen@0pt\fi
		\rel@kern{0.6}\kern-\dimen@
		\if#31
		\overline{\rel@kern{-0.6}\kern\dimen@\macc@nucleus\rel@kern{0.4}\kern\dimen@}%
		\advance\dimen@0.4\dimexpr\macc@kerna
		\let\final@kern#2%
		\ifdim\dimen@<\z@ \let\final@kern1\fi
		\if\final@kern1 \kern-\dimen@\fi
		\else
		\overline{\rel@kern{-0.6}\kern\dimen@#1}%
		\fi
	}%
	\macc@depth\@ne
	\let\math@bgroup\@empty \let\math@egroup\macc@set@skewchar
	\mathsurround\z@ \frozen@everymath{\mathgroup\macc@group\relax}%
	\macc@set@skewchar\relax
	\let\mathaccentV\macc@nested@a
	\if#31
	\macc@nested@a\relax111{#1}%
	\else
	\def\gobble@till@marker##1\endmarker{}%
	\futurelet\first@char\gobble@till@marker#1\endmarker
	\ifcat\noexpand\first@char A\else
	\def\first@char{}%
	\fi
	\macc@nested@a\relax111{\first@char}%
	\fi
	\endgroup
}
\makeatother

\newcommand{\oC}{\widebar C}

\newcommand{\oM}{\widebar M}

\newcommand{\hL}{\widehat L}

\newcommand{\tf}{\widetilde f}

\newcommand{\tp}{\widetilde p}

\newcommand{\ts}{\widetilde s}

\newcommand{\tC}{\widetilde C}

\newcommand{\tU}{\widetilde U}

\newcommand{\tW}{\widetilde W}

\newcommand{\tbf}{\widetilde{\mathbf f}}
\newcommand{\tbC}{\widetilde{\mathbf C}}

\newcommand{\IX}{I_\fX}

\newcommand{\SX}{S_\fX}

\newcommand{\fXs}{\fX_s}

\newcommand{\Sh}{\mathrm{Sh}}

\newcommand{\infcat}{$\infty$-category\xspace}
\newcommand{\infcats}{$\infty$-categories\xspace}

\newcommand{\inftopos}{$\infty$-topos\xspace}

\newcommand{\tauet}{\tau_\mathrm{\acute{e}t}}

\newcommand{\Mod}{\textrm{-}\mathrm{Mod}}

\newcommand{\Coh}{\mathrm{Coh}}
\newcommand{\Cohb}{\mathrm{Coh}^\mathrm{b}}
\newcommand{\Cohh}{\mathrm{Coh}^\heartsuit}

\newcommand{\St}{\mathrm{St}}

\newcommand{\An}{\mathrm{An}}
\newcommand{\Afd}{\mathrm{Afd}}
\newcommand{\Top}{\mathcal T\mathrm{op}}

\newcommand{\dAn}{\mathrm{dAn}}

\newcommand{\dAnk}{\mathrm{dAn}_k}

\newcommand{\cTank}{\cT_{\mathrm{an}}(k)}

\newcommand{\cTetk}{\cT_{\mathrm{\acute{e}t}}(k)}

\newcommand{\RTop}{\tensor*[^\rR]{\Top}{}}

\newcommand{\dAfd}{\mathrm{dAfd}}
\newcommand{\dAfdk}{\mathrm{dAfd}_k}

\newcommand{\trunc}{\mathrm{t}_0}

\newcommand{\CAlg}{\mathrm{CAlg}}

\newcommand{\Cat}{\mathrm{Cat}}

\newcommand{\fib}{\mathrm{fib}}
\newcommand{\DerAn}{\mathrm{Der}\an}
\newcommand{\anL}{\mathbb L\an}

\newcommand{\dAff}{\mathrm{dAff}}

\newcommand{\bfMap}{\mathbf{Map}}

\newcommand{\bfHom}{\mathbf{Hom}}
\newcommand{\dAnSt}{\mathrm{dAnSt}}

\newcommand{\Perf}{\mathrm{Perf}}

\newcommand{\dom}{\mathrm{dom}}
\newcommand{\vdim}{\mathrm{vdim}}

\newcommand{\fM}{\mathfrak M}
\newcommand{\tsigma}{\widetilde{\sigma}}

\newcommand{\st}{\mathrm{st}}

\newcommand{\ex}{\mathrm{ex}}
\newcommand{\stab}{\mathrm{stab}}

\newcommand{\Lin}{\mathrm{Lin}}
\newcommand{\Conv}{\mathrm{Conv}}
\newcommand{\sConv}{\mathrm{sConv}}
\newcommand{\dSt}{\mathrm{dSt}}
\newcommand{\dStlaft}{\mathrm{dSt}^{\mathrm{laft}}}
\newcommand{\vir}{\mathrm{vir}}
\newcommand{\MG}{\mathrm{MG}}
\newcommand{\afd}{\mathrm{afd}}

\newcommand{\kc}{k^\circ}

\newcommand{\an}{^\mathrm{an}}
\newcommand{\alg}{^\mathrm{alg}}

\newcommand{\et}{_\mathrm{\acute{e}t}}
\newcommand{\ev}{\mathrm{ev}}
\newcommand{\mult}{\mathrm{mult}}
\newcommand{\inv}{^{-1}}
\newcommand{\id}{\mathrm{id}}
\newcommand{\gn}{$n$-pointed genus $g$ }

\newcommand{\kanal}{$k$-analytic\xspace}

\newcommand{\red}{\mathrm{red}}

\newcommand{\oMpre}{\oM^\mathrm{pre}}
\newcommand{\oCpre}{\oC^\mathrm{pre}}

\newcommand{\op}{^\mathrm{op}}
\newcommand{\Cech}{\check{\mathcal C}}
\newcommand{\DM}{Deligne-Mumford\xspace}
\providecommand{\abs}[1]{\lvert#1\rvert}

\newcommand*{\longhookrightarrow}{\ensuremath{\lhook\joinrel\relbar\joinrel\rightarrow}}

\usetikzlibrary{decorations.markings} 
\tikzset{
  closed/.style = {decoration = {markings, mark = at position 0.5 with { \node[transform shape, xscale = .8, yscale=.4] {/}; } }, postaction = {decorate} },
  open/.style = {decoration = {markings, mark = at position 0.5 with { \node[transform shape, scale = .7] {$\circ$}; } }, postaction = {decorate} }
}

\DeclareMathOperator{\cofib}{cofib}

\DeclareMathOperator{\Fun}{Fun}

\DeclareMathOperator{\Map}{Map}

\DeclareMathOperator{\rank}{rank}

\DeclareMathOperator{\Sp}{Sp}

\DeclareMathOperator{\Spec}{Spec}

\DeclareMathOperator{\Sym}{Sym}
\DeclareMathOperator{\val}{val}

\DeclareMathOperator*{\colim}{colim}

\DeclareMathOperator*{\cotimes}{\widehat{\otimes}}

\begin{document}
\title{Non-archimedean quantum K-invariants}

\author{Mauro PORTA}
\address{Mauro PORTA, Institut de Recherche Mathématique Avancée, 7 Rue René Descartes, 67000 Strasbourg, France}
\email{porta@math.unistra.fr}

\author{Tony Yue YU}
\address{Tony Yue YU, Department of Mathematics M/C 253-37, California Institute of Technology, 1200 E California Blvd, Pasadena, CA 91125, USA}
\email{yuyuetony@gmail.com}
\date{July 19, 2022}
\subjclass[2020]{Primary 14N35; Secondary 14A30, 14G22, 14D23}
\keywords{Quantum K-invariant, Gromov-Witten, derived geometry, non-archimedean geometry, rigid analytic geometry}

\begin{abstract}
	We construct quantum K-invariants in non-archimedean analytic geometry.
	Contrary to the classical approach in algebraic geometry via perfect obstruction theory, we build on our previous works on the foundations of derived non-archimedean geometry, the representability theorem and Gromov compactness.
	We obtain a list of natural geometric relations between the stacks of stable maps, directly at the derived level, with respect to elementary operations on graphs, namely, products, cutting edges, forgetting tails and contracting edges.
	They imply immediately the corresponding properties of quantum K-invariants.
	The derived approach produces highly intuitive statements and functorial proofs.
	The flexibility of our derived approach to quantum K-invariants allows us to impose not only simple incidence conditions for marked points, but also incidence conditions with multiplicities.
	This leads to a new set of enumerative invariants.
	For the proofs, we further develop the foundations of derived non-archimedean geometry in this paper: we study derived lci morphisms, relative analytification, and deformation to the normal bundle.
	Our motivations come from non-archimedean enumerative geometry and mirror symmetry.

\end{abstract}

\maketitle

\tableofcontents

\section{Introduction}

Let $X$ be a smooth projective complex variety.
Its system of quantum $K$-invariants is the collection of linear maps
\[K_{g,n,\beta}^X\colon K_0(X)^{\otimes n}\longrightarrow K_0(\oM_{g,n}),\]
for all $g,n\in\bbN$ and $\beta\in H_2(X)$, defined by
\[K_{g,n,\beta}^X(a_1\otimes\dots\otimes a_n)\coloneqq\st_*\big(\cO_{\oM_{g,n}(X,\beta)}^\vir\otimes\ev_1^*a_1\otimes\dots\otimes\ev_n^*a_n\big),\]
where $\oM_{g,n}(X,\beta)$ is the moduli stack of $n$-pointed genus $g$ stable maps into $X$ of class $\beta$, $\cO^\vir$ is its virtual structure sheaf, $\st$ is the stabilization of domain, and $\ev_i$ are the evaluation maps (see \cite{Lee_Quantum_K-theory_I}).
Coupling with classes in $K_0(\oM_{g,n})$, we obtain numerical invariants that are non-negative integers, which manifests (in a subtle way) the geometry of curves in $X$.

Quantum K-invariants were first studied by Givental \cite{Givental_On_the_WDVV} in the case $g=0$ and $X$ is convex.
In this case the moduli stack $\oM_{g,n}(X,\beta)$ is smooth with expected dimension, the virtual structure sheaf is thus equal to the structure sheaf.
The virtual structure sheaves in the general case are constructed by Lee \cite{Lee_Quantum_K-theory_I}, using the techniques of perfect obstruction theory and intrinsic normal cone developed by Behrend-Fantechi \cite{Behrend_Intrinsic_normal_cone}.
We refer to \cite{Lee_A_reconstruction_theorem_in_quantum_cohomology,Givental_The_Hirzebruch-Riemann-Roch,Iritani_Reconstruction_and_convergence,Givental_Permutation-equivariant_quantum_K-theory_I,Ruan_The_level_structure} for further developments of quantum K-theory in algebraic geometry, and to \cite{Givental_Quantum_K-theory_on_flag_manifolds,Buch_Quantum_K-theory_of_Grassmannians,Okounkov_Lectures_on_K-theoretic_computations} for more specific examples and applications.

Motivated by the non-archimedean approach to mirror symmetry (see \cite{Kontsevich_Homological_mirror_symmetry,Kontsevich_Affine_structures,Nicaise_Xu_Yu_The_non-archimedean_SYZ_fibration}), in this paper we construct quantum K-invariants in non-archimedean analytic geometry, more precisely, for proper smooth rigid analytic spaces over any discrete valuation field $k$ of residue characteristic zero.
Non-archimedean geometry comes into play naturally as we study degenerations of Calabi-Yau varieties in mirror symmetry.
It has special features in comparison with classical approaches by symplectic geometry and complex geometry.
In particular, we now have a rigorous foundation of non-archimedean SYZ (Strominger-Yau-Zaslow) fibration (see \cite{Nicaise_Xu_Yu_The_non-archimedean_SYZ_fibration}), whose original archimedean version is yet far beyond reach.
More importantly, the enumerative invariants responsible for instanton corrections in the SYZ conjecture \cite{Strominger_Mirror_symmetry_is_T-duality} cannot be well-defined in the archimedean setting; while the non-archimedean enumerative invariants are more promising candidates, as already demonstrated in special cases of log Calabi-Yau varieties in \cite{Yu_Enumeration_of_holomorphic_cylinders_I,Yu_Enumeration_of_holomorphic_cylinders_II,Keel_Yu_The_Frobenius,Hacking_Keel_Yu_Secondary_fan}.
Via formal models of non-archimedean analytic spaces, non-archimedean enumerative invariants are intimately related to logarithmic enumerative invariants of the special fiber (see the works of Abramovich, Chen, Gross, Siebert \cite{Chen_Stable_logarithmic_maps_I,Abramovich_Stable_logarithmic_maps_II,Gross_Logarithmic_Gromov-Witten_invariants,Abramovich_Punctured_logarithmic_maps}).

Our approach to the construction of non-archimedean quantum K-invariants differs from the classical one in algebraic geometry via perfect obstruction theory.
Instead, we build upon our previous works on the foundations of derived non-archimedean geometry \cite{Porta_Yu_Higher_analytic_stacks,Porta_Yu_Derived_non-archimedean_analytic_spaces,Porta_Yu_Representability_theorem,Porta_Yu_Derived_Hom_spaces}.
The results of this paper are supposed to bring powerful techniques from derived geometry into the study of mirror symmetry quantum corrections via non-archimedean enumerative methods along the lines of \cite{Yu_Enumeration_of_holomorphic_cylinders_I,Yu_Enumeration_of_holomorphic_cylinders_II,Keel_Yu_The_Frobenius,Hacking_Keel_Yu_Secondary_fan}.

\medskip

Now let us sketch the results of this paper.
Recall that in order to construct any satisfactory theory of enumerative invariants, there are two main issues to be solved: compactness and transversality (see the excellent expository paper \cite{Pandharipande_13/2_ways} on the two issues for various enumerative theories).
The same applies to non-archimedean quantum K-invariants.

In this case, the issue of compactness is solved in \cite{Yu_Gromov_compactness} via non-archimedean Kähler structures and formal models of non-archimedean stable maps, referred to as the non-archimedean Gromov compactness theorem, which we will review in \cref{sec:Gromov_compactness}.
We treat the issue of transversality in the current paper using derived geometry as explained below.

Fix a smooth rigid $k$-analytic space $X$.

\begin{theorem}[see \cref{thm:stack_of_stable_maps,thm:stack_of_stable_maps_derived_lci}] \label{thmintro:derived_enhancement}
	The moduli stack $\oM_{g,n}(X)$ of $n$-pointed genus $g$ stable maps into $X$ admits a natural derived enhancement $\R\oM_{g,n}(X)$ that is a derived \kanal stack locally of finite presentation and derived lci.
\end{theorem}

Now assuming $X$ proper and endowed with a Kähler structure, combining with the non-archimedean Gromov compactness theorem (\cref{thm:Gromov_compactness}), we can immediately obtain the non-archimedean quantum K-invariants
\begin{align*}
K_{g,n,\beta}^X\colon K_0(X)^{\otimes n} &\longrightarrow K_0(\oM_{g,n})\\
a_1\otimes\dots\otimes a_n &\longmapsto \st_*\big(\ev_1^*a_1\otimes\dots\otimes\ev_n^*a_n\big),
\end{align*}
where the maps $\st$ and $\ev_i$ all depart from the derived moduli stack $\R\oM_{g,n}(X,\beta)$.

\medskip

Next we prove that the non-archimedean quantum K-invariants satisfy all the expected properties analogous to the algebraic case.
Instead of manipulating perfect obstruction theories as in \cite{Lee_Quantum_K-theory_I,Behrend_Gromov-Witten_invariants}, our proofs of the properties of the invariants will follow directly from a list of natural and intuitive geometric relations between the derived moduli stacks, which we explain below.

In order to state the geometric relations, instead of working with $n$-pointed genus $g$ stable maps, one has to work with a slight combinatorial refinement called $(\tau,\beta)$-marked stable maps for an A-graph $(\tau,\beta)$, introduced by Behrend-Manin \cite{Behrend_Stacks_of_stable_maps} for the study of Gromov-Witten invariants.
The A-graph $(\tau,\beta)$ imposes degeneration types on the domains of stable maps as well as more refined curve classes (see \cref{sec:stable_maps_associated_to_A-graphs}).
We have the associated moduli stack of $(\tau,\beta)$-marked stable maps, the corresponding derived enhancement, and thus the (more refined) non-archimedean quantum K-invariants $K_{\tau,\beta}^X$.

\begin{theorem}[see \cref{thm:products,thm:cutting_edges_derived_pullback,thm:universal_tau_marked_curve,thm:forgetting_tails,thm:contracting_edges}] \label{thm:geometric_relations}
	Let $S$ be a rigid \kanal space and $X$ a rigid \kanal space smooth over $S$.
	The derived moduli stack $\R\oM(X/S,\tau,\beta)$ of $(\tau,\beta)$-marked stable maps into $X/S$ associated to an A-graph $(\tau,\beta)$ satisfies the following geometric relations with respect to elementary operations on A-graphs:
	
	\medskip\noindent
	(1) \textbf{Products:} Let $(\tau_1, \beta_1)$ and $(\tau_2, \beta_2)$ be two A-graphs.
	We have a canonical equivalence
	\[\R\oM(X/S, \tau_1 \sqcup \tau_2, \beta_1 \sqcup \beta_2) \xrightarrow{\ \sim\ } \R \oM(X/S, \tau_1, \beta_1) \times_S  \R \oM(X/S, \tau_2, \beta_2),\]
	where the projections are given by the natural forgetful maps.

	\medskip\noindent
	(2) \textbf{Cutting edges:} Let $(\sigma,\beta)$ be an A-graph obtained from $(\tau,\beta)$ by cutting an edge $e$ of $\tau$.
	Let $i,j$ be the two tails of $\sigma$ created by the cut.
	We have a derived pullback diagram
	\[ \begin{tikzcd}
	\R \oM(X/S, \tau, \beta) \arrow{r}{c} \arrow{d}{\ev_e} & \R \oM(X/S, \sigma, \beta) \arrow{d}{\ev_i \times \ev_j} \\
	X \arrow{r}{\Delta} & X \times_S X
	\end{tikzcd} \]
	where $\Delta$ is the diagonal map, $\ev_e$ is evaluation at the section $s_e$ corresponding to the edge $e$, and $c$ is induced by cutting the domain curves at $s_e$.
			
	\medskip\noindent
	(3) \textbf{Universal curve\footnote{We mention a slight discrepancy of terminology: our ``universal curve'' property corresponds to the ``forgetting tails'' property of \cite[\S 3.4]{Lee_Quantum_K-theory_I}; while our ``forgetting tails'' property corresponds to the ``fundamental classes'' property of loc.\ cit..
		As the geometry becomes more transparent in our derived approach, we have adapted our terminology accordingly.}:} Let $(\sigma,\beta)$ be an A-graph obtained from $(\tau,\beta)$ by forgetting a tail $t$ attached to a vertex $w$.
	Let $\oCpre_w\to\oMpre_\sigma$ be the universal curve corresponding to $w$.
	We have a derived pullback diagram
	\[ \begin{tikzcd}
	\R \oM( X/S, \tau, \beta ) \arrow{r} \arrow{d} & \R \oM( X/S, \sigma, \beta ) \arrow{d} \\
	\oCpre_w \arrow{r} & \oMpre_\sigma.
	\end{tikzcd} \]
	(See \cref{sec:universal_curve} for the construction of the maps.)
	
	\medskip\noindent
	(4) \textbf{Forgetting tails:}
	The context being as above, we have a derived pullback diagram
	\[ \begin{tikzcd}
	\R \oM(X/S, \tau, \beta) \arrow{r} \arrow{d} & \oM_\tau \times_{\oM_\sigma}  \R \oM(X/S, \sigma, \beta) \arrow{d} \\
	\oCpre_w \arrow{r} & \oM_\tau \times_{\oM_\sigma} \oMpre_\sigma.
	\end{tikzcd} \]
	
	\medskip\noindent
	(5) \textbf{Contracting edges:}
	Let $(\sigma,\beta)$ be an A-graph where $\sigma$ is obtained from a modular graph $\tau$ by contracting an edge (possibly a loop) $e$.
	Let $(\tsigma,\beta)$ be obtained from $(\sigma,\beta)$ by adding $k$ tails.
	Let $(\tau^i_l,\beta^{i_j}_l)$ denote the A-graphs obtained from $(\tau,\beta)$ by replacing $e$ with a chain of $l$ edges then attaching tails and curve classes in a compatible way (see \cref{sec:contracting_edges} for details).
	We have a natural equivalence
	\[\colim_l \coprod_{i,j} \R\oM(X/S,\tau^i_l,\beta^{i_j}_l)\xrightarrow{\ \sim\ }\oM_\tau\times_{\oM_\sigma}\R\oM(X/S,\tsigma,\beta).\]
\end{theorem}

We would like to point out that in the particular case where $\tau$ is a point, the statement (3) above asserts that the forgetful map
\[ \R \oM_{g,n+1}(X/S) \longrightarrow \R \oM_{g,n}(X/S) \]
is equivalent to the universal curve $\R\oC_{g,n}(X/S)\to\R\oM_{g,n}(X/S)$.
Note that such an intuitive statement in fact incorporates all the information about virtual counts with respect to forgetting a tail, which is classically expressed and proved in terms of pullback properties of perfect obstruction theories and intrinsic normal cones (see \cite[Proposition 8]{Lee_Quantum_K-theory_I}, \cite[Axiom IV]{Behrend_Gromov-Witten_invariants}).

The construction of the forgetful map above is in fact nontrivial.
Intuitively, we simply forget the last marked point then stabilize.
We construct the stabilization étale locally, then we need to exhibit gluing data which are typically cumbersome to build in the derived setting.
We prove a universal property (see \cref{prop:checking_stabilization_on_truncation}), which solves the issue of gluing, as well as the independence of all the choices involved in the construction.

Next we state the properties of non-archimedean quantum K-invariants:

\begin{theorem}[see \cref{prop:mapping_to_a_point,prop:products,prop:cutting_edges,prop:forgetting_tails,prop:contracting_edges}] \label{thmintro:properties_of_quantum_K-invariants}
	Let $X$ be a proper smooth rigid \kanal space equipped with a Kähler structure.
	Let $(\tau,\beta)$ be an A-graph and $T_\tau$ the set of tail vertices.
	The associated quantum K-invariants
	\[K^X_{\tau,\beta}\colon K_0(X)^{\otimes \abs{T_\tau}} \longrightarrow K_0(\oM_\tau)\]
	 satisfy the following properties:
	
	\medskip\noindent
	(1) \textbf{Mapping to a point:}
	Let $(\tau,0)$ be any A-graph where $\beta$ is 0.
	Let $p_1, p_2$ be the projections
	\[\begin{tikzcd}
	X \times \oM_\tau \rar{p_1} \dar{p_2} & X\\
	\oM_\tau & \phantom{X}.
	\end{tikzcd}\]
	For any $a_i\in K_0(X), i\in T_\tau$, we have
	\[K^X_{\tau,0}\big({\textstyle\bigotimes_i} a_i \big) = p_{2*} \Big(p_1^*({\textstyle\bigotimes_i} a_i) \otimes \lambda_{-1}\big( (\bbT\an_X\boxtimes\rR^1\pi_*\cO_{\oC_\tau})^\vee\big)\Big),\]
	where $\pi\colon\oC_\tau\to\oM_\tau$ and $\lambda_{-1}(F)\coloneqq\sum_{i}(-1)^i\wedge^i F$.
	
	\medskip\noindent
	(2) \textbf{Products:}
	Let $(\tau_1, \beta_1)$ and $(\tau_2, \beta_2)$ be two A-graphs.
	For any $a_i\in K_0(X), i\in T_{\tau_1}$ and $b_j\in K_0(X), j\in T_{\tau_2}$, we have
	\[K^X_{\tau_1\sqcup\tau_2,\beta_1\sqcup\beta_2}\big(({\textstyle\bigotimes_i} a_i)\otimes({\textstyle\bigotimes_j} b_j)\big)=K^X_{\tau_1,\beta_1}\big({\textstyle\bigotimes_i} a_i\big)\boxtimes K^X_{\tau_2,\beta_2}\big({\textstyle\bigotimes_j} b_j\big).\]
	
	\medskip\noindent
	(3) \textbf{Cutting edges:}
	Let $(\sigma,\beta)$ be an A-graph obtained from $(\tau,\beta)$ by cutting an edge $e$ of $\tau$.
	Let $v,w$ be the two tails of $\sigma$ created by the cut.
	Consider the following commutative diagram
	\[ \begin{tikzcd}
	\oM_\tau \rar{d} & \oM_\sigma\\
	\R \oM(X, \tau, \beta) \uar{\st_\tau} \arrow{r}{c} \arrow{d}{\ev_e} & \R \oM(X, \sigma, \beta) \uar{\st_\sigma} \arrow{d}{\ev_v \times \ev_w} \\
	X \arrow{r}{\Delta} & X \times_S X.
	\end{tikzcd} \]
	For any $a_i\in K_0(X), i\in T_\tau$, we have
	\[d_* K^X_{\tau,\beta}\big({\textstyle\bigotimes_i} a_i\big)=K^X_{\sigma,\beta}\big(({\textstyle\bigotimes_i} a_i)\otimes\Delta_*\cO_X\big).\]
	
	\medskip\noindent
	(4) \textbf{Forgetting tails:}
	Let $(\sigma,\beta)$ be an A-graph obtained from $(\tau,\beta)$ by forgetting a tail.
	Let $\Phi\colon\oM_\tau\to\oM_\sigma$ be the forgetful map from $\tau$-marked stable curves to $\sigma$-marked stable curves.
	For any $a_i\in K_0(X), i\in T_\sigma$, we have
	\[K^X_{\tau,\beta}\big(({\textstyle\bigotimes_i} a_i)\otimes \cO_X\big)=\Phi^*K^X_{\sigma,\beta}\big({\textstyle\bigotimes_i} a_i\big).\]
	
	\medskip\noindent
	(5) \textbf{Contracting edges:}
	Let $(\sigma,\beta)$ be an A-graph where $\sigma$ is obtained from a modular graph $\tau$ by contracting an edge (possibly a loop) $e$.
	We follow the notations of \cref{sec:contracting_edges}.
	Let $\Phi\colon\oM_\tau\to\oM_\sigma$, $\Psi_l^i\colon\oM_{\tau_l^i}\to\oM_\tau$ and $\Omega\colon\oM_{\tsigma}\to\oM_\sigma$ be the induced maps on the moduli stacks of stable curves.
	For any $a_v\in K_0(X), v\in T_{\tsigma}$, we have
	\[\Phi^*\Omega_* K^X_{\tsigma,\beta}\big({\textstyle\bigotimes_v} a_v\big) = \sum_l (-1)^{l+1} \sum_{i,j} \Psi_{l*}^i K^X_{\tau^i_l,\beta^{i_j}_l}\big({\textstyle\bigotimes_v} a_v\big).\]
\end{theorem}

The properties (2-5) in \cref{thmintro:properties_of_quantum_K-invariants} follow readily from the corresponding properties in \cref{thm:geometric_relations}, with arguments given in \cref{sec:quantum_K-invariants}.
Nevertheless, the first property ``mapping to a point'' requires extra effort, whose proof is based on the construction of deformation to the normal bundle.
We develop relative derived analytification in \cref{sec:relative_analytification}, which allows us to borrow the algebraic construction of deformation to the normal bundle from Gaitsgory-Rozenblyum \cite{Gaitsgory_Study_II}.
Due to the lack of $\bbA^1$-invariance of G-theory in analytic geometry (see \cite{Kerz_Towards_a_non-archimedean_analytic_analog,Kerz_K-theory_of_non-archimedean_rings_I} for related discussions), we had to extend the construction in \cite{Gaitsgory_Study_II} over a projective line, and then apply the GAGA theorem for stacks \cite{Porta_Yu_Higher_analytic_stacks} to rely on the $\bbA^1$-invariance in algebraic geometry.

Another consequence of our relative derived analytification theory is a GAGA-type result for analytic vs algebraic formal moduli problems, which may deserve independent interest:

\begin{theorem}[see \cref{thm:FMP}]
	Let $X \in \dAfd_k$ be a derived $k$-affinoid space and let $A \coloneqq \Gamma(\cO_X\alg)$.
	There is a canonical equivalence
	\[ \mathrm{FMP}_{/\Spec(A)} \xrightarrow{\ \sim\ } \mathrm{FMP}\an_{/X} . \]
\end{theorem}

The list of Kontsevich-Manin axioms as in \cite[\S 4.3]{Lee_Quantum_K-theory_I} holds also for the non-archimedean quantum K-invariants, which are immediate consequences of \cref{thmintro:properties_of_quantum_K-invariants}.
We do not repeat those axioms here, as the properties above with respect to elementary operations on A-graphs are more fundamental properties of the invariants.

In the final section of this paper, we take advantage of the flexibility of our derived approach to introduce a generalized type of quantum K-invariants that allow not only simple incidence conditions for marked points, but also incidence conditions with multiplicities.
They satisfy a list of properties parallel to \cref{thmintro:properties_of_quantum_K-invariants}.
To the best of our knowledge, such invariants are not yet considered in the literature, even in algebraic geometry.
If the incidence condition at a marked point $p$ is given by a divisor $D\subset X$, then for curves without components mapping into $D$, the multiplicity condition imposes simply an order of tangency at $p$ with respect to $D$.
This is compatible with the ideas of relative Gromov-Witten invariants (\cite{Ionel_Relative_Gromov-Witten_invariants,Li_A_degeneration_formula}), logarithmic Gromov-Witten invariants (\cite{Chen_Stable_logarithmic_maps_I,Abramovich_Stable_logarithmic_maps_II,Gross_Logarithmic_Gromov-Witten_invariants}) as well as logarithmic quantum K-theory (\cite{Chou_The_log_product_formula}).
Nevertheless, the exact relations are intriguing, and the situation becomes more complicated for curves with components mapping into $D$.
Furthermore, our theory with multiplicities applies not only to divisors, but also to arbitrary lci subvarieties.

\bigskip
All of our constructions and statements in \cref{sec:stable_maps,sec:geometry_of_derived_stable_maps,sec:quantum_K-invariants,sec:multiplicities} carry over almost verbatim to algebraic geometry, which provide alternative constructions and proofs of the algebraic quantum K-theory by Lee \cite{Lee_Quantum_K-theory_I}.
We remark that the derived approach towards Gromov-Witten theory in algebraic geometry dates back to Kontsevich \cite{Kontsevich_Enumeration} and Ciocan-Fontanine-Kapranov \cite{Ciocan-Fontanine_Virtual_fundamental_classes} using dg-manifolds, a precursor of derived algebraic geometry.
The derived stack of stable maps are later studied by Toën \cite{Toen_Higher_and_derived_stacks,Toen_Derived_algebraic_geometry,Toen_Operations_on_derived_moduli_spaces_of_branes}, Schürg-Toën-Vezzosi \cite{Schurg_Dervied_algebraic_geometry_determinants} and Mann-Robalo \cite{Mann_Brane_actions,Mann_Gromov-Witten_theory}.
Nevertheless, the geometric relations as in \cref{thm:geometric_relations} have never really been proved, even in algebraic geometry.

We end the introduction with a few remarks comparing our derived approach with the classical approach in algebraic geometry via perfect obstruction theory:
\begin{enumerate}
	\item It is difficult to apply the classical approach in analytic geometry, because the classical approach relies on the global resolution of the perfect obstruction theory, which is an unnatural assumption in analytic geometry.
	\item The derived approach allows us to work relatively over any base space $S$.
	\item As we have seen in \cref{thm:geometric_relations}, the derived approach gives more intuitive and transparent proofs of the geometric properties of quantum K-invariants.
	\item The derived approach completely separates the issue of transversality from the issue of compactness mentioned in the beginning, while in the classical approach, the compactness is involved in the solution of transversality, due to the global resolution assumption.
\end{enumerate}

\bigskip
\paragraph{\textbf{Notations and terminology}}

We refer to \cite{Lurie_HTT,Lurie_Higher_algebra} for the background on $\infty$-categories, to \cite{Bosch_Non-Archimedean_analysis} for the classical theory of rigid analytic spaces.
We refer to \cite{Porta_Yu_Derived_non-archimedean_analytic_spaces} for the framework of derived non-archimedean geometry, to \cite{Porta_Yu_Representability_theorem,Porta_Yu_Derived_Hom_spaces} for the deformation theory and the representability theorem in derived non-archimedean geometry, and to \cite{Porta_Yu_Derived_Hom_spaces} for the construction of derived analytic mapping stacks.
We refer to \cite{Kontsevich_Gromov-Witten_classes,Behrend_Stacks_of_stable_maps,Behrend_Gromov-Witten_invariants,Lee_Quantum_K-theory_I} for the background on Gromov-Witten theory.

The symbol $\cS$ denotes the \infcat of spaces (see \cite[\S1.2.5]{Lurie_HTT}), and $\Cat_\infty$ denotes the \infcat of small \infcats (see \cite[Definition~3.0.0.1]{Lurie_HTT}).
We use homological indexing convention, i.e.\ the differential in chain complexes lowers the degree by 1.

Throughout the paper, we fix a complete non-archimedean field with nontrivial valuation $k$.
For $n\in\bbN$, we denote by $\bA^n_k$ the $k$-analytic $n$-dimensional affine space, by $\bD^n_k$ the $k$-analytic $n$-dimensional closed unit polydisk, and by $\bP^n_k$ the \kanal $n$-dimensional projective space.

We denote by $\An_k$ the category of rigid \kanal spaces, by $\dAnk$ the \infcat of derived \kanal spaces (\cite[Definition 2.5]{Porta_Yu_Derived_non-archimedean_analytic_spaces}), by $\Afd_k$ the category of rigid $k$-affinoid spaces, and by $\dAfd_k$ the \infcat of derived $k$-affinoid spaces (\cite[Definition 7.3]{Porta_Yu_Derived_non-archimedean_analytic_spaces}).
The notion of derived non-archimedean analytic space is based on the theory of pregeometry and structured topos introduced by Lurie \cite{DAG-V}.
We refer to \cite[\S 2]{Porta_Yu_Derived_Hom_spaces} for a quick summary of the definitions, and to \cite[\S 1]{Porta_Yu_Derived_non-archimedean_analytic_spaces} for a heuristic explanation of the ideas behind the abstract definitions.

A derived \kanal space $X$ consists of a pair $(\cX,\cO_X)$, where $\cX$ denotes the underlying \inftopos of étale sheaves of spaces on $X$, and $\cO_X$ denotes the structure sheaf of derived \kanal rings (see \cite[Definition 4.1]{Porta_Yu_Representability_theorem}).
The underlying sheaf of simplicial commutative rings is denoted by $\cO\alg_X$.
For $n\in\bbN$, $\mathrm t_{\leq n}(X)$ denotes the truncation of level $n$, introduced in \cite[\S3.3]{Porta_Yu_Derived_non-archimedean_analytic_spaces}.
It has the same underlying $\infty$-topos $\cX$ as $X$, and its structure sheaf $\cO_{\mathrm t_{\le n}(X)}$ is equal to the internal $n$-truncation $\tau_{\le n}^{\cX}(\cO_X)$ of $\cO_X$ (see \cite[Proposition 5.5.6.18]{Lurie_HTT}).
When $n = 0$, we write $\trunc(X)\coloneqq\mathrm t_{\leq 0}(X)$, which is a classical rigid $k$-analytic space.
	
Given $X \in \dAnk$, $\Coh^+(X)$ denotes the stable $\infty$-category of eventually connective complexes of $\cO_X\alg$-modules with coherent sheaves of homotopy groups, as in \cite[Definition 2.6]{Porta_Yu_Derived_Hom_spaces}.
For a morphism $f \colon X \to Y$, we denote by $\anL_f$ or $\anL_{X/Y}$ the analytic cotangent complex of $f$, introduced in \cite[\S 5.2]{Porta_Yu_Representability_theorem}.
Theorem~1.5 in loc.\ cit.\ gives a summary of main properties of analytic cotangent complexes.
	
We denote by $\St(\dAfd_k)$ the $\infty$-category of stacks over the étale site of derived $k$-affinoid spaces, which are technically $\cS$-valued hypercomplete sheaves.
Such a stack is a derived \kanal stack if it is geometric, in the sense that it has an atlas of derived $k$-affinoid spaces (see \cite[Definition 7.2]{Porta_Yu_Representability_theorem} and \cite[\S 2.3]{Porta_Yu_Higher_analytic_stacks}).
The representability theorem of \cite{Porta_Yu_Representability_theorem} asserts that an analytic moduli functor (in terms of a stack in $\St(\dAfd_k)$) is a derived \kanal stack if and only if it is compatible with Postnikov towers, has a global analytic cotangent complex, and its truncation is a classical rigid \kanal stack.

\bigskip
\paragraph{\textbf{Acknowledgments}}

We are very grateful to Federico Binda, Antoine Chambert-Loir, Benjamin Hennion, Felix Janda, Maxim Kontsevich, Gérard Laumon, Y.P.\ Lee, Jacob Lurie, Etienne Mann, Tony Pantev, Francesco Sala, Carlos Simpson, Georg Tamme, Bertrand To\"en and Gabriele Vezzosi for valuable discussions.
We would like to thank Marco Robalo in particular for many detailed discussions and for his enthusiasm for our work.
The authors would also like to thank each other for the joint effort.
This research was partially conducted during the period when T.Y.\ Yu served as a Clay Research Fellow.
We have also received supports from the National Science Foundation under Grant No.\ 1440140, while we were in residence at the Mathematical Sciences Research Institute in Berkeley, California, and from the Agence Nationale de la Recherce under Grant ANR-17-CE40-0014, while we were at the Université Paris-Sud in Orsay, France.

\section{Derived local complete intersection morphisms} \label{sec:derived_lci}

We fix a complete non-archimedean field $k$ with nontrivial valuation.
In this section we study properties of derived local complete intersection (derived lci for short) morphisms in derived analytic geometry.
We prove that derived lci morphisms have finite tor-amplitude.
We introduce the notion of virtual dimension, and prove that a derived lci \kanal space is underived if and only if its virtual dimension is equal to the dimension of its truncation.

\begin{definition} \label{def:classical_lci}
	A morphism $f\colon X\to Y$ of (underived) \kanal stacks is called:	\begin{enumerate}
		\item a \emph{regular embedding} if it is a closed immersion and locally on $Y$ the ideal defining $X$ can be generated by a regular sequence;
		\item \emph{local complete intersection} (\emph{lci} for short) if locally on $X$ it can be factored as a regular embedding followed by a smooth morphism.
	\end{enumerate}
\end{definition}

\begin{definition} \label{def:derived_lci}
	A morphism $f \colon X \to Y$ of derived \kanal stacks is called \emph{derived local complete intersection} (\emph{derived lci} for short) if its analytic cotangent complex $\anL_{X / Y}$ is perfect and in tor-amplitude $(-\infty,1]$.
\end{definition}

\begin{remark} \label{rem:classical_definition_perfect_complex}
	Let $X$ be a derived \kanal stack.
	Recall from \cite[Definition 7.5]{Porta_Yu_Derived_Hom_spaces} that a coherent sheaf $\cF \in \Coh^+(X)$ is called a perfect complex if there exists a derived $k$-affinoid atlas $\{U_i\}$ of $X$ such that $\Gamma(\cF|_{U_i})$ belongs to the smallest full stable subcategory of $\Gamma(\cO_{U_i}\alg)\Mod$ closed under retracts and containing $\Gamma(\cO_{U_i}\alg)$.
	By \cite[Lemma 7.6]{Porta_Yu_Derived_Hom_spaces}, $\cF \in \Coh^+(X)$ is perfect if and only if it has finite tor-amplitude.
		By \cite[Corollary 5.40]{Porta_Yu_Representability_theorem}, $\anL_{X/Y}$ belongs to $\Coh^+(X)$.
	Hence $\anL_{X/Y}$ is automatically perfect as soon as it has tor-amplitude $(-\infty,1]$.
	\end{remark}

\begin{lemma} \label{lem:lci_underived}
	A morphism $f \colon X \to Y$ of (underived) \kanal stacks is  lci as in \cref{def:classical_lci} if and only if it is derived lci as in \cref{def:derived_lci}.
\end{lemma}

\begin{proof}
	The question being local on both $X$ and $Y$, we may assume that $X$ and $Y$ are both affinoid.
	Write $X = \Sp(B)$ and $Y = \Sp(A)$, and factor $f$ as
	\[ \begin{tikzcd}
		X \arrow[hook]{r}{j} & Y \times \bD_k^n \arrow{r}{p} & Y ,
	\end{tikzcd} \]
	where $j$ is a closed immersion and $p$ is the canonical projection.
	Using \cite[Proposition 5.50]{Porta_Yu_Representability_theorem} and the transitivity fiber sequence of analytic cotangent complex, we see that $f$ is derived lci if and only if $j$ is derived lci.
	We are therefore left to check that $j$ is derived lci if and only if the ideal defining $X = \Sp(B)$ inside $Y \times \bD_k^n = \Sp( A \langle T_1, \ldots, T_n \rangle )$ can be generated by a regular sequence.
	Since $j$ is a closed immersion, \cite[Corollary 5.33]{Porta_Yu_Representability_theorem} implies that there is a canonical equivalence
	\[ \Gamma(X; \anL_{X / Y \times \bD_k^n}) \simeq \mathbb L_{B / A \langle T_1, \ldots, T_n \rangle} , \]
	where the right hand side denotes the algebraic cotangent complex.
	As both $A \langle T_1, \ldots, T_n \rangle$ and $B$ are noetherian, the conclusion now follows from Avramov's second vanishing theorem \cite[Theorem 1.2]{Avramov_Quillen_conjecture}.
\end{proof}

Here are some basic properties of derived lci morphisms:

\begin{lemma} \label{lem:derived_lci_basic_properties}
	\begin{enumerate}
		\item If $f \colon X \to Y$ and $g \colon Y \to Z$ are derived lci morphisms, then so is the composition $g \circ f \colon X \to Z$.
		\item Let
		\[ \begin{tikzcd}
			X' \arrow{d}{f'} \arrow{r} & X \arrow{d}{f} \\
			Y' \arrow{r} & Y
		\end{tikzcd} \]
		be a pullback square of derived \kanal stacks.
		If $f$ is derived lci, then so is $f'$.
		\item Let $f \colon X \to Y$ be a smooth morphism and $s \colon Y \to X$ a section.
		Then $s$ is derived lci.
	\end{enumerate}
\end{lemma}

\begin{proof}
	Statement (1) follows from the transitivity sequence of analytic cotangent complex \cite[Proposition 5.10]{Porta_Yu_Representability_theorem}.
	Statement (2) follows from \cite[Proposition 5.12]{Porta_Yu_Representability_theorem}.
	For statement (3), the transitivity sequence yields a fiber sequence
	\[ s^* \anL_f \longrightarrow \anL_{\mathrm{id}_Y} \longrightarrow \anL_s . \]
	As $\anL_{\mathrm{id}_Y} \simeq 0$, we obtain $\anL_s \simeq s^* \anL_f[1]$.
	The conclusion now follows from \cite[Proposition 5.50]{Porta_Yu_Representability_theorem}.
\end{proof}

\subsection{The local structure of derived lci morphisms}

Let $X$ be a derived \kanal stack, $\pi \colon E \to X$ a vector bundle, $s \colon X \to E$ a section, and $s_0\colon X\to E$ the zero section.
Let $Z_X(s)$ denote the derived pullback
\[ \begin{tikzcd}
Z_X(s) \arrow{r} \arrow{d}{p} & X \arrow{d}{s} \\
X \arrow{r}{s_0} & E.
\end{tikzcd} \]
Note that its truncation $\trunc(Z_X(s))$ coincides with the usual zero locus of $s$.

\begin{lemma} \label{lem:standard_derived_lci}
	The morphism $p \colon Z_X(s) \to X$ is derived lci and of finite tor-amplitude.
\end{lemma}

\begin{proof}
	\Cref{lem:derived_lci_basic_properties}(2) and (3) imply that $p \colon Z_X(s) \to X$ is derived lci.
	Now we prove that it has finite tor-amplitude.
	The question being local on $X$, we can assume $X=\Sp(A)$ to be affinoid.
	Up to shrinking $X$, we can assume that $E$ is free of rank $n$ over $X$.
	Then $E$ is the analytification relative to $A$ of a free rank $n$ vector bundle $E\alg$ on $\Spec A$ (see \cref{sec:relative_analytification});
	moreover, $s_0$ and $s$ are analytifications of algebraic sections $s_0\alg$ and $s\alg$ from $\Spec A$ to $E\alg$.
	Let $Z_X\alg(s\alg)$ be the derived intersection of $s_0\alg$ and $s\alg$.
	Then $Z_X(s)$ is the analytification relative to $A$ of $Z_X\alg(s\alg)$.
	By \cite[Proposition 4.17]{Porta_Yu_Representability_theorem}, we are left to check that $Z_X\alg(s\alg)$ has finite tor-amplitude.
	This follows from direct computation of $Z_X\alg(s\alg)$ using the Koszul resolution of $s_0\alg$.
\end{proof}

\begin{lemma} \label{lem:local_structure_derived_lci_closed_immersion}
	Let $f\colon X\to Y$ be a derived lci closed immersion of derived \kanal stacks.
	Then locally on $Y$, there exists a vector bundle $E$ on $Y$ and a section $s\colon Y\to E$ such that $X$ is equivalent to $Z_Y(s)$.
\end{lemma}

\begin{proof}
	This proof is inspired by \cite[Proposition 2.1.10]{Arinkin_Singular_support}.
	The question being local, we may assume that $Y$ is affinoid.
	Since $f$ is a closed immersion, then $X$ is also affinoid.
	Set
	\[ A \coloneqq \Gamma( Y ; \cO_Y\alg ) \quad , \quad B \coloneqq \Gamma( X ; \cO_X\alg ) . \]
	Observe that
	\[ B \simeq \Gamma(Y ; f_* \cO_X\alg) . \]
	As $f_* \cO_X\alg$ is a coherent sheaf on $Y$, the t-exact equivalence in \cite[Theorem 3.4]{Porta_Yu_Derived_Hom_spaces} implies that the canonical map $A \to B$ is a surjection on $\pi_0$.
	Therefore, \cite[Corollary 5.33]{Porta_Yu_Representability_theorem} provides us with a canonical equivalence
	\[ \Gamma(X ; \anL_{X / Y}) \simeq \mathbb L_{B / A} . \]
	\cite[Theorem 3.4]{Porta_Yu_Derived_Hom_spaces} implies that $\mathbb L_{B / A}$ is perfect and in tor-amplitude $[0,1]$.
	Let
	\[ E \coloneqq \Spec( \Sym_B(\mathbb L_{B / A}[-1]) ) , \]
	and $E\an$ the analytification of $E$ relative to $A$.
	Then locally on $A$ we can find elements $a_1, \ldots, a_n \in \ker( \pi_0(A) \to \pi_0(B) )$ whose differentials generate $\mathbb L_{B/A}[-1] \otimes_B \pi_0(B)$.
	Up to shrinking $Y$, we can lift these elements to a function
	\[ g = (g_1, \ldots, g_n) \colon Y \to \bA^n_k . \]
	Let $Z$ be the derived fiber product
	\[ \begin{tikzcd}
		Z \arrow{r} \arrow{d} & Y \arrow{d}{g} \\
		\Sp(k) \arrow{r}{0} & \bA^n_k .
	\end{tikzcd} \]
	Then $f$ induces a canonical map $X \to Z$, which is a closed immersion and whose relative cotangent complex is zero, in other words $X\xrightarrow{\sim}Z$, completing the proof.
\end{proof}

\begin{proposition} \label{prop:local_structure_derived_lci}
	Let $f \colon X \to Y$ be a morphism of derived \kanal stacks.
	The following are equivalent:
	\begin{enumerate}
		\item the morphism $f$ is derived lci;
		\item locally on $X$ and $Y$ we can factor $f$ as
		\[ \begin{tikzcd}
		X \arrow{r}{j} & W \arrow{r}{p} & Y ,
		\end{tikzcd} \]
		where $j$ is a closed immersion, $p$ is smooth, and there exists a vector bundle $E$ on $W$ and a section $s$ such that $X \simeq Z_W(s)$.
	\end{enumerate}
\end{proposition}
\begin{proof}
	The implication (2) $\Rightarrow$ (1) is the content of \cref{lem:standard_derived_lci}.
	For the converse, observe that locally on $X$ and $Y$ we can factor $f$ as
	\[ \begin{tikzcd}
	X \arrow{r}{j} & Y \times \bD^n_k \arrow{r}{p} & Y ,
	\end{tikzcd} \]
	where $j$ is a closed immersion and $p$ is the canonical projection (see \cite[Lemma 5.48]{Porta_Yu_Representability_theorem}).
		As $p$ is smooth and $f$ is derived lci, the transitivity sequence for the analytic cotangent complex implies that $j$ is derived lci.
	Thus we conclude from \cref{lem:local_structure_derived_lci_closed_immersion}.
\end{proof}

\begin{corollary} \label{cor:derived_lci_finite_tor-amplitude}
	Let $f\colon X\to Y$ be a derived lci morphism of derived \kanal stacks.
	Then $f$ has finite tor-amplitude.
	In particular, if $\cO_Y$ is locally cohomologically bounded, then so is $\cO_X$.
\end{corollary}
\begin{proof}
	This follows from \cref{prop:local_structure_derived_lci} and \cref{lem:standard_derived_lci}.
\end{proof}

\subsection{Virtual dimension}

\begin{definition}
	Let $X = (\cX, \cO_X)$ be a derived lci \kanal space, $x \in \trunc(X)$ a rigid point, and $\iota_x \colon \Sp(\kappa(x)) \to \trunc(X) \hookrightarrow X$ the associated morphism.
	The \emph{virtual dimension of $X$ at $x$}, written $\vdim_x(X)$, is the Euler characteristic of $\iota_x^*( \anL_{X} )$.
\end{definition}

\begin{lemma}
	Let $X$ be a derived \kanal space such that $\trunc(X)$ is connected.
	Let $\cF \in \Perf(X)$ be a perfect complex on $X$.
	Then the Euler characteristic of $\iota_x^*( \cF )$ for any rigid point $x\in\trunc(X)$ is independent of $x$.
\end{lemma}

\begin{proof}
	We may assume that $\trunc(X) = \Sp(A)$ for some affinoid algebra $A$.
	Let $j \colon \trunc(X) \hookrightarrow X$ denote the canonical inclusion.
	Then $j^* \cF$ is a perfect complex on $\Sp(A)$.
	By \cite[Theorem 3.4]{Porta_Yu_Derived_Hom_spaces}, $j^*\cF$ is induced by a perfect complex $M$ in $A\Mod$.
	Since $A$ is discrete, \cite[Tag 0658]{Stacks_project} shows that we can represent $M$ by a bounded complex of projective $A$-modules:
	\[ \cdots \to 0 \to F^n \to F^{n-1} \to \cdots \to F^0 \to 0 \to \cdots .\]
	Then for any rigid point $x \in \Sp(A)$, the perfect complex $\iota_x^*\cF$ is quasi-isomorphic to the complex
	\[ \cdots \to 0 \to F^n \otimes_A \kappa(x) \to F^{n-1} \otimes_A \kappa(x) \to \cdots \to F^0 \otimes_A \kappa(x) \to 0 \to \cdots . \]
	Since each $F^i$ is a projective $A$-module and $\Sp(A)$ is connected, we have
	\[ \rank_A(F^i ) = \rank_{\kappa(x)}(F^i \otimes_A \kappa(x)) . \]
	It follows that
	\[ \chi( \iota_x^* \cF ) = \sum_{i = 0}^n (-1)^i \rank_{\kappa(x)}(F^i \otimes_A \kappa(x)) = \sum_{i = 0}^n (-1)^i \rank_A(F^i) , \]
	and therefore the left hand side is independent of the rigid point $x$.
\end{proof}

\begin{corollary}
	Let $X$ be a derived lci \kanal space such that $\trunc(X)$ is connected.
	Then the virtual dimension of $X$ at any rigid point $x \in \trunc(X)$ is independent of $x$.
\end{corollary}

\begin{lemma} \label{lem:expected_dimension}
	Let $X$ be a derived lci \kanal space.
	Assume that $\vdim_x(X) = \dim_x(\trunc(X))$ for any rigid point $x\in\trunc(X)$.
	Then $\trunc(X)$ is derived lci.
\end{lemma}

\begin{proof}
	The question being local on $X$, by \cite[Lemma 5.48]{Porta_Yu_Representability_theorem}, we can assume that there exists a closed immersion
	\[ f \colon X \hookrightarrow \bD^n_k . \]
	By \cite[Corollaries 5.26 and 5.37]{Porta_Yu_Representability_theorem}, $\anL_{\bD^n_k}$ is free of rank $n$; in particular, it is perfect and in tor-amplitude $0$.
	Therefore, the transitivity fiber sequence
	\[ f^* \anL_{\bD^n_k} \longrightarrow \anL_X \longrightarrow \anL_{X / \bD^n_k} \]
	shows that $\anL_{X / \bD^n_k}$ is perfect and in tor-amplitude $[0,1]$; in other words, $f$ is derived lci.
	Let $g\coloneqq\trunc(f)$, we have a similar transitivity fiber sequence
	\[ g^* \anL_{\bD^n_k} \longrightarrow \anL_{\trunc(X)}\longrightarrow \anL_{\trunc(X)/ \bD^n_k}.\]
	Therefore, in order to prove that $\trunc(X)$ is derived lci, it suffices to prove that $g$ is derived lci.
	
	Set $A \coloneqq \Gamma(X; \cO_X\alg)$ and $B \coloneqq \pi_0(A)$.
	We have $\trunc(X) \simeq \Sp(B)$.
	By \cite[Corollary 5.33]{Porta_Yu_Representability_theorem} and \cite[Theorem 3.4]{Porta_Yu_Derived_Hom_spaces}, it is enough to check that the algebraic cotangent complex $\mathbb L_{B / k \langle T_1, \ldots, T_n \rangle}$ is perfect and in tor-amplitude $[0,1]$.
	By \cite[Corollary 5.40]{Porta_Yu_Representability_theorem}, we only need to check that it is in tor-amplitude $[0,1]$.
	This condition can be checked after base change to the local rings at the maximal ideals of $B$.
	For this, it is enough to check that for every rigid point $x \in \Sp(B)$, the local ring $B_{(x)}$ is cut out by a regular sequence.
	Let
	\[ I \coloneqq \ker\left( k\langle T_1, \ldots, T_n \rangle_{(x)} \to B_{(x)} \right).\]
	Let $\mathfrak n$ be the maximal ideal of $k\langle T_1, \ldots, T_n \rangle_{(x)}$, and $\mathfrak m$ the maximal ideal of $B_{(x)}$.
	Then $I / \mathfrak m I$ is a $\kappa(x)$-vector space of finite dimension.
	Let $m$ denote the dimension.
	Choose elements $g_1, \ldots, g_m \in I$ mapping to a basis of $I / \mathfrak m I$.
	Notice that we can furthermore assume that $I \subset \mathfrak m^2$.
	Then the Nakayama lemma implies that
	\[ I = (g_1, \ldots, g_m) . \]
	We claim that this is a regular sequence.
	Let
	\[ M \coloneqq \mathbb L_{B_{(x)} / k\langle T_1, \ldots, T_n \rangle_{(x)}}.\]
	Then \cite[Corollary 5.33 and Lemma 5.51]{Porta_Yu_Representability_theorem} imply that
	\[ \pi_1\big(M \otimes_{B_{(x)}} \kappa(x)\big) \simeq I / \mathfrak m I \quad , \quad \pi_0\big(M \otimes_{B_{(x)}} \kappa(x)\big) \simeq \mathfrak n / \mathfrak n^2 . \]
		Using the long exact sequence associated to
	\[ j^* \anL_{X / \bD^n_k} \longrightarrow \anL_{\trunc(X) / \bD^n_k} \longrightarrow \anL_{\trunc(X) / X} \]
	we deduce that
	\begin{gather*}
	\dim_{\kappa(x)}\big( \pi_1( \iota_x^* \anL_{X / \bD^n_k} ) \big) \ge \dim_{\kappa(x)}( I / \mathfrak m I ),\\
	\dim_{\kappa(x)}\big( \pi_0( \iota_x^* \anL_{X / \bD^n_k }) \big) = \dim_{\kappa(x)}(\mathfrak n / \mathfrak n^2) .
	\end{gather*}
	Therefore
	\begin{multline*}
	\dim_x(\trunc(X)) = \vdim_x(X) = \dim_{\kappa(x)}\big( \pi_0( \iota_x^* \anL_{X / \bD^n_k }) \big) - \dim_{\kappa(x)}\big( \pi_1( \iota_x^* \anL_{X / \bD^n_k} ) \big)\\
	\le \dim_{\kappa(x)}(\mathfrak n / \mathfrak n^2) - \dim_{\kappa(x)}( I / \mathfrak m I ) = n - m .
	\end{multline*}
	But since $I = (g_1, \ldots, g_m)$, we see that $\dim_x(\trunc(X)) \ge n - m$.
	It follows that $\dim_x(\trunc(X)) = n - m$, hence $(g_1, \ldots, g_m)$ is a regular sequence.
\end{proof}

\begin{proposition}
	Let $X$ be a derived lci derived \kanal space, and $j \colon \trunc(X) \to X$ the canonical closed immersion.
	The following are equivalent:
	\begin{enumerate}
		\item \label{item:X_truncated} $j$ is an equivalence;
		\item \label{item:vanishing_relative_cotangent_complex} the relative analytic cotangent complex $\anL_{\trunc(X)/X}$ vanishes;
		\item \label{item:truncation_quasi_smooth_and_pi_1} $\trunc(X)$ is derived lci and the canonical map $\pi_1( j^* \anL_X ) \to \pi_1( \anL_{\trunc(X)} )$ is an isomorphism;
		\item \label{item:virtual_dimension_equals_dimension} $\vdim_x(X) = \dim_x(\trunc(X))$ for every rigid point $x \in \trunc(X)$.
	\end{enumerate}
\end{proposition}

\begin{proof}
	The equivalence (\ref{item:X_truncated}) $\Leftrightarrow$ (\ref{item:vanishing_relative_cotangent_complex}) follows directly from \cite[Corollary 5.35]{Porta_Yu_Representability_theorem}.
	
	If (\ref{item:vanishing_relative_cotangent_complex}) holds, then the fiber sequence
	\[ j^* \anL_X \to \anL_{\trunc(X)} \to \anL_{\trunc(X)/X} \]
	yields a canonical equivalence $j^* \anL_X \simeq \anL_{\trunc(X)}$.
	In particular, $\trunc(X)$ is derived lci and $\pi_1( j^* \anL_X) \to \pi_1(\anL_{\trunc(X)})$ is an isomorphism.
	So (\ref{item:vanishing_relative_cotangent_complex}) implies (\ref{item:truncation_quasi_smooth_and_pi_1}).
	
	For the converse implication, the above fiber sequence yields for any $i \ge 0$ a long exact sequence
	\[ \pi_{i+1}(j^* \anL_X) \to \pi_{i+1}(\anL_{\trunc(X)}) \to \pi_{i+1}( \anL_{\trunc(X)/X} ) \to \pi_i( j^* \anL_X ) \to \pi_i( \anL_{\trunc(X)} ) . \]
	If $i \ge 1$, the fact that both $X$ and $\trunc(X)$ are derived lci forces
	\[ \pi_{i+1}( \anL_{\trunc(X)/X} ) = 0 . \]
	When $i = 0$, we have $\pi_1(j^* \anL_X) \simeq \pi_1( \anL_{\trunc(X)} )$ by hypothesis; while \cite[Corollary 5.35]{Porta_Yu_Representability_theorem} guarantees that
	\[ \pi_0(j^* \anL_X) \simeq \pi_0(\anL_{\trunc(X)}) . \]
	Therefore we deduce
	\[ \pi_1(\anL_{\trunc(X)/X}) \simeq \pi_0(\anL_{\trunc(X)/X}) \simeq 0.\]
	Hence (\ref{item:truncation_quasi_smooth_and_pi_1}) implies (\ref{item:vanishing_relative_cotangent_complex}).
	
	It follows from \cref{lem:lci_underived} that (\ref{item:X_truncated}) implies (\ref{item:virtual_dimension_equals_dimension}).
	Now we will complete the proof by showing that (\ref{item:virtual_dimension_equals_dimension}) implies (\ref{item:truncation_quasi_smooth_and_pi_1}).
	First of all, $\trunc(X)$ is derived lci by \cref{lem:expected_dimension}.
	In order to show that the map $\pi_1( j^* \anL_X) \to \pi_1(\anL_{\trunc(X)})$ is an isomorphism, we can reason locally on $X$.
	We can thus assume that $X$ is derived affinoid.
	In virtue of \cite[Lemma 6.3]{Porta_Yu_Derived_non-archimedean_analytic_spaces}, we can furthermore suppose that $X$ admits a closed embedding $f \colon X \hookrightarrow \bD^n_k$ for some $n$.
	Using \cite[Theorem 3.4]{Porta_Yu_Derived_Hom_spaces} and \cite[Lemma 5.51]{Porta_Yu_Representability_theorem}, we can represent $j^* \anL_X$ and $\anL_{\trunc(X)}$ respectively as two perfect complexes on $\trunc(X)$ of the form
	\begin{gather*}
	j^* \anL_X \simeq (0 \to E^{-1} \to j^* f^* \Omega\an_{\bD^n_k} \to 0) , \\
	\anL_{\trunc(X)} \simeq (0 \to F^{-1} \to j^* f^* \Omega\an_{\bD^n_k} \to 0 ) .
	\end{gather*}
	Let
	\begin{gather*}
	K_E \coloneqq \ker\big( j^* f^* \Omega\an_{\bD^n_k} \to \pi_0( j^* \anL_X ) \big) , \\
	K_F \coloneqq \ker\big( j^* f^* \Omega\an_{\bD^n_k} \to \pi_0( \anL_{\trunc(X)} ) \big) .
	\end{gather*}
	We have a morphism of short exact sequences
	\[ \begin{tikzcd}
	0 \arrow{r} & K_E \arrow{d} \arrow{r} & j^* f^* \Omega\an_{\bD^n_k} \arrow{r} \arrow[equal]{d} & \pi_0( j^* \anL_X ) \arrow{r} \arrow{d} & 0 \\
	0 \arrow{r} & K_F \arrow{r} & j^* f^* \Omega\an_{\bD^n_k} \arrow{r} & \pi_0( \anL_{\trunc(X)} ) \arrow{r} & 0 .
	\end{tikzcd} \]
	It follows from \cite[Corollary 5.35]{Porta_Yu_Representability_theorem} that the map $\pi_0( j^* \anL_X ) \to \pi_0( \anL_{\trunc(X)} )$ is an isomorphism.
	We conclude that the canonical map $K_E \to K_F$ is also an isomorphism.
	Thus, we can consider the morphism of short exact sequences
	\[ \begin{tikzcd}
	0 \arrow{r} & \pi_1( j^*\anL_X ) \arrow{r} \arrow{d} & E^{-1} \arrow{r} \arrow{d} & K_E \arrow{r} \arrow{d} & 0 \\
	0 \arrow{r} & \pi_1( \anL_{\trunc(X)} ) \arrow{r} & F^{-1} \arrow{r} & K_F \arrow{r} & 0 .
	\end{tikzcd} \]
	Using once again \cite[Corollary 5.35]{Porta_Yu_Representability_theorem} we deduce that the map $\pi_1( j^* \anL_X) \to \pi_1( \anL_{\trunc(X)} )$ is surjective.
	Hence the snake lemma implies that $E^{-1} \to F^{-1}$ is also surjective.
	Finally, since $\vdim(X) = \dim(\trunc(X))$ and $\trunc(X)$ is derived lci, we conclude that $\rank(E^{-1}) = \rank(F^{-1})$.
	In particular, $E^{-1} \to F^{-1}$ is an isomorphism.
	It follows that $\pi_1( j^* \anL_X ) \to \pi_1( \anL_{\trunc(X)} )$ is also an isomorphism.
	In particular, (\ref{item:truncation_quasi_smooth_and_pi_1}) holds, completing the proof.
\end{proof}

\section{The stack of derived \texorpdfstring{$k$}{k}-analytic spaces} \label{sec:the_stack_of_derived_k-analytic_spaces}

In this section, we study the deformation theory of derived \kanal spaces.
We prove that the moduli functor of derived \kanal spaces is cartesian, convergent and has a global analytic cotangent complex.

Let $(\dAfd_k,\tauet)$ be the $\infty$-site of derived $k$-affinoid spaces equipped with the étale topology (see \cite[\S 7]{Porta_Yu_Derived_non-archimedean_analytic_spaces}).
Let
\[ \St( \dAfd_k ) \coloneqq \Sh( \dAfd_k, \tauet )^\wedge\]
denote the \infcat of stacks (i.e.\ hypercomplete sheaves) on this site.
Consider the functor
\[ \St \colon \dAfd_k \op \longrightarrow \Cat_\infty \]
given by
\[ X \mapsto \St( \dAfd_k )_{/X} . \]
The general theory of descent for $\infty$-topoi (see \cite[6.1.3.9]{Lurie_HTT}) implies that the functor $\St$ satisfies descent for the étale topology.
By \cref{cor:stability_flatness}, the functor $\St$ admits a subfunctor $F$ which sends $X \in \dAfd_k$ to the full subcategory $(\dAnk)_{/X}^{\mathrm{flat}}$ of the overcategory $(\dAnk)_{/X}$ spanned by flat morphisms $Y \to X$.
It follows from \cite[Propositions 8.5 and 8.7]{Porta_Yu_Derived_non-archimedean_analytic_spaces} that $F$ also satisfies étale descent.

\begin{lemma} \label{lem:gluing_along_closed_immersion_is_universal}
	Let $X \to S$, $X' \to S$ and $Y \to S$ be flat morphisms of derived \kanal spaces.
	Let $i \colon X \hookrightarrow X'$, $j \colon X \hookrightarrow Y$ be closed immersions, and $Y' \coloneqq Y \amalg_X X'$ the pushout (see \cite[Theorem 6.5]{Porta_Yu_Representability_theorem} for the construction).
	For any $T \in \dAfd_k$ and any map $T \to S$, define
	\[ X_T \coloneqq X \times_S T , \quad X'_T \coloneqq X' \times_S T , \quad Y_T \coloneqq Y \times_S T , \quad Y'_T \coloneqq Y' \times_S T . \]
	Then the canonical map
	\[ X'_T \amalg_{X_T} Y_T \longrightarrow Y'_T \]
	is an equivalence.
	\end{lemma}

\begin{proof}
	Notice that the squares
	\[ \begin{tikzcd}
	X_T \arrow{d}{i_T} \arrow{r} & X \arrow{d}{i} \\
	X'_T \arrow{r} & X'
	\end{tikzcd} \quad , \quad \begin{tikzcd}
	X_T \arrow{r} \arrow{d}{j_T} & X \arrow{d}{j} \\
	Y_T \arrow{r} & Y
	\end{tikzcd} \]
	are derived pullback diagrams.
	In particular the morphisms $i_T$ and $j_T$ are closed immersions.
	So by \cite[Theorem 6.5]{Porta_Yu_Representability_theorem}, the pushout $X'_T \amalg_{X_T} Y_T$ exists in the $\infty$-category $\dAnk$ of derived $k$-analytic spaces.
	Since $Y'_T\in\dAnk$, by the universal property of pushout we obtain a canonical morphism
	\[ f \colon Y_T \amalg_{X_T} X'_T \longrightarrow Y'_T . \]
	Now we show that it is an equivalence.
	By \cite[(6.7)]{Porta_Yu_Representability_theorem}, $Y'$ is flat over $S$.
	Using \cref{cor:stability_flatness}, we deduce that $X_T$, $X'_T$ and $Y_T$ are also flat over $T$, and therefore $X'_T \amalg_{X_T} Y_T$ is flat over $T$.
	Moreover, the same result implies that $Y'_T$ is flat over $T$.
	It follows from \cite[Lemma 5.47]{Porta_Yu_Representability_theorem} that the canonical map $f$ is strong.
	We are therefore left to check that $f$ is an equivalence on the truncation.

	Let $R$ be an affinoid algebra, $A$, $A'$ and $B$ affinoid $R$-algebras, $A' \to A$ and $B \to A$ surjective morphisms, and $B' \coloneqq A' \times_A B$.
	Note $B'$ is an affinoid $R$-algebra by \cite[Lemma 6.3]{Porta_Yu_Representability_theorem}.
	Let $R \to R'$ be a morphism of affinoid algebras.
	Then it suffices to check that
	\[ \begin{tikzcd}
	B' \cotimes_R R' \arrow{r} \arrow{d} & B \cotimes_R R' \arrow{d} \\
	A' \cotimes_R R' \arrow{r} & A \cotimes_R R'
	\end{tikzcd} \]
	is pullback in $\CAlg_k$, the \infcat of simplicial commutative $k$-algebras.
	By \cite{Bosch_Formal_and_rigid_geometry_II}, we can choose formal models for the affinoid algebras preserving the flatness.
	Now the conclusion follows directly from the analogous algebraic statement.
\end{proof}

\begin{remark}
	The lemma should hold without the flatness assumption.
	One may prove it by reducing to the algebraic case via derived formal models (see António \cite{Antonio_p-adic_derived_formal_geometry}).
\end{remark}

\begin{lemma} \label{lem:equivalence_derived_affinoid_global_sections}
	Let $f \colon X \to Y$ be a morphism of derived $k$-affinoid spaces.
	Let
	\[ A \coloneqq \Gamma(Y; \cO_Y\alg) , \quad B \coloneqq \Gamma(X; \cO_X\alg) . \]
	Then $f$ is an equivalence if and only if the induced map $A \to B$ is an equivalence in $\CAlg_k$.
\end{lemma}
\begin{proof}
	The ``only if'' direction is obvious.
	Now assume that $A \to B$ is an equivalence.
	Then $\pi_0(A) \to \pi_0(B)$ is an isomorphism, which implies that $f_0 \coloneqq \trunc(f) \colon \trunc(X) \to \trunc(Y)$ is an isomorphism.
	So it suffices to verify that $f$ is strong.
	For this, we only have to check that the canonical map
	\[ f_0^*( \pi_i( \cO_Y\alg ) ) \longrightarrow \pi_i( \cO_X\alg ) \]
	is an equivalence for every $i$.
	As both sides are coherent sheaves on $\trunc(X)$, the conclusion readily follows from the assumption that $A \to B$ is an equivalence and \cite[Theorem 3.4]{Porta_Yu_Derived_Hom_spaces}.
\end{proof}

\begin{proposition} \label{prop:stack_derived_analytic_spaces_cartesian}
	The functor $F \colon \dAfd_k\op \to \Cat_\infty$ is cartesian, in the sense that for any pushout diagram
	\[ \begin{tikzcd}
	X_{01} \arrow[hook]{r}{j_0} \arrow[hook]{d}{j_1} & X_0 \arrow{d} \\
	X_1 \arrow{r} & X
	\end{tikzcd} \]
	in $\dAfd_k$ where $j_0$ and $j_1$ are closed immersions, the canonical functor
	\[ \Psi \colon (\dAnk)_{/X}^{\mathrm{flat}} \longrightarrow (\dAnk)_{/X_1}^{\mathrm{flat}} \times_{ (\dAnk)_{/X_{01}}^{\mathrm{flat}} } ( \dAnk )_{/X_0}^{\mathrm{flat}} \]
	is an equivalence.
\end{proposition}

\begin{proof}
	We first construct a left adjoint
	\begin{equation} \label{eq:dGeom_cartesian}
		\Phi \colon (\dAnk)_{/X_1}^{\mathrm{flat}} \times_{ (\dAnk)_{/X_{01}}^{\mathrm{flat}} } ( \dAnk )_{/X_0}^{\mathrm{flat}} \longrightarrow (\dAnk)_{/X}^{\mathrm{flat}}
	\end{equation}
	of $\Psi$ as follows.
	Let
	\[ \mathrm{Span} \coloneqq \{ * \leftarrow * \rightarrow *\} \]
	denote the span category.
	Consider the composition of the natural functors
	\[ (\dAnk)_{/X_1} \times_{(\dAnk)_{X_{01}}} (\dAnk)_{/X_0} \xrightarrow{\phi} \Fun( \mathrm{Span}, \dAnk ) \to \Fun( \mathrm{Span}, \RTop(\cTank) ) . \]
	We can informally describe the functor $\phi$ on objects as follows.
	An object in the fiber product on the left is the given of a diagram
	\begin{equation} \label{eq:clutching_descent_data}
		\begin{tikzcd}
			Y_1 \arrow{d} & Y_{01} \arrow{l}[swap]{i_0} \arrow{d} \arrow{r}{i_1} & Y_1 \arrow{d} \\
			X_1 & X_{01} \arrow{l}[swap]{j_1} \arrow{r}{j_1} & X_0
		\end{tikzcd}
	\end{equation}
	where both squares are cartesian.
	Then $\phi$ sends this diagram to the top line $Y_1 \leftarrow Y_{01} \rightarrow Y_0$.
		Notice that $i_0$ and $i_1$ are again closed immersions.
	Therefore, it follows from \cite[Theorem 6.5]{Porta_Yu_Representability_theorem} that we can form the pushout $Y$ in $\RTop(\cTank)$ of
	\[ \begin{tikzcd}
	Y_1 & Y_{01} \arrow{l}[swap]{i_1} \arrow{r}{i_0} & Y_0 ;
	\end{tikzcd} \]
	moreover, the result is a derived \kanal space.
	The universal property of pushout provides a canonical map $Y \to X$, and hence a functor
	\[ (\dAnk)_{/X_1} \times_{(\dAnk)_{/X_{01}}} (\dAnk)_{/X_0} \longrightarrow (\dAnk)_{/X} , \]
	which is left adjoint to the canonical functor
	\[ (\dAnk)_{/X} \longrightarrow (\dAnk)_{/X_1} \times_{(\dAnk)_{/X_{01}}} (\dAnk)_{/X_0} . \]
	We claim that this functor restricts to the full subcategories spanned by flat morphisms.
	Indeed, the claim being local, we can assume that $Y_0$, $Y_1$ and $Y_{01}$ are all derived $k$-affinoid.
	In this case, the conclusion follows from \cite[Theorem 3.4]{Porta_Yu_Derived_Hom_spaces} and \cite[16.2.3.1(4)]{Lurie_SAG}.
	
	Next we claim that both the unit and the counit of this adjunction
	\[ u \colon \id \longrightarrow \Psi \circ \Phi, \quad v \colon \Phi \circ \Psi \longrightarrow \id \]
	are equivalences.
	For $v$, this is a direct consequence of Lemmas \ref{lem:gluing_along_closed_immersion_is_universal}.
	Now we show that $u$ is also an equivalence.
	Consider an object in the fiber product
	\[ (\dAnk)_{/X_1}^{\mathrm{flat}} \times_{ (\dAnk)_{/X_{01}}^{\mathrm{flat}} } ( \dAnk )_{/X_0}^{\mathrm{flat}} \]
	which we depict as the diagram \eqref{eq:clutching_descent_data}.
	Let $Y$ be the pushout of that diagram.
	Then we have to check that the canonical morphisms
	\[ Y_0 \longrightarrow X_0 \times_X Y, \qquad Y_1 \longrightarrow X_1 \times_X Y \]
	are equivalences.
	The question being local, we can assume that $Y$ (and therefore $Y_0$, $Y_1$ and $Y_{01}$) are derived $k$-affinoid spaces.
	Write
	\begin{gather*}
		A \coloneqq \Gamma(X, \cO_X\alg), \quad B \coloneqq \Gamma(Y, \cO_X\alg), \\
		A_i \coloneqq \Gamma(X_i, \cO_{X_i}\alg), \quad B_i \coloneqq \Gamma(Y_i, \cO_{Y_i}\alg) .
	\end{gather*}
	By \cref{lem:equivalence_derived_affinoid_global_sections}, it is enough to check that the canonical map $Y_i \to X_i \times_X Y$ induces an equivalence in $\CAlg_k$ after passing to global sections.
	By \cite[Proposition 4.2]{Porta_Yu_Derived_Hom_spaces}, we can compute the global sections of $X_i \times_X Y$ as the tensor product $A_i \otimes_A B$.
	We are therefore left to show that the canonical map
	\[ A_i \otimes_A B \longrightarrow B_i \]
	is an equivalence.
	Combining diagram (6.7) in the proof of \cite[Theorem 6.5]{Porta_Yu_Representability_theorem} with \cite[Theorem 3.4]{Porta_Yu_Derived_Hom_spaces}, we conclude that the diagram
	\[ \begin{tikzcd}
		B \arrow{r} \arrow{d} & B_1 \arrow{d} \\
		B_0 \arrow{r} & B_{01}
	\end{tikzcd} \]
	is a pullback diagram in $\CAlg_{A/}$, and hence in $A \Mod$.
	The conclusion now follows from \cite[Proposition 16.2.2.1]{Lurie_SAG}.
\end{proof}

\begin{proposition} \label{prop:stack_derived_analytic_spaces_convergent}
	The functor $F \colon \dAfd_k\op \to \Cat_\infty$ is convergent, in the sense that for any $X \in \dAfd_k$ the canonical map
	\[ F( X ) \longrightarrow \lim_{n \in \mathbb N} F( \mathrm t_{\le n} X ) \]
	is an equivalence, where $\mathrm t_{\le n} X$ denote the truncation of level $n$ (see \cite[\S 3.3]{Porta_Yu_Derived_non-archimedean_analytic_spaces}).
\end{proposition}
\begin{proof}
	Let $m$ be a non-negative integer.
	We say that a derived \kanal space $X = (\cX, \cO_X)$ is $m$-truncated if $\pi_i( \cO_X\alg ) \simeq 0$ for $i > m$.
	Let $\dAn_{/X}^{\le m}$ denote the full subcategory of $\dAn_X^{\le m}$ spanned by $m$-truncated derived \kanal spaces.
	It follows from \cite[Theorem 3.23]{Porta_Yu_Derived_non-archimedean_analytic_spaces} that the natural inclusion $\dAn_{/X}^{\le m} \hookrightarrow \dAn_{/X}^{\le m+1}$ admits a right adjoint $\mathrm t_{\le m}$.
	Observe that the natural morphism
	\begin{equation} \label{eq:nilcompleteness}
		\dAn_{/X} \longrightarrow \lim_{n \in \mathbb N\op} \dAn_{/X}^{\le m} ,
	\end{equation}
	where the transition maps are given by the truncation functors, is an equivalence.
		Therefore, the vertical maps in the following commutative square
	\[ \begin{tikzcd}
		\dAn_{/X} \arrow{r} \arrow{d} & \lim_{n \in \mathbb N\op} \dAn_{/\mathrm t_{\le n}X} \arrow{d} \\
		\lim_{m \in \mathbb N\op} \dAn_{/X}^{\le m} \arrow{r} & \lim_{n \in \mathbb N\op} \lim_{m \in \mathbb N\op} \dAn_{/\mathrm t_{\le n} X}^{\le m} .
	\end{tikzcd} \]
	are equivalences.
	Moreover, for every $n \ge m$ we have a canonical equivalence
	\[ \dAn_{/\mathrm t_{\le n} X}^{\le m} \simeq \dAn_{/ \mathrm t_{\le m} X}^{\le m} . \]
	It follows that the bottom horizontal morphism in the diagram above is also an equivalence, thus completing the proof.
\end{proof}

Consider now the maximal $\infty$-groupoid functor $(-)^\simeq \colon \Cat_\infty \to \cS$.
It is right adjoint to the inclusion $\cS \hookrightarrow \Cat_\infty$, and in particular it commutes with limits.
Let $F^\simeq$ denote the composite functor
\[ \begin{tikzcd}
\dAfd_k\op \arrow{r}{F} & \Cat_\infty \arrow{r}{(-)^\simeq} & \cS .
\end{tikzcd} \]
We deduce that $F^\simeq$ satisfies étale descent and that it is cartesian and convergent.
To compute the analytic cotangent complex, we rely on the following lemma:

\begin{lemma} \label{lem:loop_computation_cotangent_complex}
	Let $G \in \St(\dAfdk)$ and let $x \colon X \to G$ be a morphism with $X \in \dAfdk$.
	Let $\Omega_x G \coloneqq X \times_G X$ and $\delta_x \colon X \to \Omega_x G$ the induced diagonal map.
	Assume that $G$ is infinitesimally cartesian.
	Then $G$ has an analytic cotangent complex at $x$ if and only if $\Omega_x G$ has an analytic cotangent complex at $\delta_x$.
	If this is the case, then there is a natural equivalence
	\[ \anL_{G,x} \simeq \anL_{\Omega_x G, \delta_x}[-1] . \]
\end{lemma}

\begin{proof}
	For any $\cF \in \Coh^{\ge 1}(X)$, we have
	\[ \DerAn_{\Omega_x G}(X; \cF) \simeq \Omega \DerAn_G(X; \cF) \simeq \Omega \DerAn_G( X; \cF[-1][1] ) . \]
	Since $G$ is infinitesimally cartesian, we can use \cite[Lemma 7.8]{Porta_Yu_Representability_theorem} to rewrite this as
	\[ \DerAn_G(X;\cF[-1]) \simeq \Map_{\Coh^+(X)}( \anL_{G,x}[1], \cF ) . \]
	Therefore, $\anL_{\Omega_x G, \delta_x} \simeq \anL_{G,x}[1]$.
\end{proof}

\begin{proposition} \label{prop:stack_derived_analytic_spaces_cotangent_complex}
	Let $X \in \dAfd_k$ and let $x \colon X \to F^\simeq$ be a point classifying a flat map $p \colon Y \to X$ from a derived \kanal space $Y$.
	If $p$ is proper, has finite coherent cohomological dimension and locally of finite presentation, then $F^\simeq$ admits an analytic cotangent complex at $x$, given by the formula
	\[ \anL_{F,x} \simeq p_+( \anL_{Y/X} )[-1] , \]
	where $p_+( \anL_{Y/X} ) \coloneqq (p_*( (\anL_{Y/X})^\vee ))^\vee$.
	\end{proposition}

\begin{proof}
	Consider the derived \kanal stack
	\[ G \coloneqq X \times_{F^\simeq} X \]
	and let $\delta_x \colon X \to G$ be the diagonal morphism induced by $x$.
	Since $F^\simeq$ is infinitesimally cartesian, \cref{lem:loop_computation_cotangent_complex} shows that $F^\simeq$ has an analytic cotangent complex at $x$ if and only if $G$ has an analytic cotangent complex at $\delta_x$;
	moreover, in this case we have
	\[ \anL_{F^\simeq,x} \simeq \anL_{G, \delta_x}[1] . \]
	Fix $\cF \in \Coh^{\ge 0}(X)$.
	Unraveling the definitions, we have
	\begin{align*}
		\DerAn_G(X; \cF) & \simeq \Map( X[\cF] , G ) \times_{\Map( X, G )} \{x\} \\
		& \simeq \Map_{/X[\cF]}( Y[p^*\cF], Y[p^*\cF] ) \times_{\Map_{/X}( Y, Y )} \{\id_Y\} \\
		& \simeq \Map_{/X}( Y[p^*\cF] , Y ) \times_{\Map_{/X}( Y, Y )} \{\id_Y\} \\
		& \simeq \DerAn_{Y/X}(Y; p^*\cF) \\
		& \simeq \Map_{\Coh^+(Y)}( \anL_{Y/X}, p^* \cF ) .
	\end{align*}
	Since $p \colon Y \to X$ is locally of finite presentation, $\anL_{Y/X}$ is perfect.
	Since $p$ is also proper flat and has finite coherent cohomological dimension, \cite[Proposition 7.11]{Porta_Yu_Derived_Hom_spaces} allows us to rewrite
	\[ \Map_{\Coh^+( Y )}( \anL_{Y/X}, p^* \cF ) \simeq \Map_{\Coh^+( X )}( p_+( \anL_{Y/X} ), \cF ),\]
	completing the proof.
\end{proof}

For application to the moduli stack of curves with marked points, we generalize the previous results concerning deformations of spaces to deformations of morphisms (varying source and target).
Let $T \in \dAfdk$ and let $p \colon X \to T$, $q \colon Y \to T$ be two proper flat derived $k$-analytic spaces over $T$.
Let $f \colon X \to Y$ be a morphism in $\dAn_{/T}$.
Let
\[ \mathrm{Def}_f \colon \dAfd_{T/\!/T}\op \longrightarrow \cS \]
be the moduli functor of deformations of $f$.
In other words, for $T' \in \dAfd_{T/\!/T}$, $\mathrm{Def}_f(T')$ is the space of commutative diagrams
\[ \begin{tikzcd}[column sep=small, row sep=small]
	X \arrow[shift left = 1pt]{rr}{f} \arrow{ddr}[swap]{p} & & Y \arrow[shift left = 2pt]{drrr} \arrow{ddl}{q} \\
	{} & & & X' \arrow{rr}[swap]{f'} \arrow[crossing over,leftarrow]{ulll} \arrow{ddr}[near start, swap]{p'} & & Y' \arrow{ddl}[near start]{q'} \\
	{} & T \arrow{drrr} \\
	{} & & & & T'
\end{tikzcd} \]
where $p'$ and $q'$ are proper and flat and the side squares are pullbacks.

\begin{proposition} \label{prop:deformation_of_morphisms}
	The functor $\mathrm{Def}_f$ is infinitesimally cartesian and convergent.
	Moreover, let $x \colon T \to \mathrm{Def}_f$ be the point corresponding to the trivial deformation.
	If $p$ and $q$ have finite coherent cohomological dimension, then $\mathrm{Def}_f$ admits an analytic cotangent complex at $x$, which fits in the following pushout diagram:
	\[ \begin{tikzcd}
		p_+ f^*( \anL_{Y/T})[-1] \arrow{r} \arrow{d} & q_+ (\anL_{Y/T})[-1] \arrow{d} \\
		p_+ (\anL_{X/T})[-1] \arrow{r} & \anL_{\mathrm{Def}_f, x} .
	\end{tikzcd} \]
\end{proposition}

\begin{proof}
	Using Propositions \ref{prop:stack_derived_analytic_spaces_cartesian} and \ref{prop:stack_derived_analytic_spaces_convergent}, we deduce that $\mathrm{Def}_f$ is infinitesimally cartesian and convergent.
	Let $G$ be the fiber product
	\[ G \coloneqq T \times_{\mathrm{Def}_f} T , \]
	and $\delta_x \colon T \to G$ the diagonal morphism.
	Then \cref{lem:loop_computation_cotangent_complex} shows that $\mathrm{Def}_f$ has an analytic cotangent complex at $x$ if and only if $G$ has an analytic cotangent complex at $\delta_x$; in this case they are related by the formula
	\[ \anL_{\mathrm{Def}_f, x} \simeq \anL_{G,\delta_x}[-1] . \]
	Let us compute $\anL_{G,\delta_x}$.
	Let $\cF \in \Coh^{\ge 0}(T)$.
	Unraveling the definitions, we have
	\begin{multline*}
		\DerAn_G(T; \cF)  \simeq \Map_{X/}( X[p^* \cF], X ) \times_{\Map_{X/}( X[p^* \cF], Y )} \Map_{Y/}( Y[q^* \cF], Y ) \\
		\simeq \Map_{\Coh^+(X)}( \anL_{X/T}, p^* \cF ) \times_{\Map_{\Coh^+(X)}(f^* \anL_{Y/T}, p^* \cF)} \Map_{\Coh^+(X)}( \anL_{Y/T}, q^* \cF )  .
	\end{multline*}
	Since $p$ and $q$ are proper, flat and have finite coherent cohomological dimension, we can rewrite the fiber product above as
	\[ \Map_{\Coh^+(X)}( p_+( \anL_{X/T}), \cF ) \times_{\Map_{\Coh^+(X)}( p_+ f^*(\anL_{Y/T} ), \cF)} \Map_{\Coh^+(X)}( q_+( \anL_{Y/T} ), \cF ) . \]
	The conclusion now follows from the Yoneda lemma.
\end{proof}

\section{The derived moduli stack of non-archimedean stable maps} \label{sec:stable_maps}

In this section, we introduce the derived moduli stacks of non-archimedean (pre)stable curves and (pre)stable maps.
We prove that the derived moduli stack of stable curves is equivalent to the embedding of the classical moduli stack (see \cref{prop:stack_of_stable_curves}).
This is not the case for stable maps.
We prove that the derived moduli stack of stable maps is a derived $k$-analytic stack that is derived lci (see Theorems \ref{thm:stack_of_stable_maps}, \ref{thm:stack_of_stable_maps_derived_lci}).

\subsection{The derived stack of (pre)stable curves}

\begin{definition} \label{def:stable_curve}
	Let $T\in\dAn_k$ be a derived \kanal space.
	An \emph{$n$-pointed genus $g$ (pre)stable curve over $T$} is a proper flat morphism $p\colon C \to T$ in $\dAnk$ together with $n$ distinct sections $s_i\colon T\to C$ such that every geometry fiber is an $n$-pointed genus $g$ (pre)stable curve.
		\end{definition}

\begin{remark} \label{rem:stable_curve}
	By definition, given $T\in\dAn_k$, the \infcat of $n$-pointed genus $g$ (pre)stable curves over $T$ is a full subcategory of $\dAn_{T^{\amalg n}/\!/T}$.
\end{remark}

Observe that for any $T'\to T$ in $\dAnk$, the pullback of any $n$-pointed genus $g$ (pre)stable curve over $T$ is an $n$-pointed genus $g$ (pre)stable curve over $T'$.
Hence we obtain the derived stacks
\[ \R\oMpre_{g,n}, \R\oM_{g,n} \colon \dAfd_k\op \longrightarrow \cS\]
of $n$-pointed genus $g$ prestable (resp.\ stable) curves.
We denote
\[\oMpre_{g,n}\coloneqq\trunc(\R\oMpre_{g,n})\xhookrightarrow{\ j\ } \R\oMpre_{g,n}, \quad\oM_{g,n}\coloneqq\trunc(\R\oM_{g,n})\xhookrightarrow{\ i\ }\R\oM_{g,n}.\]
Note that by the flatness condition, if $T$ is underived, any (pre)stable curve over $T$ is also underived.
Hence $\oMpre_{g,n}$ and $\oM_{g,n}$ parametrizes underived (pre)stable curves.

\begin{lemma} \label{lem:derived_prestable_lci}
	Let $T \in \dAn_k$ and let $[p \colon C \to T, (s_i)]$ be an $n$-pointed genus $g$ prestable curve over $T$.
	Then $p$ is derived lci, and smooth outside the nodes.
	Moreover, the sections $s_i$ are derived lci.
\end{lemma}
\begin{proof}
	Since $p \colon C \to T$ is flat, the square
	\[ \begin{tikzcd}
	\trunc(C) \arrow{r}{j_C} \arrow{d}{\trunc(p)} & C \arrow{d}{p} \arrow{d} \\
	\trunc(T) \arrow{r}{j_T} & T
	\end{tikzcd} \]
	is a pullback square.
	So we have
	\[ j_C^* \anL_{C /T} \simeq \anL_{\trunc(C) / \trunc(T)} . \]
	Since the map $\trunc(p) \colon \trunc(C) \to \trunc(T)$ is lci, it follows that $\anL_{\trunc(C) / \trunc(T)}$ is perfect and in tor-amplitude $[0,1]$.
	Therefore, the same goes for $\anL_{C / T}$.
	Moreover, $\trunc(p)$ is smooth outside the nodes;
	so by \cite[Proposition 5.50]{Porta_Yu_Representability_theorem}, $\anL_{\trunc(C) / \trunc(T)}$ has tor-amplitude $[0,0]$ outside the nodes.
	Thus, the same goes for $\anL_{C / T}$; and applying the same proposition again, we conclude that $p \colon C \to T$ is smooth outside the nodes.
	
	Let $U \subset C$ be the open complementary of the nodes.
	For every $i = 1, \ldots, n$, we see that $\trunc(s_i) \colon \trunc(T) \to \trunc(C)$ factors through $\trunc(U)$, and hence $s_i$ factor through $U$.
	In other words, the sections $s_i$ are also sections of the smooth map $U \to T$.
	So it follows from \cref{lem:derived_lci_basic_properties}(3) that they are derived lci.
\end{proof}

\begin{lemma} \label{lem:stack_of_prestable_curves}
	The derived stack $\R\oMpre_{g,n}$ is a derived \kanal stack.
\end{lemma}
\begin{proof}
	Using \cref{prop:deformation_of_morphisms}, we deduce that $\R \oMpre_{g,n}$ is infinitesimally cartesian, convergent and admits a global analytic cotangent complex.
	Moreover, its truncation is geometric, so \cite[Theorem 7.1]{Porta_Yu_Representability_theorem} implies that $\R \oMpre_{g,n}$ is a derived \kanal stack.
\end{proof}

\begin{proposition} \label{prop:stack_of_stable_curves}
	The morphisms $j \colon \oMpre_{g,n} \to \R\oMpre_{g,n}$ and $i \colon \oM_{g,n} \to \R\oM_{g,n}$ are equivalences.
\end{proposition}
\begin{proof}
	As $\oM_{g,n}$ is an open substack in $\oMpre_{g,n}$ and $\R\oM_{g,n}$ is the corresponding open substack in $\R\oMpre_{g,n}$, it is enough to prove that $j$ is an equivalence.
	Since $\oMpre_{g,n}$ is smooth, $j$ is an equivalence if and only if $\R\oMpre_{g,n}$ is smooth.
	Therefore, by \cite[Proposition 5.50]{Porta_Yu_Representability_theorem}, it is enough to check that $\anL_{\R\oMpre_{g,n}}$ is perfect and in tor-amplitude $[-1,0]$.
	By \cref{rem:classical_definition_perfect_complex}, tor-amplitude $[-1,0]$ implies perfectness.
	Hence it is enough to prove that for every (underived) $k$-affinoid space $T \in \Afd_k$ and every map $x \colon T \to \R\oMpre_{g,n}$, one has
	\begin{equation} \label{eq:vanishing_pi_i}
	\pi_i\big( x^* \anL_{\oMpre_{g,n}} \big) \simeq 0 \text{ for every }i\ge 1.
	\end{equation}
	Let $[p \colon C \to T, (s_i)]$ be the $n$-pointed genus $g$ prestable curve classified by $x$.
	Let $\mathbb T_{C/T}\an(-\sum s_i)$ be given by the fiber sequence
	\begin{equation} \label{eq:twisted_tangent_complex}
		\mathbb T_{C/T}\an(-{\textstyle\sum}s_i)\longrightarrow\mathbb T\an_{C/T} \longrightarrow \bigoplus_{i = 1}^n s_{i*} s_i^* \mathbb T\an_{C/T}.
	\end{equation}
		By \cref{prop:deformation_of_morphisms} we have
	\begin{equation} \label{eq:cotangent_complex_Mgn}
	x^* \anL_{\R \oMpre_{g,n}} \simeq \big( p_* \big( \mathbb T\an_{C/T}(-{\textstyle\sum}s_i) \big) \big)^\vee [-1] .
	\end{equation}
	Since $p \colon C \to T$ is a curve, in order to check \eqref{eq:vanishing_pi_i}, it suffices to check that
	\[ \pi_{-2} \big( p_*\big( \mathbb T^{\mathrm{an}}_{C / T}(-{\textstyle\sum}s_i) \big) \big) = 0 . \]
	Applying $p_*$ to the fiber sequence \eqref{eq:twisted_tangent_complex} and passing to the long exact sequence of homotopy groups, we are left to check the vanishings
	\[ \pi_{-2}( p_* \mathbb T\an_{C/T} ) = 0\quad \text{and}\quad\pi_{-2}( p_* s_{i*} s_i^* \mathbb T\an_{C/T} ) = 0\quad\text{for every }i = 1, \ldots, n.\]
	By \cref{lem:derived_prestable_lci}, $p \colon C \to T$ is derived lci, hence $\mathbb T\an_{C/T}$ is in tor-amplitude $[-1,0]$.
	As $T$ is underived, $s_i^* \mathbb T\an_{C/T}$ is in homological amplitude $[-1,0]$.
	Moreover, as $T$ is affinoid, \cite[Theorem 3.4]{Porta_Yu_Derived_Hom_spaces} implies that the second vanishing holds.
	As for the first vanishing, using the hypercohomology spectral sequence, we are left to check that
	\[ \pi_0( p_*( \pi_{-2} \mathbb T^{\mathrm{an}}_{C / T} ) ) \simeq \pi_{-1} ( p_*( \pi_{-1} \mathbb T^{\mathrm{an}}_{C/T} ) ) \simeq 0 . \]
	Since $p \colon C \to T$ is derived lci, $\pi_{-2} \mathbb T^{\mathrm{an}}_{C / T} = 0$.
	Moreover, since $p \colon C \to T$ is smooth outside the nodes, $\pi_{-1}( \mathbb T^{\mathrm{an}}_{C / T} )$ is supported on finitely many sections of $p$.
	As $T$ is affinoid, it follows that $\pi_{-1}( p_*( \pi_{-1} \mathbb T^{\mathrm{an}}_{C / T} ) ) \simeq 0$, completing the proof.
\end{proof}

\begin{corollary} \label{cor:universal_curve}
	Let $\oCpre_{g,n}\to\oMpre_{g,n}$, $\oC_{g,n}\to\oM_{g,n}$, $\R\oCpre_{g,n}\to\R\oMpre_{g,n}$ and $\R\oC_{g,n}\to\R\oM_{g,n}$ be the universal curves.
	We have equivalences
	\[\begin{tikzcd}
	\oCpre_{g,n} \rar{\sim} \dar & \R\oCpre_{g,n}\dar\\
	\oMpre_{g,n} \rar{\sim} & \R\oMpre_{g,n}
	\end{tikzcd}\text{and}\quad
	\begin{tikzcd}
	\oC_{g,n} \rar{\sim} \dar & \R\oC_{g,n}\dar\\
	\oM_{g,n} \rar{\sim} & \R\oM_{g,n}
	\end{tikzcd}.\]
\end{corollary}
\begin{proof}
	By definition, the vertical maps are all flat.
	By \cref{prop:stack_of_stable_curves}, the stacks $\R\oMpre_{g,n}$ and $\R\oM_{g,n}$ are underived.
	Hence the stacks $\R\oCpre_{g,n}$ and $\R\oC_{g,n}$ are also underived, from which the conclusion follows.
\end{proof}

\subsection{The derived stack of (pre)stable maps}

Fix a rigid \kanal space $S$ and a rigid \kanal space $X$ smooth over $S$.

\begin{definition} \label{def:stable_map}
	Let $T\in\dAn_{/S}$ be a derived \kanal space over $S$.
	An \emph{\gn prestable map into $X/S$ over $T$} consists of an \gn prestable curve $[C\to T, (s_i)]$ over $T$ and an $S$-map $f\colon C\to X$.
	It is called \emph{stable} if every geometric fiber $[C_t, (s_i(t)), f_t\colon C_t\to X]$ is a stable map, in the sense that its automorphism group is a finite analytic group.
\end{definition}

\begin{remark}
	An equivalent characterization for $[C_t, (s_i(t)), f_t\colon C_t\to X]$ to be stable is that every irreducible component of $C_t$ of genus 0 (resp.\ 1) which maps to a point must have at least 3 (resp.\ 1) special points on its normalization, where by special point, we mean the preimage of either a marked point or a node by the normalization map.
		We refer to \cite{Kontsevich_Enumeration,Fulton_Notes_on_stable_maps} for the theory of stable maps in algebraic geometry.
\end{remark}

\begin{remark} \label{rem:stable_map}
	By definition, given $T\in\dAn_{/S}$, the \infcat of \gn (pre)stable maps into $X/S$ over $T$ is a full subcategory of $\dAn_{T^{\amalg n}/\!/T\times_S X}$.
\end{remark}

Observe that for any $T'\to T$ in $\dAnk$, the pullback of any \gn (pre)stable map into $X/S$ over $T$ is an \gn (pre)stable map into $X/S$ over $T'$.
Hence we obtain the derived stacks
\[\R\oMpre_{g,n}(X/S),\R\oM_{g,n}(X/S)\colon\dAfd_k\op\longrightarrow\cS\]
of \gn prestable (resp.\ stable) maps into $X/S$.
More formally, we define
\[\R\oMpre_{g,n}(X/S) \coloneqq \bfMap_{\R\oMpre_{g,n}\times S}(\R\oCpre_{g,n}\times S,X\times\R\oMpre_{g,n}),\]
and $\R\oM_{g,n}(X/S)$ is the open substack of $\R\oMpre_{g,n}(X/S)$ consisting of stable maps.
We denote
\[\oMpre_{g,n}(X/S)\coloneqq\trunc(\R\oMpre_{g,n}(X/S)), \quad\oM_{g,n}(X/S)\coloneqq\trunc(\R\oM_{g,n}(X/S)).\]
Note that by the flatness condition, if $T$ is underived, any (pre)stable map into $X/S$ over $T$ is also underived.
Hence $\oMpre_{g,n}(X/S)$ and $\oM_{g,n}(X/S)$ parametrize underived (pre)stable maps into $X/S$.

\begin{theorem} \label{thm:stack_of_stable_maps}
		The derived stacks $\R\oMpre_{g,n}(X/S)$ and $\R\oM_{g,n}(X/S)$ are derived \kanal stacks locally of finite presentation over $S$.
\end{theorem}
\begin{proof}
	By the definition of $\R\oMpre_{g,n}(X/S)$ and \cref{cor:universal_curve}, we have
	\[\R\oMpre_{g,n}(X/S) =\bfMap_{\R\oMpre_{g,n}\times S}(\R\oCpre_{g,n}\times S, X\times\R\oMpre_{g,n})\simeq\bfMap_{\oMpre_{g,n}\times S}(\oCpre_{g,n}\times S, X\times\oMpre_{g,n}).\]
	By \cite[Theorem 1.1]{Porta_Yu_Derived_Hom_spaces}, the right hand side is a derived \kanal stack locally of finite presentation over $S$.
	We conclude that $\R\oMpre_{g,n}(X/S)$ is a derived \kanal stack locally of finite presentation over $S$.
	Since the stability condition is an open condition, the inclusion
	\[\R\oM_{g,n}(X/S)\hookrightarrow\R\oMpre_{g,n}(X/S)\]
	is Zariski open.
	We deduce that $\R\oM_{g,n}(X/S)$ is also a derived \kanal stack locally of finite presentation over $S$.
\end{proof}

\begin{remark}
	If $S$ is an algebraic variety over $k$ and $X$ is an algebraic variety over $S$, then by \cite{Holstein_Analytification_of_mapping_stacks}, we have
	\[\R\oMpre_{g,n}(X/S)\an\xrightarrow{\ \sim\ }\R\oMpre_{g,n}(X\an/S\an),\]
	and hence
	\[\R\oM_{g,n}(X/S)\an\xrightarrow{\ \sim\ }\R\oM_{g,n}(X\an/S\an).\]
\end{remark}

\begin{theorem} \label{thm:stack_of_stable_maps_derived_lci}
	The derived \kanal stacks $\R\oMpre_{g,n}(X/S)$ and $\R\oM_{g,n}(X/S)$ are derived lci over $S$.
\end{theorem}
\begin{proof}
	Since $\R\oM_{g,n}(X/S)\hookrightarrow\R\oMpre_{g,n}(X/S)$ is Zariski open, it suffices to prove that $\R\oMpre_{g,n}(X/S)$ is derived lci over $S$.
	
	We start by computing explicitly the relative analytic cotangent complex of $\R \oMpre_{g,n}(X/S)$ over $S$.
	Let $T \in \dAn_{/S}$ and $x\colon T\to\R\oMpre_{g,n}(X/S)$.
	Recall $\oMpre_{g,n}\simeq\R\oMpre_{g,n}$ by \cref{prop:stack_of_stable_curves}.
	Let $y\colon T\to\oMpre_{g,n}\times S$ be the composition of $x$ with the map $\R\oMpre_{g,n}(X/S)\to\oMpre_{g,n}\times S$ taking domain curves.
	Let
	\[ \begin{tikzcd}
		C \arrow{d}{p} \arrow{r}{f} & X \\
		T \arrow[bend left = 30pt]{u}{s_i}
	\end{tikzcd} \]
	be the $n$-pointed genus $g$ stable map classified by $x$.
	Then \cite[Lemma 8.4]{Porta_Yu_Derived_Hom_spaces} implies that we have a canonical equivalence
	\[ x^* \anL_{\R \oMpre_{g,n}(X/S) / (\oMpre_{g,n}\times S)} \simeq p_+( f^* \anL_{X/S} ) . \]
	The transitivity fiber sequence for the analytic cotangent complex now yields
	\[ y^* \anL_{\oMpre_{g,n}\times S} \longrightarrow x^* \anL_{\R \oMpre_{g,n}(X/S)/S} \longrightarrow x^* \anL_{\R \oMpre_{g,n}(X/S) / (\oMpre_{g,n}\times S)} . \]
	By \eqref{eq:cotangent_complex_Mgn}, we obtain
	\[ x^* \mathbb T\an_{\R \oMpre_{g,n}(X/S)/S} \simeq p_* \left( \cofib\left( \mathbb T\an_{C/T}(-{\textstyle\sum}s_i) \to f^* \mathbb T\an_{X/S} \right) \right) . \]
	At this point, in order to check that $\R \oMpre_{g,n}(X/S)$ is derived lci over $S$, it suffices to verify that for every $x \colon T \to \R \oMpre_{g,n}(X/S)$ with $T$ underived over $S$, we have
	\[ \pi_j\left( x^* \mathbb T\an_{\R \oMpre_{g,n}(X/S)/S} \right) \simeq0\quad\text{for every }j\le -2.\]
	By \cref{lem:derived_prestable_lci}, the map $p \colon C \to T$ and the sections $s_i \colon T \to C$ are all derived lci.
	Therefore, the fiber sequence \eqref{eq:twisted_tangent_complex} implies that $\mathbb T\an_{C/T}( - \sum s_i )$ is in homological amplitude $[-1,0]$.
	As $p \colon C \to T$ is a relative curve, we have
	\[ \pi_j \left( p_* \left( \mathbb T\an_{C/T}(-{\textstyle\sum}s_i) \right) \right) \simeq 0\quad\text{for every }j\le -3.\]
		Moreover, since $X$ is smooth over $S$, $f^* \mathbb T\an_{X/S}$ is concentrated in degree $0$, thus $\pi_j( p_* f^* \mathbb T\an_{X/S} ) \simeq 0$ for every $j\le -2$, whence the conclusion follows.
\end{proof}

\subsection{Stable maps associated to A-graphs} \label{sec:stable_maps_associated_to_A-graphs}

The stack of stable maps has a natural refinement indexed by combinatorial data which we describe below:

\begin{definition}
	A \emph{cycle} on a rigid \kanal space $X$ is a finite formal integral linear combination of closed irreducible reduced analytic subspaces.
	Two cycles $Z$ and $Z'$ on $X$ are called \emph{analytically equivalent} if there is a connected smooth rigid \kanal curve $T$ and a cycle $V$ on $X\times T$ flat over $T$, such that
	\[ [V_t] - [V_{t'}] = [Z] - [Z'] \]
	for two rigid points $t, t'$ on $T$.
	Let $A(X)$ denote the group of one-dimensional cycles on $X$ modulo analytic equivalence.
\end{definition}

For every class $\beta\in A(X)$, we denote by $\oM_{g,n}(X/S,\beta)$ the connected component of $\oM_{g,n}(X/S)$ consisting of stable maps such that the fundamental class of the domain curve maps to $\beta$, and we define $\R\oM_{g,n}(X/S,\beta)$ similarly.

\begin{definition}[see \cref{fig:modular_graph}, {\cite[Definition 1.5]{Behrend_Stacks_of_stable_maps}}] \label{def:A-graph}
	A \emph{modular graph} consists of the following data:
	\begin{enumerate}
		\item A finite graph $\tau$ (we allow multiple edges and loops).
		\item A subset $T_\tau$ of 1-valent vertices of $\tau$ called \emph{tail vertices}, such that each edge $e$ of $\tau$ contains at most one tail vertex;
		an edge containing a tail vertex is called a \emph{tail};
		we denote by $V_\tau$ the set of non-tail vertices of $\tau$, and by $E_\tau$ the set of non-tail edges of $\tau$.
		The tails are numbered, i.e.\ a bijection between $T_\tau$ and $\{1,\dots,\abs{T_\tau}\}$ is fixed.
		\item For each $v\in V_\tau$, a non-negative integer $g(v)$ called \emph{genus}.
	\end{enumerate}
	\begin{figure}[!ht]
		\centering
		\setlength{\unitlength}{0.6\textwidth}
		\begin{picture} (1,0.234)
			\put(0,0){\includegraphics[width=\unitlength, page=1]{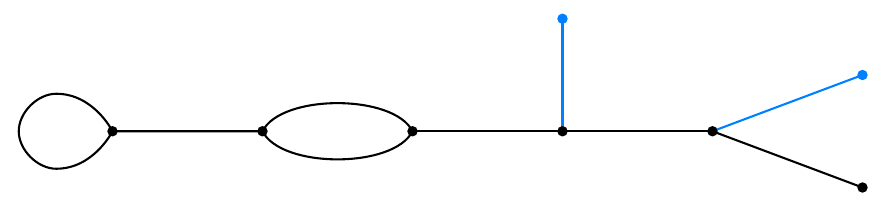}}
		\end{picture}
		\caption{A modular graph with two tails (colored blue).}
		\label{fig:modular_graph}
	\end{figure}
	For every non-tail edge $e$, we choose a point $p$ in the middle and will consider the two \emph{half-edges} incident to $p$.
	Given $v\in V_\tau$, we denote by $E_v$ the set of tails and half-edges connected to $v$;
	we call the cardinality of $E_v$ the \emph{valence} of $v$, and denote it by $\val(v)$.
	The \emph{genus} $g(\tau)$ of the modular graph $\tau$ is the genus of the underlying graph plus the sum of all $g(v)$.
	An A-graph is a pair $(\tau,\beta)$ consisting of a modular graph $\tau$ and a map $\beta\colon V_\tau\to A(X)$.
	
	For the simplicity of exposition, we will restrict to stable modular graphs, i.e.\ $2g(v)+\val(v)\ge 3$ for every vertex $v$.
	Operations on A-graphs with unstable underlying modular graphs are more complicated, and we refer to \cite{Behrend_Stacks_of_stable_maps} for details.
\end{definition}

\begin{definition}
	Let $\tau$ be a modular graph and $T \in \dAnk$ a derived \kanal space.
	A \emph{$\tau$-marked derived \kanal space}  $[X_v, (s_{v,e})_{e \in E_v}]_{v \in V_\tau}$ over $T$ is a collection of separated derived \kanal spaces $X_v$ over $T$ equipped with sections $s_{v,e} \colon T \to X_v$ for every $e \in E_v$.
	
	Note that $\tau$-marked derived \kanal spaces over $T$ form naturally an $\infty$-category $\dAn^\tau_T$, which is a full subcategory of the product $\prod_{v \in V_\tau} \dAn_{T^{\amalg \val(v)} /\!/ T}$.
\end{definition}

\begin{construction} \label{const:tau_gluing}
	Let $\tau$ be a modular graph and let $[X_v, (s_{v,e})_{e\in E_\tau}]_{v\in V_\tau}$ be a $\tau$-marked derived \kanal space over $T$.
	Since each $X_v$ is separated over $T$, the sections $s_{v,e} \colon T \to X_v$ are closed immersions.
	For each non-tail edge $e$ of $\tau$, let $e_1$ and $e_2$ be the two corresponding half-edges, and $v_1$, $v_2$ the two endpoints ($v_1 = v_2$ if $e$ is a loop).
	Using \cite[Theorem 6.5]{Porta_Yu_Representability_theorem}, we glue these sections and obtain a new derived \kanal space over $T$ (see \cref{fig:glug_sections}).
		The universal property of the gluing induces a functor
	\[ \gamma_\tau \colon \dAn_{/T}^\tau \longrightarrow \dAn_{/T} . \]
	We refer to $\gamma_\tau$ as the \emph{$\tau$-gluing functor}.
	\begin{figure}[!ht]
	\centering
	\setlength{\unitlength}{0.5\textwidth}
	\begin{picture} (1,0.767)
		\put(0,0){\includegraphics[width=\unitlength, page=1]{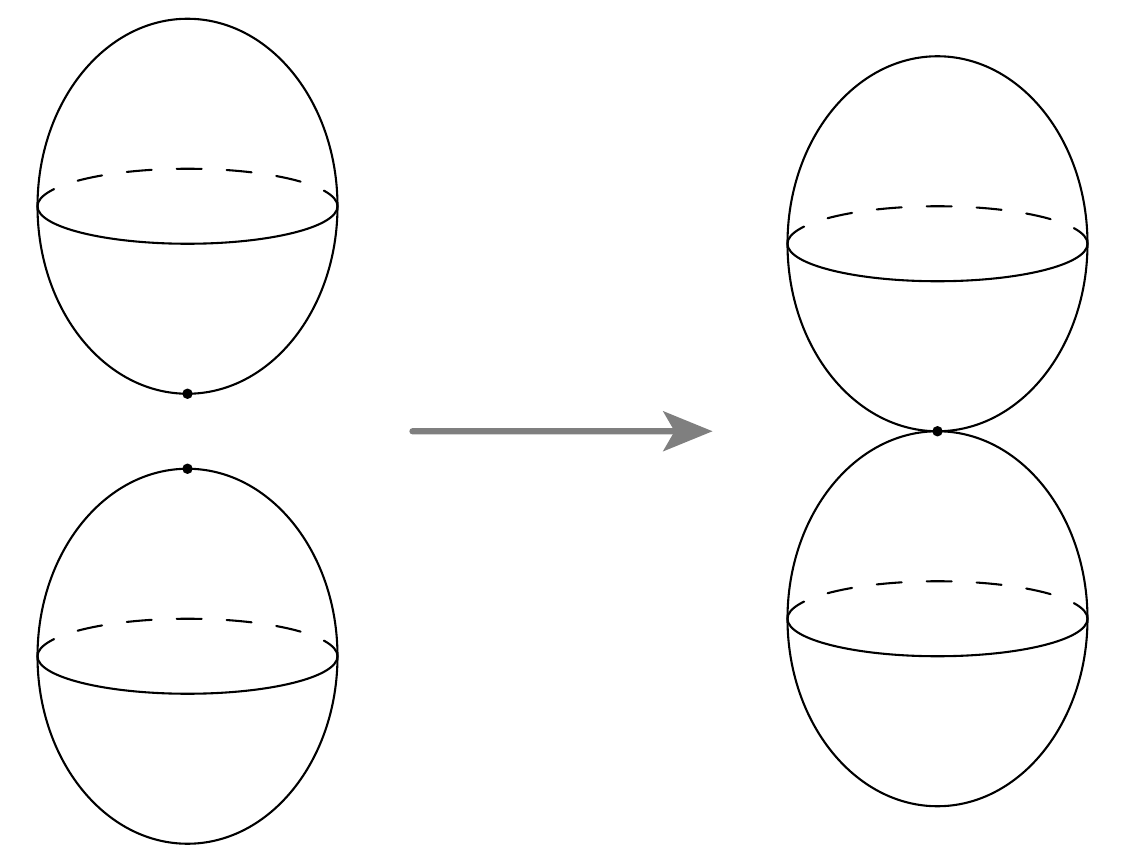}}
		\put(-0.03,0.7){$X_{v_1}$}
		\put(-0.03,0.3){$X_{v_2}$}
		\put(0.12,0.44){$s_{v_1,e}$}
		\put(0.12,0.31){$s_{v_2,e}$}
	\end{picture}
	\caption{Glue sections.}
	\label{fig:glug_sections}
	\end{figure}
\end{construction}

\begin{definition}
	Given an $n$-pointed genus $g$ prestable curve $[C,(s_i)]$ over an algebraically closed field, let $\Gamma$ be the dual graph of $C$.
	For each vertex $v$ of $\Gamma$, let $g(v)$ be the genus of the corresponding irreducible component $C_v$ of $C$.
	For each marked point $s_i$, we add a vertex $v_i$ to $\Gamma$, and an edge $e_i$ connecting $v_i$ and the vertex of $\Gamma$ corresponding to the irreducible component of $C$ containing $s_i$.
	Then we obtain a modular graph $\tau$ with $T_\tau=\{v_i\}$, which we call the modular graph associated to $[C,(s_i)]$.
\end{definition}

\begin{definition} \label{def:graph-marked_stable_curve}
	Let $\tau$ be a modular graph and $T \in \dAnk$ a derived \kanal space.
	A \emph{$\tau$-marked (pre)stable curve over $T$} consists of a $\tau$-marked derived \kanal space $[C_v, (s_{v,e})_{e \in E_v}]_{v \in V_\tau}$ over $T$ such that for every $v \in V_\tau$, $[C_v, (s_{v,e})_{e \in E_v}]$ is a $\val(v)$-pointed genus $g(v)$ (pre)stable curve over $T$.
\end{definition}

\begin{definition} \label{def:gluing}
	Let $\tau$ be a modular graph and $T \in \dAnk$ a derived \kanal space.
	Let $[C_v, (s_{v,e})]_{v \in V_\tau}$ be a $\tau$-marked (pre)stable curve over $T$.
	Let $C$ be its $\tau$-gluing (see \cref{const:tau_gluing}).
	It inherits sections $s_e$ for every non-tail edge $e$ of $\tau$ and $s_i$ for every tail edge $i$ of $\tau$.
	We call $[C, (s_i)_{i \in T_\tau}, (s_e)_{e \in E_\tau}]$ the \emph{glued curve} associated to the $\tau$-marked (pre)stable curve $[C_v, (s_{v,e})_{e \in E_v}]_{v \in V_\tau}$.
	Note that the arithmetic genus of $C$ is equal to the genus of $\tau$.
	Given any geometric point $t\in T$, consider the geometric fiber $[C_t,(s_i(t))_{i\in T_\tau}, (s_e(t))_{e\in E_\tau}]$.
	By construction, every $s_i(t)$ lies in the smooth locus of $C_t$, and every $s_e(t)$ is a node of $C_t$.
	We also remark that the modular graph associated to the pointed (pre)stable curve $[C_t,(s_i(t))_{i\in T_\tau}]$ is not $\tau$ in general, but it admits a contraction to $\tau$.
	\end{definition}

\begin{remark} \label{rem:product_decomposition_M_tau}
	Let $\oMpre_\tau$ and $\R\oMpre_\tau$ denote respectively the underived and the derived stack of $\tau$-marked prestable curves.
	Let $\oM_\tau$ and $\R\oM_\tau$ denote respectively the underived and the derived stack of $\tau$-marked stable curves.
	By definition, we have natural equivalences
	\begin{align*}
	&\oMpre_\tau\simeq\prod_{v\in V_\tau}\oMpre_{g(v),\val(v)}, &&\oM_\tau\simeq\prod_{v\in V_\tau}\oM_{g(v),\val(v)},\\
	&\R\oMpre_\tau\simeq\prod_{v\in V_\tau}\R\oMpre_{g(v),\val(v)}, &&\R\oM_\tau\simeq\prod_{v\in V_\tau}\R\oM_{g(v),\val(v)}.
	\end{align*}
	It follows that the moduli stacks above do not change when we cut an edge into two tails (see \cref{sec:cutting_edges}).

	Let $\oCpre_\tau$, $\R\oCpre_\tau$, $\oC_\tau$ and $\R\oC_\tau$ be the universal glued curves over the moduli stacks $\oMpre_\tau$, $\R\oMpre_\tau$, $\oM_\tau$ and $\R\oM_\tau$ respectively.
	They are well-defined by \cref{lem:gluing_along_closed_immersion_is_universal}.
	Then by \cref{cor:universal_curve}, we obtain natural equivalences
	\[\begin{tikzcd}
	\oCpre_\tau \rar{\sim} \dar & \R\oCpre_\tau\dar\\
	\oMpre_\tau \rar{\sim} & \R\oMpre_\tau
	\end{tikzcd}\text{and}\quad
	\begin{tikzcd}
	\oC_\tau \rar{\sim} \dar & \R\oC_\tau\dar\\
	\oM_\tau \rar{\sim} & \R\oM_\tau
	\end{tikzcd}.\]
\end{remark}

\begin{definition} \label{def:graph-marked_stable_map}
	Let $(\tau,\beta)$ be an A-graph and let $T\in\dAn_{/S}$ be a derived \kanal space over $S$.
	A \emph{$(\tau,\beta)$-marked prestable map into $X/S$ over $T$} consists of a $\tau$-marked prestable curve $[C_v,(s_{v,e})_{e\in E_v}]_{v\in V_\tau}$ over $T$ and an $S$-map $f$ from the associated glued curve $[C,(s_i)_{i\in T_\tau},(s_e)_{e\in E_\tau}]$ to $X$ such that for every geometric point $t\in T$ and every $v\in V_\tau$, the composite map $C_{v,t}\to C_t\xrightarrow{f_t} X$ has class $\beta(v)$.
	It is called \emph{stable} if every geometric fiber $[C_t,(s_i(t))_{i\in T_\tau},f_t]$ is a stable map.
	
	Note that $(\tau,\beta)$-marked (pre)stable maps into $X/S$ over $T$ form naturally an $\infty$-category, which is a full subcategory of the fiber product
	\[ \dAn_{/ T \times_S X} \times_{\dAn_{/T}} \dAn_{/T}^\tau , \]
	where the functor $\dAn_{/T}^\tau \to \dAn_{/T}$ is the $\tau$-gluing functor $\gamma_\tau$ of \cref{const:tau_gluing}.
\end{definition}

\begin{remark} \label{rem:RMXtaubeta}
	Let $\oMpre(X/S,\tau,\beta)$ and $\R\oMpre(X/S,\tau,\beta)$ denote respectively the underived and the derived stack of $(\tau,\beta)$-marked prestable maps into $X/S$.
	Let $\oM(X/S,\tau,\beta)$ and $\R\oM(X/S,\tau,\beta)$ denote respectively the underived and the derived stack of $(\tau,\beta)$-marked stable maps into $X/S$.
		As in the proof \cref{thm:stack_of_stable_maps}, we have Zariski open embeddings
	\begin{gather*}
	\oM(X/S,\tau,\beta)\longhookrightarrow\oMpre(X/S,\tau,\beta)\simeq\trunc\bfMap_{\oMpre_\tau\times S}(\oCpre_\tau\times S, X\times\oMpre_\tau),\\
	\R\oM(X/S,\tau,\beta)\longhookrightarrow\R\oMpre(X/S,\tau,\beta)\simeq\bfMap_{\oMpre_\tau\times S}(\oCpre_\tau\times S, X\times\oMpre_\tau).
	\end{gather*}
	Hence $\oM(X/S,\tau,\beta)$ is a rigid \kanal stack over $S$, and $\R\oM(X/S,\tau,\beta)$ is a derived \kanal stack over $S$.
	Similar to \cref{thm:stack_of_stable_maps_derived_lci}, the derived \kanal stack $\R\oM(X/S,\tau,\beta)$ is moreover derived lci over $S$.
	
	Let $\oCpre(X/S,\tau,\beta)$, $\R\oCpre(X/S,\tau,\beta)$, $\oC(X/S,\tau,\beta)$ and $\R\oC(X/S,\tau,\beta)$ be the universal glued curves over the moduli stacks $\oMpre(X/S,\tau,\beta)$, $\R\oMpre(X/S,\tau,\beta)$, $\oM(X/S,\tau,\beta)$ and $\R\oM(X/S,\tau,\beta)$ respectively; in other words, they are respectively the pullbacks of $\oCpre_\tau$, $\R\oCpre_\tau$, $\oC_\tau$ and $\R\oC_\tau$ along the maps taking domains.	
\end{remark}

\subsection{The stabilization map} \label{sec:stabilization_map}

The goal of this subsection is to construct the stabilization of prestable maps in families
\[ \stab \colon \R \oMpre(X/S, \tau, \beta) \longrightarrow \R \oM(X/S, \tau, \beta).\]
We begin by constructing the stabilizing contraction in the sense of the following definition:

\begin{definition} \label{def:stabilization}
	Let $T \in \dAn_{/S}$ be a derived \kanal space over $S$.
	Let $\mathbf f \coloneqq ([C_v,\allowbreak (s_{v,e})_{e \in E_v}]_{v \in V_\tau},\allowbreak f)$ and $\mathbf f' \coloneqq ([C'_v, (s'_{v,e})_{e \in E_v}]_{v \in V_\tau}, f')$ be two $\tau$-marked prestable maps into $X/S$ over $T$.
	We say that a morphism $p \colon \mathbf f \to \mathbf f'$ is a \emph{stabilizing contraction} if $\mathbf f'$ is stable and for every other $\tau$-marked stable map $\mathbf f''$ the morphism
	\[ \Map_{/T}( \mathbf f', \mathbf f'' ) \longrightarrow \Map_{/T}( \mathbf f, \mathbf f'' ) \]
	is an equivalence.
	Here the mapping spaces are computed in the $\infty$-category of $\tau$-marked prestable maps into $X/S$ over $T$ (see \cref{def:graph-marked_stable_map}).
\end{definition}

\subsubsection{A criterion for stabilization}

Before constructing the stabilizing contraction, we describe a criterion for checking its universal property.

\begin{lemma} \label{lem:derived_factorization_lemma}
		Let $Y\in\dAn_{/T}$ be a derived \kanal space over $T$, and $p\colon C\to C'$ a morphism in $\dAn_{/T}$.
	Assume the following:
	\begin{enumerate}
		\item $p$ is proper and has finite coherent cohomological dimension;
		\item both $C$ and $C'$ are proper and flat over $T$;
		\item the canonical map
		\[ \cO_{C'}\alg \longrightarrow p_* \cO_C\alg \]
		is an equivalence;
		\item $Y \to T$ is flat, separated and locally of finite presentation.
	\end{enumerate}
	Then $\bfMap_{/T}(C, Y)$ and $\bfMap_{/T}(C', Y)$ are derived \kanal stacks locally of finite presentation over $T$, and the induced map
	\begin{equation} \label{eq:deforming_maps}
	\pi \colon \bfMap_{/T}(C', Y) \longrightarrow \bfMap_{/T}(C, Y)
	\end{equation}
	is étale.
	In particular the diagram
	\[ \begin{tikzcd}
	\Map_{/T}( C', Y ) \arrow{r} \arrow{d} & \Map_{/ \trunc(T)}( C_0', \trunc(Y) ) \arrow{d} \\
	\Map_{/T}( C, Y ) \arrow{r} & \Map_{/\trunc(T)}( C_0, \trunc(Y) )
	\end{tikzcd} \]
	is a pullback square.
	\end{lemma}

\begin{proof}
	Let $U\in\dAn_{/T}$, and assume that it is underived.
	Since $C$ and $Y$ are flat over $T$, we have
	\begin{align*}
	\bfMap_{/T}( C , Y )(U) & \simeq \Map_{/T}( C \times_T U, Y \times_T U ) \\
	& \simeq \Map_{/\trunc(T)}( \trunc(C) \times_{\trunc(T)} U, \trunc(Y) \times_T U ) \\
	& \simeq \bfMap_{/\trunc(T)}( \trunc(C) , \trunc(Y))(U) .
	\end{align*}
	In other words, we have a canonical equivalence
	\[ \trunc( \bfMap_{/T}( C, Y ) ) \simeq \trunc( \bfMap_{/\trunc(T)}( \trunc(C), \trunc(Y) ) . \]
	Therefore, by \cite[Proposition 5.3.3]{Conrad_Spreading-out}, the truncation of $\bfMap_{/T}( C, Y)$ is representable.
	Then it follows from \cite[Proposition 8.5]{Porta_Yu_Derived_Hom_spaces} that $\bfMap_{/T}(C, Y)$ and $\bfMap_{/T}(C', Y)$ are derived \kanal stacks locally of finite presentation over $T$.
	
	In order to prove that the map \eqref{eq:deforming_maps} is étale, it is enough to show that the relative analytic cotangent complex of is zero.
		Denote $F \coloneqq \bfMap_{/T}( C, Y)$ and $F' \coloneqq \bfMap_{/T}(C', Y)$.
	Then it is enough to check that for any (underived) $k$-affinoid space $U$ over $T$ and any $u \colon U \to F'$, the map
	\[ u^* \pi^* \anL_{F/T} \longrightarrow u^* \anL_{F'/T} \]
	is an equivalence.
	Write $C_U \coloneqq U \times_T C$ and $C'_U \coloneqq U \times_T C'$.
	Let $f \colon C_U \to Y$ and $f' \colon C'_U \to Y$ be the morphisms corresponding to $\pi \circ u$ and $u$, respectively.
	Let
	\[ q_U \colon C_U \to U , \quad q_U' \colon C'_U \to U , \quad p_U \colon C_U \to C'_U  \]
	be the induced morphisms.
	Then $f \simeq f' \circ p_U$ and $q_U \simeq q'_U \circ p_U$.
	Using \cite[Lemma 8.4]{Porta_Yu_Derived_Hom_spaces} we can rewrite
	\[ u^* \bbT\an_{F / T} \simeq q_{U*}( f^*( \bbT\an_Y ) ) , \qquad u^* \pi^* \bbT\an_{F /T} \simeq q'_{U*}( f^{\prime *}( \bbT\an_Y ) ) . \]
	Since $p_U \colon C_U \to C_U'$ is proper, by assumption (3) and the projection formula (\cite[Theorem 1.4]{Porta_Yu_Derived_Hom_spaces}), we have
	\[ q_{U*}( f^*( \bbT\an_Y ) ) \simeq q'_{U*}( p_{U*}( p_U^*( f^{\prime *} \bbT\an_Y ) ) ) \simeq q'_{U*}( p_{U*} \cO_{C_U}\alg \otimes f^{\prime *} \bbT\an_Y ) \simeq q'_{U*}( f^{\prime*}( \bbT\an_Y ) ),\]
	completing the proof.
\end{proof}

\begin{proposition} \label{prop:checking_stabilization_on_truncation}
	Let $T \in \dAn_{/S}$ be a derived \kanal space over $S$.
	Let $\mathbf f = ( [C_v,\allowbreak (s_{v,e})_{e \in E_v}]_{v \in V_\tau}, f )$ and $\mathbf f' = ( [C_v', (s_{v,e}')_{e \in E_v}]_{v \in V_\tau}, f' )$ be two $\tau$-marked prestable maps into $X/S$ over $T$ and let $p \colon \mathbf f \to \mathbf f'$ be a morphism between them.
	Assume that:
	\begin{enumerate}
		\item the canonical map $\cO_{C'_v}\alg \to p_* \cO_{C_v}\alg$ is an equivalence;
		\item the truncation $p_0 \colon \mathbf f_0 \to \mathbf f'_0$ is a stabilizing contraction.
	\end{enumerate}
	Then $p$ is a stabilizing contraction as well.
\end{proposition}

\begin{proof}
	It follows from the definition that $\mathbf f'$ is stable if and only if its truncation $\mathbf f'_0$ is stable.
	Let now $\mathbf f''$ be any other $\tau$-marked stable map into $X/S$ over $T$.
	Write $T_0 \coloneqq \trunc(T)$.
	By assumption (2), it is enough to prove that the diagram
	\begin{equation} \label{eq:stabilization_truncation}
		\begin{tikzcd}
			\Map_{/T}( \mathbf f', \mathbf f'' ) \arrow{r} \arrow{d} & \Map_{/T_0}( \mathbf f'_0, \mathbf f''_0 ) \arrow{d} \\
			\Map_{/T}( \mathbf f, \mathbf f'' ) \arrow{r} & \Map_{/T_0}( \mathbf f_0, \mathbf f''_0 )
		\end{tikzcd}
	\end{equation}
	is a pullback square.
	Write
	\begin{align*}
		&\mathbf f = ( [C_v, (s_{v,e})_{e \in E_v}]_{v \in V_\tau}, f ), &&\mathbf C \coloneqq [C_v, (s_{v,e})_{e \in E_v}]_{v \in V_\tau} , \\
		&\mathbf f_0 = ( [(C_v)_0, ((s_{v,e})_0)_{e \in E_v}]_{v \in V_\tau}, f_0 ), &&\mathbf C_0 \coloneqq [(C_v)_0, ((s_{v,e})_0)_{e \in E_v}]_{v \in V_\tau} .
	\end{align*}
	Let $[C, (s_i)_{i \in T_\tau}, (s_e)_{e \in E_\tau}]$ and $[C_0, ( (s_i)_0 )_{i \in T_\tau}, ((s_e)_0)_{e \in E_\tau}]$ be the associated glued curves.
	We introduce analogous notations for $\mathbf f'$ and $\mathbf f''$.
	By \cref{def:graph-marked_stable_map}, we have
	\[ \Map_{/T}( \mathbf f, \mathbf f'' ) \simeq \Map_{/T \times_S X}( C, C'' ) \times_{\Map_{/T}(C, C'')} \Map_{\dAn_{/T}^\tau}( \mathbf C, \mathbf C'' ) . \]
	The other mapping spaces appearing in diagram \eqref{eq:stabilization_truncation} have analogous descriptions.
	We are therefore left to check that the diagrams
	\begin{equation} \label{eq:stabilization_truncation_I}
		\begin{tikzcd}
			\Map_{/ T}( C', C'' ) \arrow{r} \arrow{d} & \Map_{/T_0}( C'_0, C''_0 ) \arrow{d} \\
			\Map_{/ T}( C, C'' ) \arrow{r} & \Map_{/ T_0}( C_0, C''_0 )
		\end{tikzcd} ,
	\end{equation}
	\begin{equation}\label{eq:stabilization_truncation_II}
		\begin{tikzcd}
			\Map_{/ T \times_S X}(C', C'') \arrow{r} \arrow{d} & \Map_{/ T_0 \times_S X}( C'_0, C''_0 ) \arrow{d} \\
			\Map_{/ T \times_S X}( C, C'' ) \arrow{r} & \Map_{/T_0 \times_S X}( C_0, C''_0 )
		\end{tikzcd},
	\end{equation}
	\begin{equation}\label{eq:stabilization_truncation_III}
		\begin{tikzcd}
			\Map_{\dAn_{/T}^\tau}( \mathbf C', \mathbf C'' ) \arrow{r} \arrow{d} & \Map_{\dAn_{/T_0}^\tau}( \mathbf C'_0, \mathbf C''_0 ) \arrow{d} \\
			\Map_{\dAn_{/T}^\tau}( \mathbf C, \mathbf C'' ) \arrow{r} & \Map_{\dAn_{/T_0}^\tau}( \mathbf C_0, \mathbf C''_0 )
		\end{tikzcd}
	\end{equation}
	are pullback squares.
	
	For diagram \eqref{eq:stabilization_truncation_I}, the statement follows directly from \cref{lem:derived_factorization_lemma}.
	For diagram \eqref{eq:stabilization_truncation_II}, we rewrite
	\[ \Map_{/ T \times_S X}( C, C'' ) \simeq \fib\big( \Map_{/T}( C, C'' ) \to \Map_{/T}( C, T \times_S X ) \big) , \]
	where the fiber is taken at the given map $f \colon C \to T \times_S X$.
	The other mapping spaces appearing in diagram \eqref{eq:stabilization_truncation_II} admit analogous descriptions.
	Applying \cref{lem:derived_factorization_lemma} twice, we deduce that diagram \eqref{eq:stabilization_truncation_II} is also a pullback.
	
	We are left to prove that diagram \eqref{eq:stabilization_truncation_III} is a pullback.
	Unraveling the definition of the $\infty$-category $\dAn_{/T}^\tau$, we can rewrite
	\[ \Map_{\dAn_{/T}^\tau}( \mathbf C, \mathbf C'' ) \simeq \prod_{v \in V_\tau} \bfMap_{T^{\amalg \val(v)} /\!/ T}( C_v, C''_v ) . \]
	Fix $v \in V_\tau$ and let $n \coloneqq \val(v)$.
	It is then enough to prove that the diagram
	\begin{equation} \label{eq:stabilization_truncation_IV}
		\begin{tikzcd}
			\Map_{T^{\amalg n} /\!/ T}( C'_v, C''_v ) \arrow{r} \arrow{d} & \Map_{T_0^{\amalg n} /\!/ T_0}( (C'_v)_0, (C''_v)_0 ) \arrow{d} \\
			\Map_{T^{\amalg n} /\!/ T}( C_v, C''_v ) \arrow{r} & \Map_{T_0^{\amalg n} /\!/ T_0}( (C_v)_0, (C''_v)_0 )
		\end{tikzcd}
	\end{equation}
	is a pullback.
	Consider the diagram
	\[ \begin{tikzcd}
		\Map_{T^{\amalg n} /\!/ T}( C', C'' ) \arrow{r} \arrow{d} & \Map_{T^{\amalg n} /\!/ T}( C, C'') \arrow{d} \arrow{r} & * \arrow{d} \\
		\Map_{/T}( C', C'') \arrow{r} & \Map_{/T}(C, C'') \arrow{r} & \Map_{/T}(T^{\amalg n }, C'') , 
	\end{tikzcd} \]
	where the right vertical arrow selects the sections of $C''_v$.
	The outer and the right squares are pullbacks by definition.
	It follows that the left square is also a pullback.
	Similarly,
	\[ \begin{tikzcd}
		\Map_{T_0^{\amalg n} /\!/ T_0}( (C_v')_0, (C''_v)_0) \arrow{r} \arrow{d} & \Map_{/T_0}( (C'_v)_0, (C_0'')_v) \arrow{d} \\
		\Map_{T_0^{\amalg n} /\!/ T_0}( (C_v)_0, (C''_v)_0 ) \arrow{r} & \Map_{/T_0}( (C_v)_0, (C_v'')_0 )
	\end{tikzcd} \]
	is a also pullback square.
	Therefore, in order to prove that diagram \eqref{eq:stabilization_truncation_IV} is a pullback it is enough to prove that
		\[ \begin{tikzcd}
		\Map_{/T}( C'_v, C_v'') \arrow{r} \arrow{d} & \Map_{/T_0}((C_v')_0, (C_v'')_0) \arrow{d} \\
		\Map_{/T}(C_v, C_v'') \arrow{r} & \Map_{/T_0}((C_v)_0, (C_v'')_0) 
	\end{tikzcd} \]
	is a pullback square.
	This follows from \cref{lem:derived_factorization_lemma}, completing the proof.
\end{proof}

We remark that it is enough to check assumption (1) of \cref{prop:checking_stabilization_on_truncation} at the level of truncations, by the following lemma:

\begin{lemma} \label{lem:checking_pushforward_on_truncation}
	Let $T \in \dAnk$, and let $p \colon C \to C'$ be a morphism in $\dAn_{/T}$.
	Assume that:
	\begin{enumerate}
		\item $p \colon C \to C'$ is proper and has finite coherent cohomological dimension;
		\item both $C$ and $C'$ are flat over $T$;
		\item let $p_0 \colon C_0 \to C'_0$ denote the truncation of $p$, the canonical map
		\[ \cO_{C'_0}\alg \longrightarrow p_{0*} \cO_{C_0}\alg \]
		is an equivalence.
	\end{enumerate}
	Then the canonical map
	\[ \cO_{C'}\alg \longrightarrow p_* \cO_C\alg \]
	is an equivalence as well.
\end{lemma}
\begin{proof}
	Since both $C$ and $C'$ are flat over $T$, the diagram
	\[ \begin{tikzcd}
	C_0 \arrow{r} \arrow{d} & C \arrow{d} \\
	C_0' \arrow{r}{j} & C'
	\end{tikzcd} \]
	is a pullback square.
	Since $p$ is proper and has finite coherent cohomological dimension, we see that $p_* \cO_C\alg$ belongs to $\Coh^+( C' )$.
	Therefore, the map $\alpha \colon \cO_{C'}\alg \to p_* \cO_C\alg$ is an equivalence if and only if $j^*(\alpha)$ is an equivalence.
	Since $p$ is proper, by derived base change (\cite[Theorem 1.5]{Porta_Yu_Derived_Hom_spaces})), $j^*(\alpha)$ is canonically identified with the map $\cO_{C'_0}\alg \to p_{0*} \cO_{C_0}\alg$.
	The conclusion follows.
\end{proof}

\subsubsection{Construction of stabilization} \label{sec:construction_of_stabilization}

Now let us construct the stabilizing contraction in \cref{def:stabilization} for any $\tau$-marked prestable map $\mathbf f \coloneqq ([C_v, (s_{v,e})_{e \in E_v}]_{v \in V_\tau}, f)$ into $X/S$ over $T$.
We first reason étale locally on $T$.
Let $t\in T$ be a geometric point and let $([(C_v)_t, (s_{v,e}(t))_{e \in E_v}]_{v \in V_\tau}, f_t)$ denote the associated stable map.
Up to replacing $T$ by a connected étale neighborhood of $t$, we can assume that for every geometric point $u = ([(C_v)_u, (s_{v,e}(u))_{e \in E_v}]_{v \in V_\tau}, f_u)\in T_0\coloneqq\trunc(T)$ and every $v\in V_\tau$, the following conditions are satisfied:
\begin{enumerate}
	\item $(C_v)_u$ is a partial smoothing of $(C_v)_t$, in the sense that the deformation from $(C_v)_t$ to $(C_v)_u$ does not create any new nodes.
	\item Let $D$ be an irreducible component of $(C_v)_u$ specializing to a union $E$ of irreducible components of $(C_v)_t$.
	If $f_t$ is not constant on one component of $E$, then $f_u$ is not constant on $D$.
\end{enumerate}
To every irreducible component of $(C_v)_t$ on which $f_t$ is not constant, we add three new marked points away from the nodes and the existing marked points.
In this way we obtain a new $\tau$-marked prestable map $([(C_v)_t, (s_{v,e}(t),p_{v,j}(t))]_{v \in V_\tau}, f_t)$.
By \cite[Proposition 5.50(3)]{Porta_Yu_Representability_theorem}, up to shrinking $T$, we can extend the marked points $p_{v,j}(t)$ to disjoint sections $p_{v,j}$ of $C_v \to T$, away from the nodes and the existing sections.
Using the equivalence $\oMpre_\tau \simeq \R \oMpre_\tau$, we apply stabilization for $\tau$-marked prestable curves (see \cite[0E8A]{Stacks_project}) to the family $[C_v, (s_{v,e},p_{v,j})]_{v \in V_\tau}$, and obtain a $\tau$-marked stable curve $[C'_v,(s'_{v,e},p'_{v,j})]_{v \in V_\tau}$ together with stabilization maps $c_v \colon C_v \to C'_v$.
This induces a stabilization map of glued curves $c\colon C\to C'$, and let $c_0\colon C_0\to C'_0$ be its truncation.
By construction, on every geometric fiber, the map $c_0$ contracts only irreducible components on which $f_0$ is constant.
So the map $f_0 \colon C_0 \to X$ factors as $C_0 \xrightarrow{c_0} C'_0\xrightarrow{f'_0} X$.
Moreover, since only rational components are contracted, the natural map
\[ \cO_{C_0'}\alg \longrightarrow c_{0*} \cO_{C_0}\alg \]
is an equivalence.
Using \cref{lem:checking_pushforward_on_truncation}, we deduce that the natural map
\begin{equation} \label{eq:pushforward_structures_sheaf}
	\cO_{C'}\alg \longrightarrow c_* \cO_C\alg
\end{equation}
is also an equivalence.
Therefore \cref{lem:derived_factorization_lemma} implies that the map $f_0' \colon C'_0 \to X$ can be lifted in a unique (up to a contractible space of choices) way to a map $f' \colon C' \to X$ making the diagram
\[ \begin{tikzcd}
	C \arrow{r}{f} \arrow{d}{c} & X \arrow{d} \\
	C' \arrow{ur}[swap]{f'} \arrow{r} & S
\end{tikzcd} \]
commutative.
Forgetting the new marked points, we obtain a $\tau$-marked prestable map
\[ \mathbf f' \coloneqq ( [C'_v, (s'_{v,e})_{e \in E_v}]_{v \in V_\tau}, f') \]
into $X/S$ over $T$, together with a morphism
\[ c \colon \mathbf f \longrightarrow \mathbf f' \]
We claim that $c$ is a stabilizing contraction.
Since the canonical map \eqref{eq:pushforward_structures_sheaf} is an equivalence, \cref{prop:checking_stabilization_on_truncation} implies that it is enough to show that the truncation $c_0 \colon \mathbf f_0 \to \mathbf f'_0$ is a stabilizing contraction.
This follows from the uniqueness statement in \cite[Tag 08EA]{Stacks_project}.

In order to glue this construction over $T$, it is enough to show that it is compatible with base change.
In other words, we have to prove that if $c \colon \mathbf f \to \mathbf f'$ is the stabilizing contraction over $T$ constructed above and $T' \to T$ is any morphism, then the induced morphism
\[ c \times_T T' \colon \mathbf f \times_T T' \longrightarrow \mathbf f' \times_T T' \]
is also a stabilizing contraction.
Using once again \cref{prop:checking_stabilization_on_truncation}, we are left to prove the same statement at the level of truncation, where it follows from \cite[Tag 08EB]{Stacks_project}.
At this point, we obtain the following proposition, the goal of \cref{sec:stabilization_map}.

\begin{proposition} \label{prop:stabilization_morphism}
	There is a morphism of derived \kanal stacks
	\[ \stab \colon \R \oMpre(X/S, \tau, \beta) \longrightarrow \R \oM(X/S, \tau, \beta) \]
	which sends every $\tau$-marked prestable map $\mathbf f$ over any $T \in \dAn_{/S}$ to the stabilization $\mathbf f'$ of $\mathbf f$ in the sense of \cref{def:stabilization}.
\end{proposition}
\begin{proof}
	This follows from the construction of stabilizing contraction and its compatibility with base change.
\end{proof}

\section{Geometry of derived stable maps} \label{sec:geometry_of_derived_stable_maps}

In this section, we prove a list of geometric relations between derived moduli stacks of \kanal stable maps with respect to elementary operations on $A$-graphs, namely, products, cutting edges, universal curve, forgetting tails and contracting edges.
These relations are very natural and intuitive, and help replace the study of compatibilities of perfect obstruction theories in the classical approach (cf.\ \cite{Behrend_Gromov-Witten_invariants,Lee_Quantum_K-theory_I}).
They will give rise to the geometric properties of quantum K-invariants in \cref{sec:quantum_K-invariants}.

We fix a rigid \kanal space $S$, and a rigid \kanal space $X$ smooth over $S$.
For any $A$-graph $(\tau,\beta)$, let $\R\oM(X/S,\tau,\beta)$ denote the derived stack of $(\tau,\beta)$-marked stable maps into $X/S$ as in \cref{rem:RMXtaubeta}.
We prove in the following a list of geometric properties of $\R\oM(X/S,\tau,\beta)$ with respect to operations on $(\tau,\beta)$.

\subsection{Products} \label{sec:products}

Let $(\tau_1, \beta_1)$ and $(\tau_2, \beta_2)$ be two A-graphs.
Let $(\tau_1 \sqcup \tau_2, \beta_1 \sqcup \beta_2)$ be their disjoint union.

\begin{theorem} \label{thm:products}
	We have a canonical equivalence
	\[\R\oM(X/S, \tau_1 \sqcup \tau_2, \beta_1 \sqcup \beta_2) \xrightarrow{\ \sim\ } \R \oM(X/S, \tau_1, \beta_1) \times_S  \R \oM(X/S, \tau_2, \beta_2),\]
	where the projections are given by the natural forgetful maps.
\end{theorem}

\begin{proof}
	The forgetful maps
	\[ p_i \colon \oMpre_{\tau_1 \sqcup \tau_2} \longrightarrow \oMpre_{\tau_i} , \qquad i = 1,2 \]
	exhibit $\oMpre_{\tau_1 \sqcup \tau_2}$ as the product of $\oMpre_{\tau_1}$ and $\oMpre_{\tau_2}$.
	Let $p_i^* \oCpre_{\tau_i}$ be the pullback of the universal curve $\oCpre_{\tau_i} \to \oMpre_{\tau_i}$ along $p_i$.
	Then we have a canonical equivalence
	\[ \oCpre_{\tau_1 \sqcup \tau_2} \simeq p_1^* \oCpre_{\tau_1} \amalg p_2^* \oCpre_{\tau_2}.\]
	Note that the disjoint union on the right is also a disjoint union of (representable) stacks.
	So $\bfMap_{\oMpre_{\tau_1 \sqcup \tau_2}\times S}( \oCpre_{\tau_1 \sqcup \tau_2}\times S , X \times \oMpre_{\tau_1 \sqcup \tau_2} )$ is canonically equivalent to
	\[ \bfMap_{\oMpre_{\tau_1 \sqcup \tau_2}\times S}( p_1^* \oCpre_{\tau_1}\times S, X \times \oMpre_{\tau_1 \sqcup \tau_2} ) \times_{\oMpre_{\tau_1 \sqcup \tau_2}\times S} \bfMap_{\oMpre_{\tau_1 \sqcup \tau_2}\times S}( p_2^* \oCpre_{\tau_2}\times S, X \times \oMpre_{\tau_1\sqcup \tau_2} ) . \]
	As $\oMpre_{\tau_1 \sqcup \tau_2} \simeq \oMpre_{\tau_1} \times \oMpre_{\tau_2}$, we have
	\[ \bfMap_{\oMpre_{\tau_1 \sqcup \tau_2}\times S}( p_1^*\oCpre_{\tau_1}\times S, X \times \oMpre_{\tau_1 \sqcup \tau_2} ) \simeq \bfMap_{\oMpre_{\tau_1}\times S}( \oCpre_{\tau_1}\times S, X \times \oMpre_{\tau_1}) \times \oMpre_{\tau_2},\]
	and similarly
	\[ \bfMap_{\oMpre_{\tau_1 \sqcup \tau_2}\times S}( p_2^*\oCpre_{\tau_2}\times S, X \times \oMpre_{\tau_1 \sqcup \tau_2} ) \simeq \bfMap_{\oMpre_{\tau_2}\times S}( \oCpre_{\tau_2}\times S, X \times \oMpre_{\tau_2}) \times \oMpre_{\tau_1}.\]
	Therefore, $\bfMap_{\oMpre_{\tau_1 \sqcup \tau_2}\times S}( \oCpre_{\tau_1 \sqcup \tau_2}\times S, X \times \oMpre_{\tau_1 \sqcup \tau_2} )$ becomes canonically equivalent to
	\[ \bfMap_{\oMpre_{\tau_1}\times S}( \oCpre_{\tau_1}\times S, X \times \oMpre_{\tau_1} ) \times_S \bfMap_{\oMpre_{\tau_2}\times S}(\oCpre_{\tau_2}\times S, X \times \oMpre_{\tau_2}) . \]
		Note that on the level of truncations, the inclusions
	\begin{gather*}
	\oM(X/S,\tau_1,\beta_1)\longhookrightarrow\trunc\bfMap_{\oMpre_{\tau_1}\times S}(\oCpre_{\tau_1}\times S, X\times\oMpre_{\tau_1}),\\
	\oM(X/S,\tau_2,\beta_2)\longhookrightarrow\trunc\bfMap_{\oMpre_{\tau_2}\times S}(\oCpre_{\tau_2}\times S, X\times\oMpre_{\tau_2}),\\
	\oM(X/S,\tau_1\sqcup\tau_2,\beta_1\sqcup\tau_2)\longhookrightarrow\trunc\bfMap_{\oMpre_{\tau_1\sqcup\tau_2}\times S}(\oCpre_{\tau_1\sqcup\tau_2}\times S, X\times\oMpre_{\tau_1\sqcup\tau_2})
	\end{gather*}
	are all open, and we have a canonical isomorphism
	\[ \oM(X/S, \tau_1 \sqcup \tau_2, \beta_1 \sqcup \beta_2) \simeq \oM(X/S, \tau_1, \beta_1) \times_S \oM(X/S, \tau_2, \beta_2). \]
	Therefore, the conclusion follows from the equivalence of sites of \cite[Lemma 7.16]{Porta_Yu_Derived_non-archimedean_analytic_spaces}.
\end{proof}

\subsection{Cutting edges} \label{sec:cutting_edges}

Let $(\sigma,\beta)$ be an A-graph obtained from $(\tau,\beta)$ by cutting an edge $e$ of $\tau$ (see \cref{fig:cutting_edges}).
\begin{figure}[!ht]
	\centering
	\setlength{\unitlength}{0.5\textwidth}
	\begin{picture} (1,0.543)
		\put(0,0){\includegraphics[width=\unitlength, page=2]{images/modular_graph}}
		\put(-0.1,0.42){$\tau:$}
		\put(-0.1,0.10){$\sigma:$}
		\put(0.4,0.44){$e$}
		\put(0.27,0.13){$i$}
		\put(0.62,0.13){$j$}
	\end{picture}
	\caption{We cut the edge $e$, and replace the two half-edges by two tails.
	All tails are colored in blue.}
	\label{fig:cutting_edges}
\end{figure}
Let $i,j$ be the two tails of $\sigma$ created by the cut.
The edge $e$ gives a section $s_e$ of the universal curve $\oCpre_\tau\to\oMpre_\tau$.
Each geometric fiber of $s_e$ gives the node corresponding to $e$.

\begin{theorem} \label{thm:cutting_edges_derived_pullback}
	We have a derived pullback diagram
	\[ \begin{tikzcd}
		\R \oM(X/S, \tau, \beta) \arrow{r}{c} \arrow{d}{\ev_e} & \R \oM(X/S, \sigma, \beta) \arrow{d}{\ev_i \times \ev_j} \\
		X \arrow{r}{\Delta} & X \times_S X
	\end{tikzcd} \]
	where $\Delta$ is the diagonal map, $\ev_e$ is evaluation at the section $s_e$, and $c$ is induced by cutting the domain curves at $s_e$.
\end{theorem}
\begin{proof}[Proof of \cref{thm:cutting_edges_derived_pullback}]
	In this case, by \cref{rem:product_decomposition_M_tau}, we have naturally $\oMpre_\tau\simeq\oMpre_\sigma$, which we denote by $\oMpre$ for simplicity.
	
	Let $s_i, s_j \colon \oMpre \to \oCpre_\sigma$ denote the two sections corresponding to the tails $i$ and $j$.
	By \cref{rem:product_decomposition_M_tau}, we have a derived pushout square
	\[ \begin{tikzcd}
		\oMpre \amalg \oMpre \arrow{r}{s_i\sqcup s_j} \arrow{d} & \oCpre_\sigma \arrow{d} \\
		\oMpre \arrow{r}{s_e} & \oCpre_\tau .
	\end{tikzcd} \]
	Then for any $T \in \dAn_{/S}$ and $T \to \oMpre$, by \cref{lem:gluing_along_closed_immersion_is_universal}, the diagram above remains a pushout after base change to $T$.
	Therefore, we deduce that
	\begin{equation} \label{eq:cutting_edges_II}
	\begin{tikzcd}
	\bfMap_{\oMpre\times S}( \oCpre_\tau\times S , X \times \oMpre ) \arrow{r} \arrow{d} & \bfMap_{\oMpre\times S}( \oCpre_\sigma\times S, X \times \oMpre ) \arrow{d}{\ev_i \times \ev_j} \\
	X \arrow{r}{\Delta} & X \times_S X
	\end{tikzcd}
	\end{equation}
	is a derived pullback square.
	
	In order to complete the proof, it is enough to prove that the diagram
	\[ \begin{tikzcd}
	\R \oM(X/S, \tau, \beta) \arrow{r} \arrow{d} & \R\oM(X/S, \sigma, \beta) \arrow{d} \\
	\bfMap_{\oMpre}(\oCpre_\tau, X \times \oMpre) \arrow{r} & \bfMap_{\oMpre}( \oCpre_\sigma, X \times \oMpre )
	\end{tikzcd} \]
	is a derived pullback square as well.
	By \cref{rem:RMXtaubeta}, the two vertical maps are Zariski open immersions, so it is enough to check that the diagram above is a pullback after passing to the truncations.
	At the truncated level, it follows directly from the definitions.
\end{proof}

\subsection{Universal curve} \label{sec:universal_curve}

Let $(\sigma,\beta)$ be an A-graph obtained from $(\tau,\beta)$ by forgetting a tail.
By forgetting the marked point associated to the tail, we obtain a map of derived \kanal stacks
\[ \R \oMpre(X/S, \tau, \beta) \longrightarrow \R \oMpre(X/S, \sigma, \beta) . \]
Let $\pi \colon \R \oM(X/S, \tau, \beta) \to \R \oM(X/S, \sigma, \beta)$ be the composition
\[ \begin{tikzcd}
	\R \oM(X/S, \tau, \beta) \to \R \oMpre(X/S, \tau, \beta) \to \R \oMpre(X/S, \sigma, \beta) \xrightarrow{\stab} \R \oM(X/S, \sigma, \beta) ,
\end{tikzcd} \]
where $\stab$ is the stabilization map in \cref{prop:stabilization_morphism}.

Assume that the forgotten tail was attached to the vertex $w$ of $\sigma$.
Let $\oCpre_w \to \oMpre_\sigma$ be the universal curve corresponding to the vertex $w$, in other words, $\oCpre_w$ is defined by the following pullback diagram:
\[ \begin{tikzcd}
	\oCpre_w \arrow{r} \arrow{d} & \oMpre_\sigma \arrow{d} \\
	\oCpre_{g(w), \val(w)} \arrow{r} & \oMpre_{g(w), \val(w)} .
\end{tikzcd} \]
For any derived $k$-analytic stack $F$ and any map $f\colon  F \to \oMpre_\sigma$, we have a fiber sequence
\[ \Map\big( F, F \times_{\oMpre_{g(w), \val(w)}} \oCpre_{g(w), \val(w)} \big) \longrightarrow \Map( F , \oCpre_w ) \longrightarrow \Map( F, \oMpre_\sigma) . \]
In other words, the space of maps $F \to \oCpre_w$ over $\oMpre_\sigma$ is equivalent to the space of sections of the prestable curve over $F$ corresponding to the vertex $w$.
Taking $F = \R \oM(X/S, \tau, \beta)$ and the map $f$ to be the composition
\[ \R \oM(X/S, \tau, \beta) \xrightarrow{\ \pi\ } \R \oM(X/S, \sigma, \beta) \longrightarrow \oMpre_\sigma , \]
and using the section associated to the forgotten tail in $\tau$, we obtain a canonical map
\[ \lambda\colon\R \oM(X/S, \tau, \beta) \longrightarrow \oCpre_w . \]
We can now state the main result of this subsection:

\begin{theorem} \label{thm:universal_tau_marked_curve}
	The diagram
	\[ \begin{tikzcd}
		\R \oM( X/S, \tau, \beta ) \arrow{r}{\pi} \arrow{d}{\lambda} & \R \oM( X/S, \sigma, \beta ) \arrow{d} \\
		\oCpre_w \arrow{r} & \oMpre_\sigma
	\end{tikzcd} \]
	is a derived pullback square.
\end{theorem}

\begin{proof}
		The diagram commutes by construction.
	Write
	\[ \R \oC_w(X/S, \sigma, \beta) \coloneqq \oCpre_w \times_{\oMpre_\sigma} \R \oM(X/S, \sigma, \beta) . \]
	We need to prove that the natural map
	\[ \R \oM( X/S, \tau, \beta ) \longrightarrow \R \oC_w(X/S, \sigma, \beta)\]
	is an equivalence.
	Let $T \in \dAn_{/S}$ be any derived \kanal space over $S$.
	We first show that the functor of $\infty$-groupoids
	\begin{equation} \label{eq:universal_curve_functor_of_points}
		\R \oM(X/S, \tau, \beta)(T) \longrightarrow \R \oC_w(X/S, \sigma, \beta)(T)
	\end{equation}
	is fully faithful.
	Let $\mathbf f = ([C_v, (s_{v,e})]_{v \in V_\tau}, f)$ and $\tbf = ([\tC_v, (\ts_{v,e})]_{v \in V_\tau}, \tf)$ be two $\tau$-marked stable maps into $X/S$ over $T$.
	Let $\mathbf f' = ([C_v', (s'_{v,e})]_{v \in V_\sigma}, f')$ and $\tbf' = ([\tC'_v, (\ts'_{v,e})]_{v \in V_\sigma}, \tf')$ be their images via $\pi$.
	Let $s'\colon T\to C'_v$ and $\ts'\colon T\to\tC'_v$ be the images of the extra sections corresponding to the forgotten tail.
	So the pairs $(\mathbf f', s')$ and $(\tbf', \ts')$ are the images of $\mathbf f$ and $\tbf$ in $\R \oC_w(X/S, \sigma, \beta)$ under the map \eqref{eq:universal_curve_functor_of_points}.
	Given any $\alpha \colon (\mathbf f', s') \to (\tbf', \ts')$, denote by $\Map_{/T}^\alpha( \mathbf f, \tbf )$ the fiber of
	\[ \Map_{/T}(\mathbf f, \tbf) \longrightarrow \Map_{/T}((\mathbf f', s'), (\tbf', \ts')) \]
	at $\alpha$.
	In order to prove the full faithfulness of \eqref{eq:universal_curve_functor_of_points}, it is enough to prove that for every equivalence $\alpha \colon (\mathbf f', s') \to (\tbf', \ts')$, the space $\Map_{/T}^\alpha( \mathbf f, \tbf')$ is contractible.
	This question is local on $T$, so we are free to replace $T$ by an étale covering.
	Let $U_v'$ be the complement in $C_v'$ of all the sections and all the nodes, and let
	\[ U_v \coloneqq C_v \times_{C_v'} U_v' . \]
	Define similarly $\tU_v'$ and $\tU_v$.
	Note that the equivalence $\alpha$ restricts to an equivalence $U_v' \simeq \tU_v'$.
	Moreover, the canonical morphisms $U_v \to U_v'$ and $\tU_v \to \tU_v'$ are also equivalences.
	It follows that $\alpha$ induces an equivalence $\beta \colon U_v \simeq \tU_v$.
	Up to replacing $T$ by an étale covering, we can add extra sections $(p_{v,j}')$ to $U_v'$ in such that $\mathbf C^{\prime \ex} \coloneqq [C_v', (s'_{v,e}, p_{v,j}')]_{v \in V_\tau}$ is a stable pointed curve.
	Via the equivalence $U_v \simeq U_v'$, we lift the sections $p_{v,j}'$ of $U_v'$ to sections $p_{v,j}$ of $U_v$.
	Write
	\[ \mathbf C^\ex \coloneqq [C_v, (s_{v,e}, p_{v,j})]_{v \in V_\tau} \]
	and
	\[ \mathbf f^\ex \coloneqq ( \mathbf C^\ex, f ) , \qquad \mathbf f^{\prime \ex} \coloneqq ( \mathbf C^{\prime \ex}, f' ) . \]
	Then the map $\mathbf f \to \mathbf f'$ induces a map $\mathbf f^\ex \to \mathbf f^{\prime \ex}$, and the underlying map $\mathbf C^\ex \to \mathbf C^{\prime \ex}$
	is a stabilizing contraction of pointed curves (see \cref{sec:construction_of_stabilization}).
	Using the equivalences $\alpha$ and $\beta$, we add compatible sections $\tp_{v,j}$ and $\tp'_{v,j}$ to $\tU_v$ and $\tU'_v$.
	We define similar notations $\tbf^\ex$, $\tbf^{\prime \ex}$, $\widetilde{\mathbf C}^\ex$, $\widetilde{\mathbf C}^{\prime \ex}$.
	As above, the map $\tbf\to\tbf'$ induces a map $\tbf^\ex\to\tbf'^\ex$, and the underlying map $\tbC^\ex\to\tbC'^\ex$ is a stabilizing contraction of pointed curves.
	Write
	\begin{align*}
	&\mathbf f_U \coloneqq ( [U_v]_{v \in V_\tau}, f ), & \mathbf f^\ex_U \coloneqq ([U_v, (p_{v,j})]_{v \in V_\tau}, f) , \\
	&\tbf_U \coloneqq ( [\tU_v]_{v \in V_\tau},\tf ), & \tbf^\ex_U \coloneqq ([\tU_v, (p_{v,j})]_{v \in V_\tau},\tf).
	\end{align*}
	Let $W_v$ denote the complement in $C_v$ of the sections $p_{v,j}$, and define similarly $\tW_v$.
	Set
	\begin{align*}
	\mathbf f_W \coloneqq \mathbf f^\ex_W \coloneqq ( [W_v, (s_{v,e})]_{v \in V_\tau}, f ),\\
	\tbf_W \coloneqq \tbf^\ex_W \coloneqq ( [\tW_v, (\ts_{v,e})]_{v \in V_\tau}, \tf ).
	\end{align*}
	We also introduce the notations $\mathbf f_{U\cap W}$, $\mathbf f^\ex_{U\cap W}$, $\tbf_{U\cap W}$ and $\tbf^\ex_{U\cap W}$ in the same way as above.
	
	Observe that the diagrams
	\[ \begin{tikzcd}
		\Map_{/T}^\alpha( \mathbf f, \tbf ) \arrow{r} \arrow{d} & \Map_{/T}^\alpha( \mathbf f_W, \tbf_W ) \arrow{d} \\
		\Map_{/T}^\alpha( \mathbf f_U, \tbf_U ) \arrow{r} & \Map_{/T}^\alpha( \mathbf f_{U \cap W}, \tbf_{U \cap W} )
	\end{tikzcd} \]
	and
	\[ \begin{tikzcd}
		\Map_{/T}^\alpha( \mathbf f^\ex, \tbf^\ex ) \arrow{r} \arrow{d} & \Map_{/T}^\alpha( \mathbf f^\ex_W, \tbf^\ex_W ) \arrow{d} \\
		\Map_{/T}^\alpha( \mathbf f^\ex_U, \tbf^\ex_U ) \arrow{r} & \Map_{/T}^\alpha( \mathbf f^\ex_{U \cap W}, \tbf^\ex_{U \cap W} )
	\end{tikzcd} \]
	are pullback squares.
	Unraveling the definitions, we see that the natural maps
	\begin{gather*}
		\Map_{/T}^\alpha( \mathbf f^\ex_U, \tbf^\ex_U ) \longrightarrow \Map_{/T}^\alpha( \mathbf f_U, \tbf_U ) \\
		\Map_{/T}^\alpha( \mathbf f^\ex_W, \tbf^\ex_W ) \longrightarrow \Map_{/T}^\alpha( \mathbf f_W, \tbf_W ) \\
		\Map_{/T}^\alpha( \mathbf f^\ex_{U \cap W}, \tbf^\ex_{U \cap W} ) \longrightarrow \Map_{/T}^\alpha( \mathbf f_{U \cap W}, \tbf_{U \cap W} )
	\end{gather*}
	are equivalences.
	It follows that the induced map
	\[ \Map_{/T}^\alpha( \mathbf f^\ex, \tbf^\ex ) \longrightarrow \Map_{/T}^\alpha( \mathbf f, \tbf )  \]
	is an equivalence as well.
	By \cite[Corollary 2.3.2.5]{Lurie_HTT}, the forgetful map
	\[ \Map_{/T}^\alpha( \mathbf f^\ex, \tbf^\ex ) \longrightarrow \Map_{/T}^\alpha( \mathbf C^\ex, \widetilde{\mathbf C}^\ex ) \]
	is an equivalence.
		As $\mathbf C^\ex$ and $\tbC^\ex$ are stable pointed curves, it follows from \cite[Corollary 2.6]{Knudsen_The_projectivity_II} that $\Map_{/T}^{\alpha}( \mathbf C^\ex, \widetilde{\mathbf C}^\ex )$ is contractible.
		We deduce that $\Map_{/T}^\alpha( \mathbf f, \mathbf f' )$ is contractible, and conclude the full faithfulness of the functor  of $\infty$-groupoids \eqref{eq:universal_curve_functor_of_points}.
	
	Now it remains to prove that the functor \eqref{eq:universal_curve_functor_of_points} is also essentially surjective.
	Since this functor is fully faithful, we can reason locally on $T$.
	At this point, using the equivalence $\oM_\tau \simeq \R \oM_\tau$, the same proof of \cite[Proposition 4.5]{Behrend_Stacks_of_stable_maps} applies.
\end{proof}

\begin{corollary} \label{cor:universal_curve_forgetting_tail}
	The forgetful map
	\[ \R \oM_{g,n+1}(X/S) \longrightarrow \R \oM_{g,n}(X/S) \]
	canonically factors through the universal curve $\R \oC_{g,n}(X/S)$, and the resulting morphism
	\[ \R \oM_{g,n+1}(X/S) \longrightarrow \R \oC_{g,n}(X/S) \]
	is an equivalence.
\end{corollary}

\begin{proof}
	By definition, $\R \oC_{g,n}(X/S)$ fits in the pullback diagram
	\[ \begin{tikzcd}
		\R \oC_{g,n}(X/S) \arrow{r} \arrow{d} & \R \oM_{g,n}(X/S) \arrow{d} \\
		\oCpre_{g,n} \arrow{r} & \oMpre_{g,n} .
	\end{tikzcd} \]
	So the corollary follows directly from \cref{thm:universal_tau_marked_curve}.
\end{proof}

\subsection{Forgetting tails} \label{sec:forgetting_tails}

Let $(\sigma,\beta)$ be an A-graph obtained from $(\tau,\beta)$ by forgetting a tail attached to a vertex $w$ of $\tau$.
We follow the notations of \cref{sec:universal_curve}.
Consider the commutative diagram
\[ \begin{tikzcd}
	\R \oM(X/S, \tau, \beta) \arrow{r}{\pi} \arrow{d} & \R \oM(X/S, \sigma, \beta) \arrow{d} \\
	\oM_\tau \arrow{r}{\Phi} & \oM_\sigma ,
\end{tikzcd} \]
which induces a map
\[\Psi \colon \R \oM(X/S, \tau, \beta) \longrightarrow \oM_\tau \times_{\oM_\sigma} \R \oM(X/S, \sigma, \beta) .  \]

\begin{theorem} \label{thm:forgetting_tails}
	The diagram
	\[ \begin{tikzcd}
	\R \oM(X/S, \tau, \beta) \arrow{r}{\Psi} \arrow{d}{\lambda} & \oM_\tau \times_{\oM_\sigma}  \R \oM(X/S, \sigma, \beta) \arrow{d} \\
	\oCpre_w \arrow{r}{\rho} & \oM_\tau \times_{\oM_\sigma} \oMpre_\sigma
	\end{tikzcd} \]
	is a derived pullback square.
	Moreover, the canonical map
	\[\cO\alg_{\oM_\tau \times_{\oM_\sigma}  \R \oM(X/S, \sigma, \beta)}\longrightarrow\Psi_*\big(\cO\alg_{\R\oM(X/S,\tau,\beta)}\big)\]
	is an equivalence.
\end{theorem}
\begin{proof}
	Consider the diagram
	\[ \begin{tikzcd}
		\R \oM(X/S, \tau, \beta) \arrow{r}{\Psi} \arrow{d}{\lambda} & \oM_\tau \times_{\oM_\sigma}  \R \oM(X/S, \sigma, \beta) \arrow{r} \arrow{d}{\mu} & \R \oM(X/S, \sigma, \beta) \arrow{d} \\
		\oCpre_w \arrow{r}{\rho} & \oM_\tau \times_{\oM_\sigma} \oMpre_\sigma \arrow{d} \arrow{r} & \oMpre_\sigma \arrow{d} \\ {} & \oM_\tau \arrow{r}{\Phi} & \oM_\sigma .
	\end{tikzcd} \]
	The right outer rectangle and the right lower square are pullbacks by definition.
	Hence the right upper square is also a pullback.
	The top outer rectangle is a pullback square in virtue of \cref{thm:universal_tau_marked_curve}.
	It follows that the left upper square is a pullback, which shows the first statement in the proposition.
	
	Now by derived base-change (\cite[Theorem 6.8]{Porta_Yu_Derived_Hom_spaces}), the canonical map
	\[ \mu^*\rho_*\big( \cO\alg_{\oCpre_w} \big) \longrightarrow\Psi_*\lambda^* \big( \cO\alg_{\oCpre_w} \big) \]
	is an equivalence.
	Moreover, since $\rho$ is a blow-down of rational curves (fiberwise over $\oMpre_\sigma$),
	we have
	\[\cO\alg_{\oM_\tau\times_{\oM_\sigma}\oMpre_\sigma}\xrightarrow{\ \sim\ }\rho_* \big( \cO\alg_{\oCpre_w} \big).\]
	Since pullbacks of structure sheaves are always structure sheaves, we conclude that the canonical map
	\[ \cO\alg_{\oM_\tau \times_{\oM_\sigma}  \R \oM(X/S, \sigma, \beta)}\longrightarrow\Psi_* \big(\cO\alg_{\R\oM(X/S,\tau,\beta)} \big) \]
	is an equivalence, completing the proof.
\end{proof}

\subsection{Contracting edges} \label{sec:contracting_edges}

Let $\sigma$ be a modular graph obtained from $\tau$ by contracting an edge (possibly a loop) $e$, and $v_e \in V_\sigma$ the vertex of $\sigma$ corresponding to the contracted edge $e$.
Let $\tsigma$ be a modular graph obtained from $\sigma$ by adding $k$ tails, and $k_e$ the number of extra tails attached to $v_e$.

The main result here is \cref{thm:contracting_edges}.
In order to rigorously construct the colimit diagram in \eqref{eq:contracting_edge}, we start with the following:

\begin{construction} \label{const:contracting_edges_modular_graph}
	For each integer $l \ge 1$, let $\tau_l$ denote the modular graph obtained from $\tau$ by replacing $e$ with a chain $\gamma_l$ of $l$ edges and $l+1$ vertices such that all the $(l-1)$ new vertices have genus $0$.
	Let $c_l \colon \tau_l \to \sigma$ be the map contracting $\gamma_l$.
		For every map $i \colon \{1, \ldots, k_e\} \to V_{\gamma_l}$,
	let $\tau_l^i$ be the graph obtained from $\tau_l$ by adding $k$ extra tails in the unique way compatible with $i$, $c_l$ and $\tsigma$ (see \cref{fig:contracting_edges}).
	\begin{figure}[!ht]
	\centering
	\setlength{\unitlength}{0.7\textwidth}
	\begin{picture} (1,0.523)
	\put(0,0){\includegraphics[width=\unitlength, page=3]{images/modular_graph}}
	\put(0.15,0.47){$e$}
	\end{picture}
	\caption{The modular graphs $\tau$, $\tau_l$, $\sigma$, $\tau_l^i$ and $\tsigma$ respectively, where $l=2$, $k_e=2$.
		All tails are colored in blue.}
	\label{fig:contracting_edges}
	\end{figure}
	Let $c_l^i \colon \tau_l^i \to \sigma$ be the composition of forgetting tails $\tau_l^i \to \tau_l$ with the contraction $c_l \colon \tau_l \to \sigma$.
	Let $\MG$ be the category of modular graphs whose morphisms are compositions of forgetting tails and contractions.
	Let $\cD$ be the full subcategory of $\MG_{/\sigma}$ spanned by the morphisms $c_l^i \colon \tau_l^i \to \sigma$.
	Define
	\[ \Lambda \coloneqq \Lambda_{\tau \to \sigma, \tsigma} \coloneqq \cD_{/\tau \to \sigma} . \]
		An object in $\Lambda$ consists of an integer $l \ge 1$, a function $i \colon \{1, \ldots, k_e\} \to V_{\gamma_l}$ and a map $a \colon \tau_l^i \to \tau$ in $\MG$ making the diagram
	\[ \begin{tikzcd}
		\tau_l^i \arrow{d}[swap]{a} \arrow{dr}{c_l^i} \\
		\tau \arrow{r} & \sigma
	\end{tikzcd} \]
	commutative.
	Note that the commutative diagram above can be extended in a unique way to
	\[ \begin{tikzcd}
		\tau_l^i \arrow{d}[swap]{a} \arrow{r} & \tsigma \arrow{d} \\
		\tau \arrow{r} & \sigma .
	\end{tikzcd} \]
	\end{construction}

\begin{construction} \label{const:contracting_edges_colimit_domain}
	We functorially attach to every object $(a \colon \tau_l^i \to \tau)$ in $\Lambda$ the natural commutative diagram
	\begin{equation} \label{eq:trick_avoiding_full_functoriality}
	\begin{tikzcd}
	\oMpre_{\tau_l^i} \arrow{r} \arrow{d} & \oMpre_{\tsigma} \arrow{d} \\
	\oMpre_\tau \arrow{r} & \oMpre_\sigma .
	\end{tikzcd}
	\end{equation}
		Combining the universal properties of stabilization (see \cref{def:stabilization}) and of gluing along closed immersions, the diagram
	\begin{equation} \label{eq:commutativity_between_stabilization_and_contracting_edges}
	\begin{tikzcd}
	\oMpre_\tau \arrow{r} \arrow{d} & \oMpre_\sigma \arrow{d} \\
	\oM_\tau \arrow{r} & \oM_\sigma
	\end{tikzcd}
	\end{equation}
	canonically commutes.
	Therefore we obtain a canonical map
	\begin{equation} \label{eq:contracting_edges_colimit_domain}
	\colim_{(a\colon\tau_l^i\to\tau) \in \Lambda} \oMpre_{\tau_l^i} \longrightarrow \oM_\tau \times_{\oM_\sigma} \oMpre_{\tsigma}.
	\end{equation}
\end{construction}

\begin{lemma} \label{lem:contracting_edges_colim}
	The canonical map \eqref{eq:contracting_edges_colimit_domain} is an equivalence in $\dAnk$.
\end{lemma}

We will prove the following lemma first:

\begin{lemma} \label{lem:stab_is_smooth}
	For any modular graph $\tau$, the stabilization map
	\[ \stab \colon \oMpre_\tau \longrightarrow \oM_\tau \]
	is smooth.
\end{lemma}
\begin{proof}
	Since both $\oMpre_\tau$ and $\oM_\tau$ split as products as in \cref{rem:product_decomposition_M_tau} and the stabilization map preserves the product, it is enough to prove the lemma in the case when $\tau$ has a single non-tail vertex.
	Fix $T \in \dAfdk$ and $x \colon T \to \oMpre_\tau$.
	Let $[p \colon C \to T,(s_i)]$ be the associated prestable pointed curve over $T$.
	Let $y \coloneqq \stab \circ x \colon T \to \oM_\tau$, $[q\colon C'\to T, (s'_i)]$ the associated prestable pointed curve, and $c\colon C\to C'$ the stabilizing contraction.
	Then we have a fiber sequence
	\begin{equation} \label{eq:tangent_stabilization_map}
	x^* \mathbb T\an_{\stab} \longrightarrow x^* \mathbb T\an_{\oMpre_\tau} \longrightarrow y^* \mathbb T\an_{\oM_\tau} .
	\end{equation}
	To prove that the map $\stab$ is smooth, it is enough to show that $\pi_i( x^* \mathbb T\an_{\stab} ) \simeq 0$ for every $i < 0$.
	By \eqref{eq:cotangent_complex_Mgn}, we have
	\[ x^* \mathbb T\an_{\oMpre_\tau} \simeq p_*( \mathbb T\an_{C / T}( - {\textstyle\sum} s_i ) )[1] \quad \text{and} \quad y^* \mathbb T\an_{\oM_\tau} \simeq q_*( \mathbb T\an_{C' / T}(- {\textstyle\sum} s_i') )[1] . \]
	By \cref{lem:checking_pushforward_on_truncation}, we have a canonical equivalence $c_*( \cO_C\alg ) \simeq \cO\alg_{C'}$.
	So by the projection formula, the unit transformation
	\[ \mathbb T\an_{C'/T} \longrightarrow c_* c^* \mathbb T\an_{C'/T} \]
	is also an equivalence.
	In particular, the canonical map $\phi \colon \mathbb T\an_{C/T} \to c^* \mathbb T\an_{C'/T}$ induces a morphism
	\[ \alpha \colon c_* \mathbb T\an_{C/T} \longrightarrow \mathbb T\an_{C'/T} . \]
	On the other hand, $\phi$ also induces a morphism
	\[ \beta \colon \bigoplus c_* s_{i*}s_i^* \mathbb T\an_{C/T} \longrightarrow \bigoplus s'_{i*} s_i^{\prime *} \mathbb T\an_{C'/T} . \]
	The morphisms $\alpha$ and $\beta$ fit in the following commutative diagram
	\[ \begin{tikzcd}
	c_* \mathbb T\an_{C/T} \arrow{r} \arrow{d}{\alpha} & \bigoplus c_* s_{i*} s_i^* \mathbb T\an_{C/T} \arrow{d}{\beta} \\
	\mathbb T\an_{C'/T} \arrow{r} & \bigoplus s'_{i*} s_*^{\prime *} \mathbb T\an_{C'/T}.
	\end{tikzcd} \]
		We complete the diagram as follows:
	\[ \begin{tikzcd}
	\fib(\gamma) \arrow{r} \arrow{d} & \fib(\alpha) \arrow{r} \arrow{d} & \fib(\beta) \arrow{d} \\
	c_* \mathbb T\an_{C/T}(- {\textstyle\sum} s_i) \arrow{r} \arrow{d}{\gamma} & c_* \mathbb T\an_{C/T} \arrow{d}{\alpha} \arrow{r} & \bigoplus c_* s_{i*} s_i^* \mathbb T\an_{C/T} \arrow{d}{\beta} \\
	\mathbb T_{C'/T}\an( - {\textstyle\sum} s_i' ) \arrow{r} & \mathbb T\an_{C'/T} \arrow{r} & \bigoplus s'_{i*} s_i^{\prime*} \mathbb T\an_{C'/T}.
	\end{tikzcd} \]
	Comparing the fiber sequence \eqref{eq:tangent_stabilization_map} with the leftmost column of the above diagram, we obtain a canonical identification
	\[ q_*( \fib(\gamma) )[1] \simeq x^* \mathbb T\an_{\stab} . \]
	Therefore, in order to prove that $\pi_i( x^* \mathbb T\an_{\stab} ) \simeq 0$ for $i < 0$, it suffices to prove that
	\[ \pi_i( q_*( \fib(\gamma))) = 0, \quad i \le -2 . \]
	
	The fiber sequence \eqref{eq:tangent_stabilization_map} immediately implies that this vanishing holds for $i \le -3$.
	We can therefore check that $\pi_{-2}( q_*( \fib(\gamma) ) ) = 0$ on the geometric points of $T$.
	In other words, we are left to prove the vanishing when $T$ is the spectrum of an algebraically closed field.
	We now make the following observations:
	\begin{enumerate}
		\item The sheaf $\fib(\alpha)$ is supported on finitely many points of $C'$.
		Indeed, applying $c_*$ to the fiber sequence $\mathbb T\an_{C/C'} \to \mathbb T\an_{C} \to c^* \mathbb T\an_{C'}$, we obtain an equivalence $\fib(\alpha) \simeq c_* \mathbb T\an_{C/C'}$, and the right hand side is supported on finitely many points of $C'$.
		\item The sheaf $\fib(\beta)$ is supported on finitely many points of $C'$.
			\end{enumerate}
	These two observations imply that
	\[ \fib(\gamma) \simeq \fib\big( \fib(\alpha) \to \fib(\beta) \big) \]
	is also supported on finitely many points of $C'$.
	Hence, in order to obtain the vanishing $\pi_{-2}(q_*(\fib(\gamma)))\allowbreak \simeq 0$, it suffices to show that $\pi_{-2}( \fib(\gamma) ) \simeq 0$.
	As $\pi_{-2}( c_* \mathbb T\an_{C}( - \sum s_i ) )\simeq 0$,  we are then left to prove that the natural map
	\[\pi_{-1}( c_* \mathbb T\an_{C}( - {\textstyle\sum} s_i) ) \longrightarrow \pi_{-1}( \mathbb T\an_{C'}( - {\textstyle\sum} s_i') ).\]
	is surjective.
	We have
	\begin{multline*}
	\pi_{-1}( c_* \mathbb T\an_{C}( - {\textstyle\sum} s_i) )
	\twoheadrightarrow\pi_0\big(c_*\pi_{-1}(\bbT\an_C(-{\textstyle\sum} s_i))\big)
	\simeq c_*\Big(\bigoplus_{\nu\text{ node of }C}i_{\nu*}\cO_\nu\Big)\\
	\twoheadrightarrow \bigoplus_{\nu'\text{ node of }C'}i_{\nu'*}\cO_{\nu'}
	\simeq\pi_{-1}( \mathbb T\an_{C'}( - {\textstyle\sum} s_i') ),
	\end{multline*}
	where $i_\nu$ and $i_{\nu'}$ denote the inclusion of points.
		This completes the proof.
\end{proof}

\begin{proof}[Proof of \cref{lem:contracting_edges_colim}]
	The map $\oMpre_{\tsigma} \to \oM_\sigma$ factors as
	\[ \oMpre_{\tsigma} \longrightarrow \oMpre_\sigma \longrightarrow \oM_\sigma , \]
	where the first map is forgetting tails at the level of prestable curves, and the second map is stabilization.
	The first map is flat because forgetting each tail is isomorphic to the projection from the universal curve.
	The second map is flat by \cref{lem:stab_is_smooth}.
	So $\oMpre_{\tsigma} \to \oM_\sigma$ is flat, and therefore the fiber product $\oM_\tau \times_{\oM_\sigma} \oMpre_{\tsigma}$ is underived.
	By the construction of gluing along closed immersions (see \cite[(6.7)]{Porta_Yu_Representability_theorem}), the colimit is also underived.
	Therefore, it is enough to prove the lemma at the underived level.
	Observe that the normal crossing substack
	\[\oM_\tau\times_{\oM_\sigma}\oMpre_{\tsigma}\subset\oMpre_{\tsigma}\]
	is stratified by the images of
	\[\coprod_{(a\colon\tau_l^i\to\tau) \in \Lambda} \oMpre_{\tau_l^i},\]
	and the stratum given by $\oMpre_{\tau_l^i}$ has codimension $l$.
	Note it is enough to prove the lemma étale locally on $\oMpre_{\tsigma}$.
	After an étale base change, $\oM_\tau \times_{\oM_\sigma} \oMpre_{\tsigma}\subset\oMpre_{\tsigma}$ becomes a simple normal crossing divisor.
	At this point, the conclusion follows by induction from \cite[Theorem 6.5]{Porta_Yu_Representability_theorem}.
\end{proof}

Next we upgrade \cref{const:contracting_edges_modular_graph} to take into account curve classes.

\begin{construction}
	Notation as in \cref{const:contracting_edges_modular_graph}.
	Given $\beta \colon V_\sigma \to A(X)$, we define a category $\Lambda_\beta$ as follows.
	An object of $\Lambda_\beta$ consists of an object $(a\colon\tau_l^i\to\tau)$ of $\Lambda$ and a map $\beta^{i_j}_l\colon V_{\tau_l^i}\to A(X)$ that adds up to $\beta$ under the contraction $c_l^i \colon \tau_l^i \to \sigma$.
	A morphism in $\Lambda_\beta$ is a morphism between the underlying objects in $\Lambda$ that is compatible with the curve classes.
	Note we have a forgetful functor $\Lambda_\beta \to \Lambda$.
	(In the special case when $\sigma$ has a single non-tail vertex, $g(\sigma)=0$ and $k=0$, the category $\Lambda_\beta$ is equivalent to the one considered in \cite[\S 4.2]{Mann_Brane_actions}.)
\end{construction}

Using the description via mapping stacks as in \cref{rem:RMXtaubeta}, we can functorially attach to every object $(a\colon\tau_l^i\to\tau, \beta^{i_j}_l) \in \Lambda_\beta$ a commutative diagram
\begin{equation} \label{eq:Lambda_beta}
\begin{tikzcd}
	\R\oM(X/S, \tau_l^i, \beta_l^{i_j}) \arrow{r} \arrow{d} & \R\oM(X/S, \tsigma, \beta) \arrow{d} \\
	\oMpre_{\tau_l^i} \arrow{r} & \oMpre_{\tsigma} .
\end{tikzcd}
\end{equation}
Combining with diagrams \eqref{eq:trick_avoiding_full_functoriality} and \eqref{eq:commutativity_between_stabilization_and_contracting_edges}, we obtain a natural map
\[ p^{i_j}_l \colon \R\oM(X/S, \tau_l^i, \beta_l^{i_j}) \longrightarrow \oM_\tau \times_{\oM_\sigma} \R\oM(X/S,\tsigma,\beta).\]
By the functoriality of diagrams \eqref{eq:Lambda_beta} and \eqref{eq:trick_avoiding_full_functoriality}, we obtain a canonical map
\begin{equation} \label{eq:contracting_edge}
\colim_{(a\colon\tau_l^i\to\tau) \in \Lambda} \coprod_j \R\oM(X/S, \tau_l^i, \beta_l^{i_j}) \longrightarrow \oM_\tau \times_{\oM_\sigma} \R\oM(X/S,\tsigma,\beta) .
\end{equation}

\begin{theorem} \label{thm:contracting_edges}
	The canonical map \eqref{eq:contracting_edge} is an equivalence in $\dAnk$.
	Moreover, we have
	\[\sum_l (-1)^{l+1} \sum_{i,j} p_{l*}^{i_j}\big[\cO_{\R\oM(X/S,\tau^i_l,\beta^{i_j}_l)}\big] = \big[\cO_{\oM_\tau\times_{\oM_\sigma}\R\oM(X/S,\tsigma,\beta)}\big] \]
	in $K_0\big( \Cohb \big( \oM_\tau\times_{\oM_\sigma}\R\oM(X/S,\tsigma,\beta) \big) \big)$.
\end{theorem}

\begin{lemma} \label{lem:contracting_edges_pullback}
	For every $(a\colon\tau_l^i\to\tau)\in\Lambda$, the diagram
	\[\begin{tikzcd}
	\coprod_j \R\oM(X/S,\tau^i_l,\beta^{i_j}_l)\arrow{r}\arrow{d}&\R\oM(X/S,\tsigma,\beta)\arrow{d}\\
	\oMpre_{\tau_l^i}\arrow{r}&\oMpre_{\tsigma}
	\end{tikzcd}\]
	is a pullback.
\end{lemma}
\begin{proof}
	By \cref{rem:product_decomposition_M_tau}, we have an underived pullback square
	\[\begin{tikzcd}
	\oCpre_{\tau^i_l} \rar \dar & \oCpre_{\tsigma} \dar\\
	\oMpre_{\tau^i_l} \rar & \oMpre_{\tsigma}.
	\end{tikzcd}\]
	It is also a derived pullback square as the vertical maps are flat.
	Hence we obtain a derived pullback square
	\[\begin{tikzcd}
	\bfMap_{\oMpre_{\tau^i_l}\times S}( \oCpre_{\tau^i_l}\times S , X \times \oMpre_{\tau^i_l}) \arrow{r} \arrow{d} & \bfMap_{\oMpre_{\tsigma}\times S}( \oCpre_{\tsigma}\times S, X \times \oMpre_{\tsigma}) \arrow{d} \\
	\oMpre_{\tau^i_l} \arrow{r} & \oMpre_{\tsigma}.
	\end{tikzcd}\]
	Now it is enough to prove that the diagram
	\[ \begin{tikzcd}
	\coprod_j\R \oM(X/S, \tau^i_l,\beta^{i_j}_l) \arrow{r} \arrow{d} & \R\oM(X/S, \tsigma, \beta) \arrow{d} \\
	\bfMap_{\oMpre_{\tau^i_l}\times S}( \oCpre_{\tau^i_l}\times S , X \times \oMpre_{\tau^i_l}) \rar & \bfMap_{\oMpre_{\tsigma}\times S}( \oCpre_{\tsigma}\times S, X \times \oMpre_{\tsigma})
	\end{tikzcd} \]
	is a pullback square.
	By \cref{rem:RMXtaubeta}, the two vertical maps are Zariski open immersions, so it is enough to check the pullback property after passing to the truncations.
	At the truncated level, it follows directly from the definitions.
\end{proof}

\begin{proof}[Proof of \cref{thm:contracting_edges}]
	The first statement follows combining Lemmas \ref{lem:gluing_along_closed_immersion_is_universal}, \ref{lem:contracting_edges_colim} and \ref{lem:contracting_edges_pullback}.
	For the second one, consider the diagram
	\[ \begin{tikzcd}
		\displaystyle \colim_{(a\colon\tau_l^i\to\tau) \in \Lambda} \coprod_j \R \oM(X/S, \tau_l^i, \beta_l^{i_j}) \arrow{r} \arrow{d}{d_l^i} & \oM_\tau \times_{\oM_\sigma} \R \oM(X/S, \tsigma, \beta) \arrow{d}{d} \\
		\displaystyle \colim_{(a\colon\tau_l^i\to\tau) \in \Lambda} \oMpre_{\tau_l^i} \arrow{r} &  \oM_\tau \times_{\oM_\sigma} \oMpre_{\tsigma} ,
	\end{tikzcd} \]
	where $d \colon \R \oM(X/S, \tsigma, \beta) \to \oMpre_{\tsigma}$ and $d_l^i \colon \coprod_j \R \oM(X/S, \tau_l^i, \beta_l^{i_j}) \to \oMpre_{\tau_l^i}$ are the maps induced by taking domain.
	It follows from \cite[Lemma 3]{Lee_Quantum_K-theory_I}, that
	\[ \big[\cO_{\oM_\tau \times_{\oM_\sigma} \oMpre_{\tsigma}} \big] = \sum_l (-1)^{l+1} \sum_i q_{l*}^i \big[ \cO_{\oMpre_{\tau_l^i}} \big] \]
	in $K_0(\oMpre_{\tsigma})$,
		where
	\[ q_l^i \colon \oMpre_{\tau_l^i} \longrightarrow \oM_\tau \times_{\oM_\sigma} \oMpre_{\tsigma}. \]
	Let
	\[p^{i}_l \colon \coprod_j \R\oM(X/S, \tau_l^i, \beta_l^{i_j}) \longrightarrow \oM_\tau \times_{\oM_\sigma} \R\oM(X/S,\tsigma,\beta)\]
	be induced from all $p_l^{i_j}$.
	We have
	\begin{align*}
		\big[ \cO_{\oM_\tau \times_{\oM_\sigma} \R \oM(X/S, \tsigma, \beta)} \big] & = d^* \big[ \cO_{\oM_\tau \times_{\oM_\sigma} \oMpre_{\tsigma}} \big] \\
		& = \sum_l (-1)^{l+1} \sum_i d^* q_{l*}^i \big[ \cO_{\oMpre_{\tau_l^i}} \big] \\
		& = \sum_l (-1)^{l+1} \sum_{i} p_{l*}^{i} d_l^{i*} \big[ \cO_{\oMpre_{\tau_l^i}} \big] \\
		& = \sum_l (-1)^{l+1} \sum_{i,j} p_{l*}^{i_j} \big[ \cO_{\R\oMpre(X/S, \tau_l^i, \beta_l^{i_j})} \big] ,
	\end{align*}
	where the third equality follows from the derived base change theorem (see \cite[Theorem 1.5]{Porta_Yu_Derived_Hom_spaces}) and \cref{lem:contracting_edges_pullback}.
\end{proof}

\section{Deformation to the normal bundle} \label{sec:deformation_to_the_normal_bundle}

In this section we construct the deformation to the normal bundle of the inclusion $\trunc(X) \hookrightarrow X$ for any derived \kanal space $X$.
This is worked out in derived algebraic geometry by Gaitsgory-Rozenblyum in \cite[\S 9.2]{Gaitsgory_Study_II}.
Although it is possible to give a parallel treatment in the analytic setting, it is more straightforward to obtain this deformation via analytification relative to $X$ of its algebraic counterpart.
For this reason, we begin this section by discussing analytification relative to a derived \kanal space.

\subsection{Relative derived analytification} \label{sec:relative_analytification}

In \cite[\S 3]{Porta_Yu_Representability_theorem} we introduced and studied the derived analytification functor in derived non-archimedean geometry in the absolute case.
Here we extend the construction to the relative setting, first over any derived $k$-affinoid space.

Let $\dAff_k^{\mathrm{aft}}$ be the $\infty$-category of derived affine $k$-schemes almost of finite presentation, and $j \colon \dAff_k^{\mathrm{aft}} \hookrightarrow \dAff_k$ the inclusion into the \infcat of all derived affine $k$-schemes.
Equipping both sides with the étale topology, $j$ becomes a continuous and cocontinuous morphism of sites (see \cite[\S 2.4]{Porta_Yu_Higher_analytic_stacks} for the terminology).
We denote
\[ \dStlaft_k \coloneqq \Sh( \dAff_k^{\mathrm{aft}}, \tauet )^\wedge , \qquad \dSt_k \coloneqq \Sh(\dAff_k, \tauet)^\wedge,\]
the \infcats of stacks (i.e.\ hypercomplete sheaves) on these sites.
We refer to objects in $\dStlaft_k$ as \emph{derived stacks locally almost of finite type}.
Left Kan extension along $j$ induces a fully faithful functor
\[ j_s \colon \dStlaft_k \longhookrightarrow \dSt_k , \]
which has a right adjoint $j^s$.

Recall from \cref{sec:the_stack_of_derived_k-analytic_spaces} that
\[ \St(\dAfd_k) \coloneqq \Sh(\dAfd_k, \tauet)^\wedge . \]
The derived analytification functor (in the absolute case) introduced in \cite[\S 3]{Porta_Yu_Representability_theorem} induces a functor
\[ (-)\an \colon \dAff^{\mathrm{aft}}_k \longrightarrow \St(\dAfd_k) . \]
Left Kan extension along $\dAff^{\mathrm{aft}}_k \hookrightarrow \dStlaft_k$ induces a functor
\[ (-)^{\mathrm{an}, \mathrm{laft}} \colon \dStlaft_k \longrightarrow \St(\dAfd_k) . \]
On the other hand, \emph{right} Kan extension along $j \colon \dAff^{\mathrm{aft}}_k \hookrightarrow \dAff_k$ induces a functor
\[ (-)\an \colon \dAff_k \longrightarrow \St(\dAfd_k) . \]

\begin{example} \label{eg:analytification_affines_not_aft}
	Let $\Spec(A) \in \dAff_k$ be a derived affine $k$-scheme, not necessarily almost of finite type.
	Then
	\[ \Spec(A)\an \xrightarrow{\ \sim\ } \lim_{B \to A} \Spec(B)\an , \]
	where the limit ranges over all the morphisms $B \to A$ with $B$ almost of finite presentation.
\end{example}

\begin{lemma}
	Let $U \to V$ be an étale covering in $\dAff_k$ and $U_\bullet \colon \mathbf \Delta\op \to \St(\dAfd_k)$ its \v{C}ech nerve.
	Let $|U\an_\bullet|$ denote the colimit of the simplicial diagram $U\an_\bullet$.
	Then the canonical map
	\[ | U_\bullet\an | \longrightarrow V\an \]
	is an equivalence in $\St(\dAfd_k)$.
\end{lemma}

\begin{proof}
	It follows from \cite[Remark 2.2.8]{DAG-V} that we can assume there exists $V' \in \dAff_k^{\mathrm{aft}}$ and an étale covering $U' \to V'$ such that
	\[ \begin{tikzcd}
	U \arrow{r} \arrow{d} & V \arrow{d} \\
	U' \arrow{r} & V'
	\end{tikzcd} \]
	is a pullback square.
	Let $U'_\bullet$ be the \v{C}ech nerve of $U' \to V'$.
	Since
	\[ (-)\an \colon \dAff_k^{\mathrm{aft}} \longrightarrow \dAnk \]
	is a morphism of sites, we see that
	\[ | (U_\bullet')\an | \longrightarrow (V')\an \]
	is an equivalence in $\St(\dAfd_k)$.
	Since colimits are stable under pullbacks in $\St(\dAfd_k)$, the diagram
	\[ \begin{tikzcd}
	{| U_\bullet\an |} \arrow{r} \arrow{d} & V \arrow{d} \\
	{| (U_\bullet')\an |} \arrow{r} & (V')\an
	\end{tikzcd} \]
	is a pullback square.
	It follows that the top horizontal morphism is an equivalence, completing the proof.
\end{proof}

The above lemma guarantees that left Kan extension along $(-)\an \colon \dAff_k \to \St(\dAfd_k)$ induces a colimit preserving functor
\[ (-)\an \colon \dSt_k \longrightarrow \St(\dAfd_k) . \]
The following lemma ensures that there is no ambiguity when analytifying a derived stack locally almost of finite type:

\begin{lemma}
	For any $F \in \dStlaft_k$, there is a canonical equivalence
	\[ F^{\mathrm{an}, \mathrm{laft}} \simeq ( j_s(F) )\an . \]
\end{lemma}

\begin{proof}
	As both $j_s \colon \dStlaft_k \to \dSt_k$ and $(-)\an \colon \dSt_k \to \St(\dAfd_k)$ commute with colimits, it is enough to check that the two constructions agree when $F \in \dAff_k^{\mathrm{aft}}$.
	In this case, the result follows from the full faithfulness of the functors $\dAff_k^{\mathrm{aft}} \hookrightarrow \dAff_k$ and $\dAff_k^{\mathrm{aft}} \to \dStlaft_k$.
\end{proof}

Fix now a derived $k$-affinoid space $X \in \dAfd_k$.
Let $A \coloneqq \Gamma(\cO_X\alg)$.
For any $\Spec(B) \in \dAff_k^{\mathrm{afp}}$, we have
\begin{align*}
\Map_{\dSt_k}( \Spec(A), \Spec(B) ) & \simeq \Map_{\RTop(\cTetk)}( \Spec(A), \Spec(B) ) \\
& \simeq \Map_{\RTop(\cTetk)}( X\alg, \Spec(B) ) \\
& \simeq \Map_{\dAnk}( X, \Spec(B)\an )
\end{align*}
When $\Spec(B) \in \dAff_k$, we get
\begin{align*}
\Map_{\dSt_k}( \Spec(A), \Spec(B) ) & \simeq \lim_{\Spec(B') \to \Spec(B)} \Map_{\dSt_k}( \Spec(A), \Spec(B') ) \\
& \simeq \lim_{\Spec(B') \to \Spec(B)} \Map_{\St(\dAfd_k)}( X, \Spec(B')\an ) \\
& \simeq \Map_{\St(\dAfd_k)}( X, \Spec(B)\an ) ,
\end{align*}
where the colimits range over all the morphisms $\Spec(B') \to \Spec(B)$ with $B'$ almost of finite presentation.
In particular, taking $B = A$, we obtain a canonical morphism
\[ \eta_X \colon X \to \Spec(A)\an . \]
Pulling back along $\eta_X$ gives a functor
\[ \eta_X^* \colon \dAnSt_{/\Spec(A)\an \times Y\an} \longrightarrow \dAnSt_{/X \times Y\an} . \]
At this point, we can give the following definition.

\begin{definition}
	Let $X \in \dAfd_k$ be a derived $k$-affinoid space and let $A \coloneqq \Gamma(\cO_X\alg)$.
	The functor $(-)\an_X$ of \emph{analytification relative to $X$} is the composition
	\[ \dSt_{/\Spec(A)} \xrightarrow{(-)\an} \dAnSt_{/\Spec(A)\an} \xrightarrow{\ \eta_X^*\ } \dAnSt_{/X},\]
\end{definition}

\begin{lemma} \label{lem:relative_analytification}
	Let $X \in \dAfd_k$ be a derived $k$-affinoid space, $A \coloneqq \Gamma(\cO_X\alg)$ and $S \coloneqq \Spec(A)$.
	The following hold:
	\begin{enumerate}
		\item \label{lem:relative_analytification:identity} The analytification of $S$ relative to $X$ coincides with $X$ itself.
		
		\item \label{lem:relative_analytification:affine_space} The analytification of $\bbA^n_A$ relative to $X$ coincides with $\bA^n_X$.
		
		\item \label{lem:relative_analytification:truncations} The analytification of $\mathrm t_{\le n} S$ relative to $X$ coincides with $\mathrm t_{\le n} X$.
		
		\item  \label{lem:relative_analytification:reduced_part} The analytification of $S_\red$ relative to $X$ coincides with $X_\red$.
	\end{enumerate}
\end{lemma}

\begin{proof}
	(\ref{lem:relative_analytification:identity}) follows directly from the definition.
	Let us prove (\ref{lem:relative_analytification:affine_space}).
	Since $(-)\an \colon \dSt_k \to \St(\dAfd_k)$ commutes with finite limits, we have
	\[ (\bbA^n_A)\an \simeq ( \bbA^n_k)\an \times S\an \simeq \bA^n_k \times S\an . \]
	Thus we deduce that
	\[ (\bbA^n_A)\an_X \simeq \bA^n_X . \]
	
	For (\ref{lem:relative_analytification:truncations}), we start by considering the composite functor
	\[ (\dAff_k^{\mathrm{aft}})_{S/} \to (\dAff_k^{\mathrm{aft},\le n})_{\mathrm t_{\le n} S/} \to (\dAff_k^{\mathrm{aft}})_{\mathrm t_{\le n} S/} , \]
	where $(\dAff_k^{\mathrm{aft}, \le n})_{\mathrm t_{\le n} S/}$ denotes the full subcategory of $(\dAff_k^{\mathrm{aft}})_{\mathrm t_{\le n} S/}$ spanned by $n$-truncated derived affine schemes almost of finite type.
	The universal property of the truncation shows that the second functor is final.
	
	Since the functor $\mathrm t_{\le n}(-)$ commutes with limits, the equivalence
	\[ S \simeq \lim_{S' \in (\dAff_k^{\mathrm{aft}})_{S/}} S' , \]
	implies
	\[ \mathrm t_{\le n} S \simeq \lim_{S' \in (\dAff_k^{\mathrm{aft}})_{S/}} \mathrm t_{\le n} S' . \]
	For any $n$-truncated derived affine scheme almost of finite type $T \in \dAff_k^{\mathrm{aft}}$, \cite[7.2.4.26]{Lurie_Higher_algebra} shows that the canonical map
	\[ \colim_{S' \in (\dAff_k^{\mathrm{aft}})_{S/}} \Map_{\dAff_k}( \mathrm t_{\le n} S', T ) \longrightarrow \Map_{\dAff_k}( \mathrm t_{\le n } S , T ) \]
	is an equivalence.
	This implies that $(\dAff_k^{\mathrm{aft}})_{S/} \to (\dAff_k^{\mathrm{aft}, \le n})_{\mathrm t_{\le n} S/}$ is final as well.
	
	Combining with \cref{eg:analytification_affines_not_aft}, we deduce that
	\[ (\mathrm t_{\le n} S)\an \simeq \lim_{S' \in (\dAff^{\mathrm{aft}}_k)_{S/}} (\mathrm t_{\le n} S')\an . \]
	At this point, for any $n$-truncated $k$-affinoid space $U \in \dAfd_k$, we have
	\begin{align*}
	\Map_{\St(\dAfd_k)}( U, (\mathrm t_{\le n} S)\an ) & \simeq \lim_{S' \in (\dAff_k^{\mathrm{aft}})_{S/}} \Map_{\St(\dAfd_k)}( U,  ( \mathrm t_{\le n} S' )\an ) \\
	& \simeq \lim_{S' \in (\dAff_k^{\mathrm{aft}})_{S/}} \Map_{\St(\dAfd_k)}( U, \mathrm t_{\le n}( S^{\prime \mathrm{an}} ) ) \\
	& \simeq \lim_{S' \in (\dAff_k^{\mathrm{aft}})_{S/}} \Map_{\St(\dAfd_k)}( U, S^{\prime \mathrm{an}} ) \\
	& \simeq \Map_{\St(\dAfd_k)}( U, S\an ) .
	\end{align*}
	The second equivalence follows from the fact that $B$ is almost of finite type and the flatness of the analytification (see \cite[Corollary 5.15]{Porta_Derived_complex_analytic_geometry_I} and \cite[Proposition 4.17]{Porta_Yu_Representability_theorem}).
	The third equivalence follows from the universal property of the $n$-truncation.
	This implies that
	\[ \Map_{\St(\dAfd_k)}( U, ( \mathrm t_{\le n} S)\an_X ) \simeq \Map_{\St(\dAfd_k)}( U, X ) . \]
	Therefore, by the universal property of the $n$-truncation again, we conclude that $(\mathrm t_{\le n} S)\an_X$ coincides with $\mathrm t_{\le n}(X)$.
	
	Finally, (\ref{lem:relative_analytification:reduced_part}) follows along the same lines of the proof of (\ref{lem:relative_analytification:truncations}).
\end{proof}

\begin{proposition}
	Let $X \in \dAfd_k$ be a derived $k$-affinoid space, $A \coloneqq \Gamma(\cO_X\alg)$, and $Y$ a derived algebraic stack locally almost of finite type over $\Spec(A)$.
	Then $Y\an_X$ is a derived \kanal stack.
	\end{proposition}

\begin{proof}
	It follows by induction on the geometric level combining \cref{lem:relative_analytification}(\ref{lem:relative_analytification:affine_space}-\ref{lem:relative_analytification:truncations}) as in the proof of \cite[Proposition 3.7]{Porta_Yu_Representability_theorem}.
\end{proof}

We can further extend the definition of relative analytification over global (i.e.\ non-affinoid) bases as follows:
Given $X \in \dAnk$ any derived \kanal space, let $X\et^\afd$ denote the full subcategory of the étale site $X\et$ spanned by étale morphisms $U \to X$ where $U$ is derived $k$-affinoid.
For every $U \in X\et^\afd$, write $A_U \coloneqq \Gamma(\cO_U\alg) $.
Then the derived analytification functors
\[ (-)\an_U \colon \dSt_{/ \Spec(A_U)} \longrightarrow \dAnSt_{/ U} \]
glue together to provide
\[ (-)\an_X \colon \lim_{U \in X\et^\afd} \dSt_{/\Spec(A_U)} \longrightarrow \lim_{U \in X\et^\afd} \dAnSt_{/U} \simeq \dAnSt_{/X} , \]
where the last equivalence is a consequence of the descent theory for $\infty$-topoi, see \cite[6.1.3.9]{Lurie_HTT}.

\subsection{Formal moduli problems in derived analytic geometry}

As an application of the relative derived analytification introduced above, we prove here that all formal moduli problems over a given derived $k$-affinoid space are algebraic, see \cref{thm:FMP}.

\begin{definition}
	For any derived \kanal space $X \in \dAnk$, we define $X_\red \coloneqq \trunc(X)_\red$, called the reduced \kanal space associated to $X$.
\end{definition}

\begin{definition}
	A morphism $Y \to X$ in $\dAnk$ is called a \emph{nil-isomorphism} if the induced morphism $Y_\red \to X_\red$ is an isomorphism.
	We denote by $\mathrm{Nil}\an_{/X}$ the full subcategory of $(\dAnk)_{/X}^{\mathrm{ft}}$ spanned by nil-isomorphisms $Y \to X$ of finite type.
	\end{definition}

Let $X$ be a derived $k$-affinoid space.
If $Y \in \dAnk$ is a derived \kanal space and $Y \to X$ is a nil-isomorphism, then $Y$ is also a derived $k$-affinoid space.
Moreover, we have the following basic result:

\begin{lemma} \label{lem:analytic_cotangent_complex_nil-isomorphism}
	Let $Y \to X$ be a nil-isomorphism of derived $k$-affinoid spaces.
	Let $A \coloneqq \Gamma(\cO_X\alg)$ and $B \coloneqq \Gamma(\cO_Y\alg)$.
	Then:
	\begin{enumerate}
		\item The morphism $A \to B$ is a nil-isomorphism almost of finite type.
		\item Let $\anL_{B/A} \coloneqq \Gamma(\anL_{Y/X})$; the canonical morphism
		\[ \bbL_{B/A} \longrightarrow \anL_{B/A} \]
		is an equivalence.
		\item If $Y \to X$ is of finite type, then so is $A \to B$.
	\end{enumerate}
\end{lemma}

\begin{proof}	
	It follows from \cite[Theorem 3.4]{Porta_Yu_Derived_Hom_spaces} that $A_\red$ and $B_\red$ are the global sections of $X_\red$ and $Y_\red$, respectively.
	This implies that the morphism $A \to B$ is a nil-isomorphism.
	Let us show that it is also almost of finite type.
	Since both $A$ and $B$ are noetherian in the derived sense (see \cite[7.2.4.30]{Lurie_Higher_algebra}), the derived Hilbert basis theorem \cite[7.2.4.31]{Lurie_Higher_algebra} implies that $A \to B$ is almost of finite presentation if and only if $\pi_0(A) \to \pi_0(B)$ is of finite presentation.
	Since $A_\red \to B_\red$ is an isomorphism, we can find a map
	\[ \phi \colon \pi_0(A)[X_1, \ldots, X_n] \longrightarrow \pi_0(B) \]
	which becomes surjective after composing with $\pi_0(B) \to \pi_0(B)_\red = B_\red$.
	Since $\pi_0(B)$ is noetherian, its nilradical $\cN \mathrm{il}_{\pi_0(B)}$ is generated by finitely many elements $b_1, \ldots, b_m \in \pi_0(B)$.
	We now consider the extension of $\phi$
	\[ \pi_0(A)[X_1, \ldots, X_n, Y_1, \ldots, Y_m] \longrightarrow \pi_0(A) \]
	which sends $Y_i$ to $b_i$.
	This map is surjective, and hence $\pi_0(B)$ is of finite type as $\pi_0(A)$-algebra.
	
	We now prove (2).
	Let $i \colon \Spec(B_\red) \to \Spec(B)$ be the natural inclusion.
	Since $A \to B$ is almost of finite type, \cite[7.4.3.18]{Lurie_Higher_algebra} implies that $\bbL_{B/A}$ is almost of finite presentation.
	Using \cite[Corollary 5.40]{Porta_Yu_Representability_theorem} and \cite[Theorem 3.4]{Porta_Yu_Derived_Hom_spaces}, we deduce that $\anL_{B/A}$ is almost of finite presentation too.
	Therefore, by Nakayama's lemma, it is enough to check that $i^* \bbL_{B/A} \to i^* \anL_{B/A}$ is an equivalence.
	Consider the following diagram:
	\[ \begin{tikzcd}
		i^* \bbL_{B/A} \arrow{r} \arrow{d} & \bbL_{B_\red/A} \arrow{r} \arrow{d} & \bbL_{B_\red / B} \arrow{d} \\
		i^* \anL_{B/A} \arrow{r} & \anL_{B_\red / A} \arrow{r} & \anL_{B_\red / B} .
	\end{tikzcd} \]
	Since $Y \to X$ is a nil-isomorphism, the map $Y_\red \simeq X_\red \to X$ is a closed immersion.
	Therefore, \cite[Corollary 5.33]{Porta_Yu_Representability_theorem} implies that both the middle and right vertical arrows are equivalences.
	As the horizontal lines are fiber sequences, it follows that the left vertical arrow is also an equivalence.
	
	Finally, (3) follows immediately from (2).
\end{proof}

Let $X \in \dAfd_k$ be a derived $k$-affinoid space and let $A \coloneqq \Gamma(\cO_X\alg)$.
Denote by $\mathrm{Nil}_{/\Spec(A)}$ the $\infty$-category of nil-isomorphisms of finite type $\Spec(B) \to \Spec(A)$.
By the above lemma, taking global sections and applying $\Spec$, we obtain a canonical functor
\[ (-)\alg_X \colon \mathrm{Nil}\an_{/X} \longrightarrow \mathrm{Nil}_{/\Spec(A)} . \]
The following result is a generalization of \cite[Proposition 8.2]{Porta_Derived_complex_analytic_geometry_I}.

\begin{proposition} \label{prop:nil-isomorphisms_are_algebraic}
	Let $X \in \dAfd_k$ be a derived $k$-affinoid space and let $A \coloneqq \Gamma(\cO_X\alg)$.
	The functor $(-)\alg_X \colon \mathrm{Nil}\an_{/X} \to \mathrm{Nil}_{/\Spec(A)}$ is an equivalence.
\end{proposition}

\begin{proof}
	Consider the relative analytification functor introduced in \cref{sec:relative_analytification}:
	\[ (-)\an_X \colon (\dSt_k)^{\mathrm{laft}}_{/\Spec(A)} \longrightarrow \dAnSt_{/X} . \]
		Let $\Spec(B) \to \Spec(A)$ be a nil-isomorphism of finite type.
	Set $Y \coloneqq \Spec(B)\an_X$.
	By the transitivity property for pullback squares (\cite[4.4.2.1]{Lurie_HTT}), given any $\Spec(B') \to \Spec(B)$, we have a canonical equivalence
	\[ \Spec(B')\an_Y \simeq \Spec(B')\an_X . \]
	In particular, \cref{lem:relative_analytification}(\ref{lem:relative_analytification:reduced_part}) implies that the functor $(-)\an_X$ restricts to a functor
	\[ (-)\an_X \colon \mathrm{Nil}_{/\Spec(A)} \longrightarrow \mathrm{Nil}\an_{/X} . \]
	
	It is then enough to see that $(-)\alg_X$ and $(-)\an_X$ form an equivalence of $\infty$-categories.
	First we observe that $(-)\alg_X$ is left adjoint to $(-)\an_X$.
	Moreover, $(-)\alg_X$ is conservative: indeed, let $Y' \to Y$ be a morphism in $\mathrm{Nil}\an_{/X}$ such that the morphism $B \to B'$ obtained by passing to global sections is an equivalence.
	By \cref{lem:analytic_cotangent_complex_nil-isomorphism}, we have
	\[ \Gamma(Y', \anL_{Y' / Y}) \simeq \anL_{B'/B} \simeq 0 . \]
	Using \cite[Theorem 3.4]{Porta_Yu_Derived_Hom_spaces}, we conclude that $\anL_{Y'/Y} \simeq 0$, and hence $Y' \to Y$ is étale.
	As $Y'_\red \to Y_\red$ is an isomorphism, we conclude that $Y' \to Y$ is an equivalence.
	
	At this point, we are left to prove that for any $\Spec(B) \in \mathrm{Nil}_{/\Spec(A)}$ the unit morphism
	\[ \Spec(B) \longrightarrow (\Spec(B)\an_X)\alg_X \]
	is an equivalence.
	This follows from the compatibility of $(-)\an_X$ and $(-)\alg_X$ with the Postnikov tower decomposition, which are respectively guaranteed by \cite[Proposition 4.17]{Porta_Yu_Representability_theorem} and \cite[Theorem 3.23]{Porta_Yu_Derived_non-archimedean_analytic_spaces}.
	\end{proof}

\begin{definition}
	Let $X \in \dAfd_k$ be a derived $k$-affinoid space.
	An \emph{analytic formal moduli problem over $X$} is a functor
	\[ F \colon (\mathrm{Nil}_{/X}\an)\op \longrightarrow \cS \]
	satisfying the following conditions:
	\begin{enumerate}
		\item The space $F(X_\red)$ is contractible;
		\item Let
		\[ \begin{tikzcd}
		S \arrow{r}{f} \arrow{d} & S' \arrow{d} \\
		T \arrow{r} & T'
		\end{tikzcd} \]
		be a pushout square in $\mathrm{Nil}_{/X}$ where $f \colon S \to S'$ is a $k$-analytic square-zero extension.
		Then the canonical map
		\[ F(T') \longrightarrow F(T) \times_{F(S)} F(S') \]
		is an equivalence.
	\end{enumerate}
	We denote by $\mathrm{FMP}\an_{/X}$ the $\infty$-category of analytic formal moduli problems over $X$.
\end{definition}

\begin{theorem} \label{thm:FMP}
	Let $X \in \dAfd_k$ be a derived $k$-affinoid space and let $A \coloneqq \Gamma(\cO_X\alg)$.
	There is a canonical equivalence
	\[ \mathrm{FMP}_{/\Spec(A)} \xrightarrow{\ \sim\ } \mathrm{FMP}\an_{/X} . \]
\end{theorem}

\begin{proof}
	It follows from \cite[Proposition 5.1.2.2(c)]{Gaitsgory_Study_II} and \cref{prop:nil-isomorphisms_are_algebraic}.
\end{proof}

\subsection{The deformation to the normal bundle in the algebraic setting} \label{sec:deformation_algebraic}

Here we review the deformation to the normal bundle in the algebraic setting following Gaitsgory-Rozenblyum \cite{Gaitsgory_Study_II}.
Let
\[ \mathrm{Bifurc}^\bullet_{\mathrm{scaled}} \colon \mathbf \Delta\op \longrightarrow \mathrm{dSt}_{/\mathbb A^1_k} \]
be as in \cite[\S 9.2.2.6]{Gaitsgory_Study_II}, which we will denote by $\rB^\bullet$ to simplify the notation.
Recall that $\rB^n$ is obtained by gluing together $n+1$ copies of $\mathbb A^1_k$ along $0$.
Let $X \in \dStlaft_k$ be a derived algebraic stack locally almost of finite type over $k$, and $j \colon \trunc(X) \hookrightarrow X$ the inclusion of its truncation.
Consider the following fiber product in $\Fun( \mathbf \Delta\op, \dSt_{/\mathbb A^1_k})$:
\[ \begin{tikzcd}
\mathfrak D_X^\bullet \arrow{r} \arrow{d} & X \times \bbA^1_k \arrow{d} \\
\bfMap_{/\bbA^1_k}( \rB^\bullet, X_0 \times \bbA^1) \arrow{r} & \bfMap_{/\bbA^1_k}( \rB^\bullet, X \times \bbA^1_k ) ,
\end{tikzcd} \]
where $X \times \bbA^1_k$ is viewed as a constant simplicial object in $\dSt_{/\bbA^1_k}$.
It is proven in \cite[Lemma 9.2.3.4]{Gaitsgory_Study_II} that $\mathfrak D_X^\bullet$ is a groupoid object in $\dSt_{/\bbA^1_k}$.
In particular, Theorem 5.2.3.2 in loc.\ cit.\ (see also \cite[Theorem 3.9]{Calaque_Grivaux_FMP} and \cite{Nuiten_Koszul_duality_Lie_algebroids}) allows to associate to $\mathfrak D_X^\bullet$ a formal moduli problem $\mathfrak D_X$ under $X_0 \times \bbA^1_k$ and over $X \times \bbA^1_k$.
Observe that this construction is functorial in $X$.
The following proposition summarizes the main properties of $\mathfrak D_X$.

\begin{proposition} \label{prop:properties_algebraic_deformation}
	Let $X \in \dStlaft_k$ be a derived algebraic stack locally almost of finite type over $k$.
	The following hold:
	\begin{enumerate}
		\item \label{prop:properties_algebraic_deformation:geometricity} The total space $\mathfrak D_X$ is a derived algebraic stack locally almost of finite type over $X \times \bbA^1_k$.
		If $X$ is a derived \DM stack (resp.\ a derived scheme, a derived affine scheme), then so is $\mathfrak D_X$.
		
		\item \label{prop:properties_algebraic_deformation:truncation} There is a canonical equivalence
		\[ \trunc( \mathfrak D_X ) \simeq \trunc(X) \times \bbA^1_k . \]
		
		\item \label{prop:properties_algebraic_deformation:trivialization} There is a canonical equivalence
		\[ \mathfrak D_X \times_{\bbA^1_k} \bbG_m \simeq X \times \bbG_m . \]
		
		\item \label{prop:properties_algebraic_deformation:central_fiber} There is a canonical equivalence
		\[ \mathfrak D_X \times_{\bbA^1_k} 0 \simeq \Spec_{\trunc(X)}\big( \Sym_{\cO_{\trunc(X)}}( \bbL_{\trunc(X)/X}[-1])\big) . \]
		Moreover, given any morphism $f \colon X \to Y$, let $\mathfrak f \colon \mathfrak D_X \to \mathfrak D_Y$ denote its deformation.
		Then the fiber of $\ff$ at $0 \in \bbA^1_k$
		\[ \mathfrak f_0 \colon \mathfrak D_X \times_{\bbA^1_k} 0 \longrightarrow \mathfrak D_Y \times_{\bbA^1_k} 0 \]
		coincides with the morphism induced by the map
		\[ f^* \bbL_{\trunc(Y) / Y} \longrightarrow \bbL_{\trunc(X) / X} . \]
	\end{enumerate}
\end{proposition}

\begin{proof}
	Property (\ref{prop:properties_algebraic_deformation:truncation}) follows from the canonical identification
	\[ \trunc( \bfMap_{/\bbA^1_k}( \rB^\bullet, X \times \bbA^1_k) ) \simeq \bfHom_{/\bbA^1_k}( \rB^\bullet, \trunc(X) \times \bbA^1_k ) , \]
	which holds because $\rB^\bullet$ is flat over $\bbA^1_k$.
	Property (\ref{prop:properties_algebraic_deformation:central_fiber}) follows from \cite[Proposition 9.2.3.6]{Gaitsgory_Study_II}.
	
	Let us prove property (\ref{prop:properties_algebraic_deformation:geometricity}).
	Proceeding as in \S 9.5.1 in loc.\ cit.\ and using Theorem 9.5.1.3 as main input, we constructs sequence
	\[ \fX^{(0)} \hookrightarrow \fX^{(1)} \hookrightarrow \fX^{(2)} \hookrightarrow \cdots \]
	of derived stacks enjoying the following properties:
	\begin{enumerate}[(i)]
		\item We have $\fX^{(0)} = X_0 \times \bbA^1_k$, and $\fX^{(n+1)}$ is obtained as a square-zero extension of $\fX^{(n)}$ by an element in $\Coh^{\ge n}( \fX^{(n)} )$.
		In particular, each $\fX^{(n)}$ is geometric, and the connectivity assumption guarantees that the colimit $\colim_n \fX^{(n)}$ is also geometric.
		If $X$ is a \DM stack (resp.\ a derived scheme, a derived affine scheme), then the same goes for each $\fX^{(n)}$ and for the colimit.
		
		\item The above sequence is equipped with compatible maps $\fX^{(n)} \to \mathfrak D_X$, and the induced morphism
		\[ \colim_n \fX^{(n)} \longrightarrow \mathfrak D_X \]
		is an equivalence.
		This is the content of Proposition 5.2.2 in loc.\ cit.
			\end{enumerate}
	
	We are left to prove property (\ref{prop:properties_algebraic_deformation:trivialization}).
	Observe that $\rB^n \times_{\bbA^1_k} \bbG_m \simeq \bbG_m^{\sqcup n}$.
	Unraveling the definitions, we obtain
	\[ \mathfrak D_X^\bullet \times_{\bbA^1_k} \bbG_m \simeq \Cech( \trunc(X) \times \bbG_m \to X \times \bbG_m ) . \]
	Example 5.2.3.3 in loc.\ cit.\ shows that the formal moduli problem associated to this formal groupoid coincides with the formal completion of $\trunc(X) \times \bbG_m$ inside $X \times \bbG_m$.
	This formal completion can be written as
	\[ (\trunc(X) \times \bbG_m)_{\mathrm{dR}} \times_{(X \times \bbG_m)_{\mathrm{dR}}} ( X \times \bbG_m ) . \]
	As
	\[ (\trunc(X) \times \bbG_m)_{\mathrm{dR}} \longrightarrow (X \times \bbG_m)_{\mathrm{dR}} \]
	is an equivalence, the conclusion follows.
	\end{proof}

\begin{remark} \label{rem:extended_deformation}
	By \cref{prop:properties_algebraic_deformation}(\ref{prop:properties_algebraic_deformation:trivialization}) we can extend the deformation $\mathfrak D_X \to \bbA^1_k$ to a deformation
	\[ \widetilde{\mathfrak D}_X \to \bbP^1_k , \]
	which is trivial near infinity.
\end{remark}

\begin{remark} \label{rem:non_laft_deformation}
	The construction of $\widetilde{\mathfrak D}_X$ in \cref{rem:extended_deformation} induces a functor
	\[ \widetilde{\mathfrak D} \colon \dAff_k^{\mathrm{aft}} \longrightarrow (\dSt_k)_{/\bbP^1_k} \]
	sending every $X \in \dAff_k^{\mathrm{aft}}$ to $\widetilde{\mathfrak D}_X$.
	Right Kan extension along $\dAff_k^{\mathrm{aft}} \hookrightarrow \dAff_k$ provides a functor
	\[ \widetilde{\mathfrak D} \colon \dAff_k \longrightarrow (\dSt_k)_{/\bbP^1_k} . \]
	Moreover, for every $X \in \dAff_k$, the deformation
	\[ \mathfrak D_X \coloneqq \widetilde{\mathfrak D}_X \times_{\bbP^1_k} \bbA^1_k \]
	inherits all the properties of \cref{prop:properties_algebraic_deformation}.
\end{remark}

\subsection{Analytifying the deformation} \label{sec:deformation_analytic}

Let $X \in \dAnk$ be a derived \kanal space.
As in \cref{sec:relative_analytification}, we denote by $X\et^\afd$ the full subcategory of the small étale site $X\et$ of $X$ spanned by étale morphisms $U \to X$, where $U$ is derived $k$-affinoid.
For any $U \in X\et^\afd$, put $A_U \coloneqq \Gamma( \cO_U\alg )$ and $U\alg \coloneqq \Spec(A_U)$.

\begin{lemma} \label{lem:gluing_deformation}
	For any morphism $V \to U$ in $X\et^\afd$, the diagram
	\[ \begin{tikzcd}
		\widetilde{\mathfrak D}_{V\alg} \arrow{r} \arrow{d} & \widetilde{\mathfrak D}_{U\alg} \arrow{d} \\
		V\alg \arrow{r} & U\alg
	\end{tikzcd} \]
	is a pullback square, where $\widetilde{\mathfrak D}$ denotes the extended deformation introduced in \cref{rem:non_laft_deformation}.
\end{lemma}

\begin{proof}
	By construction, it suffices to prove the analogous statement for $\mathfrak D_{V\alg}$ and $\mathfrak D_{U\alg}$.
	Choose a cofiltered diagram $F \colon I \to (\dAff_k^{\mathrm{ft}})^{\Delta^1}$ whose limit $\lim_{i \in I} F(i)$ is equivalent to $V\alg \to U\alg$.
	We write
	\[ F(i) \colon U_i \longrightarrow V_i . \]
	Then, unraveling the definition, we find canonical equivalences
	\[ \mathfrak D_{U\alg} \simeq \lim_{i \in I} \mathfrak D_{U_i} , \qquad \mathfrak D_{V\alg} \simeq \lim_{i \in I} \mathfrak D_{V_i} . \]
	Applying \cref{prop:properties_algebraic_deformation}(\ref{prop:properties_algebraic_deformation:truncation}) to each $\mathfrak D_{U_i}$ and $\mathfrak D_{V_i}$, and passing to the limit, we see that
	\[ \trunc( \mathfrak D_{U\alg} ) \simeq \trunc(U\alg) \times \bbA^1_k, \qquad \trunc( \mathfrak D_{V\alg} ) \simeq \trunc(V\alg) \times \bbA^1_k . \]
	Therefore, the natural map
	\[ \mathfrak D_{V\alg} \longrightarrow V\alg \times_{U\alg} \mathfrak D_{U\alg} \]
	induces an equivalence on the truncation.
	By \cref{prop:properties_algebraic_deformation}(1), both $\mathfrak D_{V\alg}$ and $V\alg \times_{U\alg} \mathfrak D_{U\alg}$ are affine derived schemes, almost of finite presentation over $V\alg$.
	Write
	\[ \mathfrak D_{V\alg} \simeq \Spec(R), \qquad V\alg \times_{U\alg} \mathfrak D_{U\alg} \simeq \Spec(R') . \]
	By \cite[Theorem 1.3]{Porta_Yu_Derived_Hom_spaces}, $A_V$ is noetherian in the sense of \cite[Definition 7.2.4.30]{Lurie_Higher_algebra}.
	Since $R$ and $R'$ are $A_V$-algebras almost of finite presentation, Proposition 7.2.4.31 in loc.\ cit.\ implies that they are noetherian as well.
	Therefore applying the same result again, we see that the maps $R \to \pi_0(R)$ and $R' \to \pi_0(R)$ are almost of finite presentation.
	At this point, \cite[Lemma 3.1.1]{Halpern-Leistner_Preygel_Categorical_properness} implies that the map $R' \to R$ is an equivalence if and only if the induced map of \v{C}ech nerves
	\[ \Cech\big( \trunc(V\alg) \times \bbA^1_k \to \mathfrak D_{V\alg} \big) \longrightarrow \Cech\big( \trunc(V\alg) \times \bbA^1_k \to V\alg \times_{U\alg} \mathfrak D_{V\alg} \big) \]
	is an equivalence of simplicial objects.
	We have
	\[ \Cech\big( \trunc(V\alg) \times \bbA^1_k \to \mathfrak D_{V\alg} \big) \simeq \lim_{i \in I} \Cech\big( \trunc(V_i) \times \bbA^1_k \to \mathfrak D_{V_i\alg} \big) , \]
	and similarly for the other \v{C}ech nerve.
	It follows from the construction of $\mathfrak D_{V_i}$ (via \cite[Theorem 5.2.3.2]{Gaitsgory_Study_II}) that we can identify this \v{C}ech nerve with
	\[ \mathfrak D^\bullet_{V\alg} \coloneqq \lim_{i \in I} \mathfrak D^\bullet_{V_i} . \]
	Similarly, we can identify the other \v{C}ech nerve with
	\[ V\alg \times_{U\alg} \mathfrak D^\bullet_{U\alg} \coloneqq \lim_{i \in I} V_i \times_{U_i} \mathfrak D^\bullet_{U_i} . \]
	We are thus reduced to prove that for every $[n] \in \mathbf \Delta$, the diagram
	\[ \begin{tikzcd}
		\mathfrak D^n_{V\alg} \arrow{r} \arrow{d} & \mathfrak D^n_{U\alg} \arrow{d} \\
		V\alg \arrow{r} & U\alg
	\end{tikzcd} \]
	is a pullback square.
	Unwinding the definitions, we see that it is enough to prove that the diagram
	\[ \begin{tikzcd}
		\bfMap_{/\bbA^1_k}\big( \rB^n, \trunc(V\alg) \times \bbA^1_k \big) \arrow{r} \arrow{d} & \bfMap_{/\bbA^1_k}\big( \rB^n, V\alg \times \bbA^1_k \big) \arrow{d} \\
		\bfMap_{/\bbA^1_k}\big( \rB^n, \trunc(U\alg) \times \bbA^1_k \big) \arrow{r} & \bfMap_{/\bbA^1_k}\big( \rB^n, U\alg \times \bbA^1_k \big)
	\end{tikzcd} \]
	is a pullback square.
	For this, it is enough to verify that
	\[ \begin{tikzcd}
		\trunc(V\alg) \arrow{r} \arrow{d} & V\alg \arrow{d} \\
		\trunc(U\alg) \arrow{r} & U\alg
	\end{tikzcd} \]
	is a pullback.
	Since $V \to U$ is \'etale, \cref{lem:equivalent_formulation_flatness} implies that $A_U \to A_V$ is flat and therefore the previous square is indeed a pullback.
	The conclusion follows.
\end{proof}

Thanks to the above lemma, we can interpret the assignment $U \mapsto \widetilde{\mathfrak D}_{U\alg}$ as an element in the limit
\[ \lim_{U \in X\et^\afd} \dSt_{/U\alg} . \]
We can therefore apply the relative analytification functor
\[ (-)\an_X \colon \lim_{U \in X\et^\afd} \dSt_{/U\alg} \longrightarrow \dAnSt_{/X} , \]
and obtain an analytic deformation
\[ \widetilde{\mathfrak D}_X \longrightarrow X \times \mathbf P^1_k , \]
which enjoys the analytic version of all the properties given in \cref{prop:properties_algebraic_deformation}.

\section{Non-archimedean Gromov compactness} \label{sec:Gromov_compactness}

In this section, we review the compactness results from \cite{Yu_Gromov_compactness}.
Starting from this section until the end of the paper we will assume that the base field $k$ has discrete valuation and residue characteristic zero.

In order to obtain a proper moduli stack of stable maps, we need to equip the ambient space $X$ a Kähler structure.
Let us recall the definition of Kähler structure in non-archimedean geometry following \cite[\S 3]{Yu_Gromov_compactness}.

Let $X$ be a quasi-compact smooth rigid \kanal space, $\fX$ an snc formal model of $X$, and $\SX$ the associated Clemens polytope, i.e.\ the dual intersection complex of the special fiber $\fXs$.
Let $\{D_i\}_{i\in\IX}$ be the set of irreducible components of the reduction $(\fXs)_\red$ of $\fXs$, and $\mult_i$ the multiplicity of $D_i$ in $\fXs$ for every $i\in\IX$.
For every non-empty subset $I\subset\IX$, let $D_I\coloneqq\bigcap_{i\in I}D_i$.
Every $D_I$ is assumed to be either empty or geometrically irreducible.
Let $N^1(D_I,\R)$ denote the vector space of codimension-one cycles in $D_I$ with real coefficients modulo numerical equivalence (with respect to proper curves in $D_I$).

\begin{definition}
	A \emph{simple function} $\varphi$ on $\SX$ is a real valued function that is affine on every face of $\SX$.
	For $i\in\IX$, let $\varphi(i)$ denote the value of $\varphi$ at the vertex $i$ of $\SX$.
\end{definition}

\begin{definition}
	Let $\Delta^I$, $I\subset\IX$ be a face of $\SX$.
	For any simple function $\varphi$ on $\SX$, we define the \emph{derivative} of $\varphi$ with respect to $\Delta^I$ as
	\[\partial_I \varphi = \sum_{i\in\IX}\mult_i\cdot\varphi(i)\cdot [D_i]|_{D_I} \in N^1(D_I,\R) ,\]
	where $[D_i]$ denotes the divisor class of $D_i$.
	The function $\varphi$ is said to be \emph{linear} (resp.\ \emph{convex}, \emph{strictly convex}) along the open simplicial face $(\Delta^I)^\circ$ if $\partial_I \varphi$ is trivial (resp.\ nef, ample) in $N^1(D_I,\R)$.
\end{definition}

Let $\Lin_\fX$ (resp.\ $\Conv_\fX$, $\sConv_\fX$) denote the sheaf of linear (resp.\ convex, strictly convex) functions on $\SX$.

\begin{definition}
	A \emph{virtual line bundle} $L$ on the $k$-analytic space $X$ with respect to the formal model $\fX$ is a torsor over the sheaf $\Lin_\fX$.
	A \emph{strictly convex metrization} $\hL$ of a virtual line bundle $L$ is a global section of the sheaf $\sConv_\fX\otimes_{\Lin_\fX} L$, where $\Lin_\fX$ acts on $\sConv_\fX$ via additions $\psi\mapsto(\varphi\mapsto\varphi+\psi)$.
	Given $\hL$, we obtain a collection of numerical classes $\partial_i\varphi_i\in N^1(D_I,\R)$ for every $i\in\IX$, which we call the \emph{curvature} of $\hL$, and denote by $c(\hL)$.
\end{definition}

\begin{remark}
	It was pointed out by a referee that virtual line bundles may be related to logarithmic line bundles developed in \cite{Molcho_The_logarithmic_Picard_group}.
\end{remark}

\begin{definition} \label{def:kahler}
	A \emph{Kähler structure} $\hL$ on $X$ consists of a choice of an snc formal model $\fX$ of $X$, a virtual line bundle $L$ on $X$ with respect to $\fX$ equipped with a strictly convex metrization $\widehat L$.
\end{definition}

Now we can state the Gromov compactness theorem in non-archimedean geometry, one main result of \cite{Yu_Gromov_compactness}:

\begin{theorem} \label{thm:Gromov_compactness}
	Let $X$ be a proper smooth \kanal space equipped with a Kähler structure $\hL$.
	For any A-graph $(\tau,\beta)$, the (underived) stack $\oM(X,\tau,\beta)$ of $(\tau,\beta)$-marked stable maps into $X$ is a proper \kanal stack.
\end{theorem}
\begin{proof}
	This follows from \cite[Theorem 1.3]{Yu_Gromov_compactness} and the remark below.
\end{proof}

\begin{remark} \label{rem:degree}
	For any closed 1-dimensional irreducible reduced analytic subspace $C\subset X$, we define $C\cdot\hL$ to be the degree of the pullback of $\hL$ to the normalization $\tC$ of $C$ (see \cite[\S 5]{Yu_Gromov_compactness}).
	This extends to any 1-dimensional cycle $Z$ of $X$ by linearity, and factors through analytic equivalence, so $\beta\cdot \hL$ is well-defined for any $\beta\in A(X)$.
	Indeed, for any connected smooth $k$-affinoid curve $T$, $t,t'\in T$ two rigid points, and any closed 2-dimensional irreducible reduced analytic subspace $V$ of $X\times T$ flat over $T$, we take a formal model $\fT$ of $T$ and a formal model $\fV$ of $V$ flat over $\fT$.
	Let $t_s,t'_s\in\fT_s$ be the specializations of $t$ and $t'$ respectively.
	Flatness implies that $[\fV_{s,t_s}]\cdot c(\hL)=[\fV_{s,t'_s}]\cdot c(\hL)$.
	By \cite[Propositions 4.3 and 5.7]{Yu_Gromov_compactness}, we have $[\fV_{s,t_s}]\cdot c(\hL)=V_t\cdot\hL$ and $[\fV_{s,t'_s}]\cdot c(\hL)=V_{t'}\cdot\hL$, hence we obtain $V_t\cdot\hL=V_{t'}\cdot\hL$.
	
	Let $(\tau,\beta)$ be an A-graph.
	Without loss of generality, we assume it to be connected.
	Let $g$ be the genus of $\tau$ and $n$ the number of tails of $\tau$.
	Recall $\beta$ is a map $V_\tau\to A(X)$; let $\beta\cdot\hL\coloneqq\sum_{v\in V_\tau}\beta(v)\cdot\hL$.
	Let $\oM_{g,n}(X,\beta\cdot\hL)$ denote the stack of $n$-pointed genus $g$ \kanal stable maps into $X$ whose degree with respect to $\hL$ is at most $\beta\cdot\hL$  (as in \cite{Yu_Gromov_compactness}).
	We conclude that $\oM(X,\tau,\beta)$ is a closed substack of $\oM_{g,n}(X,\beta\cdot\hL)$.
\end{remark}

\begin{lemma} \label{lem:boundedness}
	Let $\ff \colon \fX \to \fY$ be a proper morphism of formal \DM stacks locally finitely presented over $k^\circ$.
	Let $f \colon X \to Y$ denote the induced map between the generic fibers.
	Then the higher direct image functor $f_*$ sends $\Cohb(X)$ to $\Cohb(Y)$.
\end{lemma}
\begin{proof}
	The question being étale local on the target, we can assume that $\fY$ is affine.
	Let $\abs{\fX_s}$ and $\abs{\fX}$ denote the coarse moduli spaces.
	Then the map $\ff\colon\fX\to\fY$ factors as
	$\fX \xrightarrow{\mathfrak p} \abs{\fX} \xrightarrow{\overline{\ff}} \fY$, so $f \colon X \to Y$ factors as $X \xrightarrow{p} \abs{\fX}_\eta \xrightarrow{\widebar{f}} Y$.
	By \cite[Proposition 3.6]{Abramovich_Tame_stacks}, étale locally over $\abs{\fX_s}$ the \DM stack $\fX_s$ is isomorphic to the quotient of an affine scheme by a finite group.
	Using the equivalence between the étale sites of $\abs{\fX_s} \simeq \abs{\fX}_s$ and of $\abs{\fX}$, we deduce that étale locally over $\abs{\fX}_\eta$, the \kanal \DM stack $X = \fX_\eta$ is isomorphic to the quotient of an affinoid space by a finite group.
	Since the base field is assumed to have characteristic zero, we deduce that the map $p \colon X \to \abs{\fX}_\eta$ has cohomological dimension zero.
		By \cite[\S B.2]{Conrad_Spreading-out}, the generic fiber $\abs{\fX}_\eta$ of the formal algebraic space $\abs{\fX}$ is a \kanal space.
	Since $\ff$ is proper, so are $\overline{\ff}$ and $\widebar f$.
	Therefore, by Kiehl's proper mapping theorem \cite{Kiehl_Der_Endlichkeitssatz}, the higher direct image functor $\widebar{f}_*$ sends $\Cohb(\abs{\fX_\eta})$ to $\Cohb(Y)$.
	As $f_* = \widebar{f}_* \circ p_*$, the conclusion follows.
\end{proof}

\begin{proposition} \label{prop:boundedness}
	Let $X$ be a proper smooth \kanal space equipped with a Kähler structure $\hL$.
	Let $(\tau,\beta)$ be an A-graph.
	Let $\st$ denote the composite map
	\[\st\colon\oM(X,\tau,\beta)\longrightarrow\oMpre_\tau\longrightarrow\oM_\tau,\]
	taking domain of stable map and then stabilizing the $\tau$-marked prestable curve.
	Then the higher direct image functor $\st_*$ sends $\Cohb(\oM(X,\tau,\beta))$ to $\Cohb(\oM_\tau)$.
\end{proposition}

\begin{proof}
	Let $g$ be the genus of $\tau$, $n$ the number of tails of $\tau$, and $A\coloneqq\beta\cdot\hL$.
	By \cref{rem:degree}, it suffices to prove the proposition for $\oM_{g,n}(X,A)$ instead of $\oM(X,\tau,\beta)$.
	By \cite[Theorem 8.9]{Yu_Gromov_compactness}, the moduli stack $\oM_{g,n}(\fX,A)$ of $n$-pointed genus $g$ formal stable maps into $\fX$ with degree bounded by $A$ is a formal model of the \kanal stack $\oM_{g,n}(X,A)$.
	Moreover, the moduli stack $\widebar{\fM}_{g,n}$ of $n$-pointed genus $g$ formal stable curves over $\kc$ is a formal model of $\oM_{g,n}$.
	Let
	\[ \mathfrak{st} \colon \oM_{g,n}(\fX,A) \longrightarrow \overline{\mathfrak M}_{g,n} \]
	be the map taking domain and then stabilizing.
	The generic fiber of $\mathfrak{st}$ is canonically equivalent to $\st$.
	Now the statement follows from \cite[Theorem 1.3]{Yu_Gromov_compactness} and \cref{lem:boundedness}.
\end{proof}

\section{Quantum K-invariants} \label{sec:quantum_K-invariants}

In this section, we introduce the quantum K-invariants in non-archimedean analytic geometry, and prove a list of natural geometric properties, namely, mapping to a point, products, cutting edges, forgetting tails and contracting edges.
They are analogous to the axioms in algebraic quantum K-theory in the work of Lee \cite{Lee_Quantum_K-theory_I}, and to the Behrend-Manin axioms \cite{Behrend_Stacks_of_stable_maps} for Gromov-Witten invariants.
The proof of the mapping to a point property relies on the theory of deformation to the  normal bundle that we studied in \cref{sec:deformation_to_the_normal_bundle}.
The remaining properties are consequences of the corresponding geometric relations of the derived moduli stacks established in \cref{sec:geometry_of_derived_stable_maps}.

Fix a proper smooth \kanal space $X$.
For any $A$-graph $(\tau,\beta)$, let $\R\oM(X,\tau,\beta)$ denote the derived stack of $(\tau,\beta)$-marked stable maps into $X$ as in \cref{rem:RMXtaubeta}.
It is a proper derived \kanal stack by \cref{thm:Gromov_compactness}.
Let $\st$ denote the composite map
\[\st\colon\R\oM(X,\tau,\beta)\longrightarrow\oMpre_\tau\longrightarrow\oM_\tau,\]
taking domain of stable map and then stabilizing the $\tau$-marked prestable curve.

For every tail $i\in T_\tau$, we have a section $s_i$ of the universal curve
\[\begin{tikzcd}
\R\oC(X,\tau,\beta) \arrow{r}{f} \arrow{d} & X\\
\R\oM(X,\tau,\beta). \arrow[bend left]{u}{s_i} & \\
\end{tikzcd}\]
Composing with the universal map $f$, we obtain the evaluation map of the $i$-th marked point
\[\ev_i\colon\R\oM(X,\tau,\beta)\longrightarrow X.\]

\begin{definition} \label{def:quantum_K-invariants}
	The \emph{quantum K-invariants} of $X$ are the collection of linear maps
	\begin{align*}
		K^X_{\tau,\beta}\colon K_0(X)^{\otimes \abs{T_\tau}} &\longrightarrow K_0(\oM_\tau)\\
		\bigotimes_{i\in T_\tau} a_i & \longmapsto \st_*\bigg(\bigotimes_{i\in T_\tau} \ev_i^*(a_i)\bigg),
	\end{align*}
	where $(\tau,\beta)$ is any A-graph.
	\Cref{prop:boundedness} ensures that the image of $\st_*$ lies in $K_0(\oM_\tau)$. 
\end{definition}

\subsection{Mapping to a point} \label{sec:mapping_to_a_point}

\begin{proposition} \label{prop:mapping_to_a_point}
	Let $(\tau,0)$ be any A-graph where $\beta$ is 0.
	Let $p_1, p_2$ be the projections
	\[ \begin{tikzcd}
	X\times \oM_\tau \arrow{r}{p_1} \arrow{d}{p_2} & X \\
	\oM_\tau & \phantom{X} .
	\end{tikzcd} \]
	For any $a_i\in K_0(X), i\in T_\tau$, we have
	\[K^X_{\tau,0}\big({\textstyle\bigotimes_i} a_i \big) = p_{2*} \Big(p_1^*({\textstyle\bigotimes_i} a_i) \otimes \lambda_{-1}\big( (\bbT\an_X\boxtimes\rR^1\pi_*\cO_{\oC_\tau})^\vee\big)\Big),\]
	where $\pi\colon\oC_\tau\to\oM_\tau$ and $\lambda_{-1}(F)\coloneqq\sum_{i}(-1)^i\wedge^i F$.
\end{proposition}

Let us introduce some notations before the proof.
Consider the universal \emph{underived} $\tau$-marked stable map
\[ \begin{tikzcd}
	\oC(X, \tau, 0) \arrow{d} \arrow{r} & X \\
	\oM(X, \tau, 0) \arrow[dashed]{ur} & \phantom{X}.
\end{tikzcd} \]
Note $\beta=0$ implies that every stable map above is constant, so the dashed arrow uniquely exists.
Combining with $\st \colon \oM(X, \tau, 0) \to \oM_\tau$, we obtain an isomorphism
\[ \oM(X, \tau, 0) \xrightarrow{\ \sim\ } X\times\oM_\tau .\]
Similarly, we have $\oC(X, \tau, 0) \simeq X \times \oC_\tau$.
Now we consider their derived enhancements:
\begin{equation} \label{eq:mapping_to_a_point}
	\begin{tikzcd}
		\oC_\tau \arrow{d}{\pi} & X \times \oC_\tau \arrow[bend left = 16]{rr}{q_1} \arrow{l}[swap]{q_2} \arrow{d}{p} \arrow[hook]{r}{i} & \R \oC(X, \tau, 0) \arrow{r} & X \\
		\oM_\tau & X \times \oM_\tau \arrow{l}{p_2} \arrow[hook]{r}[swap]{j} \arrow[bend left = 0]{urr}[near start]{p_1} & \R \oM(X, \tau, 0) \arrow[leftarrow, crossing over]{u} & \phantom{X} .
	\end{tikzcd}
\end{equation}
Note that contrary to the underived situation, in general, there is no map from $\R\oM(X,\tau,0)$ to $X$ making the diagram commutative.

\begin{lemma} \label{lem:mapping_to_a_point}
	Let $j \colon \oM(X, \tau, 0) \hookrightarrow \R \oM(X, \tau, 0)$ be the canonical inclusion of the truncation.
	There is a canonical equivalence
	\[ \anL_j [-1] \simeq \big( \rR^1 \pi_* \cO_{\oC_\tau} \boxtimes \bbT\an_X \big)^\vee[1] . \]
\end{lemma}

\begin{proof}
	Consider the fiber sequence
	\[ \bbT\an_j \longrightarrow \bbT\an_{\oM(X,\tau,0)} \longrightarrow j^* \bbT\an_{\R \oM(X,\tau, 0)} . \]
	As in \cref{thm:stack_of_stable_maps_derived_lci}, we have
	\[ j^* \bbT\an_{\R \oM(X, \tau, 0)} \simeq p_* \Big( \cofib \big( \bbT\an_{\oC(X,\tau,0)/\oM(X,\tau,0)}(-{\textstyle \sum} s_i) \to q_1^* \bbT\an_{X/S} \big) \Big) . \]
	Since $\oC(X, \tau, 0) \simeq X \times \oC_\tau $ and $\oM(X, \tau, 0) \simeq X \times \oM_\tau$, we obtain
	\[ p_* \big( \bbT\an_{\oC(X,\tau,0) / \oM(X,\tau, 0)}( -{\textstyle \sum} s_i ) \big) \simeq p_2^* \pi_* \bbT\an_{\oC_\tau / \oM_\tau}( - {\textstyle \sum} s_i ) . \]
	Moreover, the canonical map
	\[ p_2^* \pi_* \bbT\an_{\oC_\tau / \oM_\tau}( - {\textstyle \sum} s_i ) \longrightarrow p_* q_1^* \bbT\an_X \]
	is nullhomotopic,  and therefore
	\[ j^* \bbT\an_{\R \oM(X, \tau, 0)} \simeq  p_* q_1^* \bbT\an_X \oplus p_2^* \pi_* \bbT\an_{\oC_\tau / \oM_\tau}[1]. \]
	On the other hand, using the equivalence $\oM(X, \tau, 0) \simeq \oM_\tau \times X$, we find:
	\[ \bbT\an_{\oM(X,\tau,0)} \simeq p_1^* \bbT\an_X \oplus p_2^* \bbT\an_{\oM_\tau}\simeq p_1^* \bbT\an_X \oplus p_2^* \pi_* \bbT\an_{\oC_\tau / \oM_\tau}[1] . \]
	Using the projection formula, we obtain:
	\begin{align*}
		\bbT\an_j & \simeq \fib\big( p_1^* \bbT\an_X \to p_* q_1^* \bbT\an_X \big) \\
		& \simeq \fib\big( p_1^* \bbT\an_X \to p_2^* p_1^* \bbT\an_X \otimes \pi_* \cO_{\oC_\tau} \big) \\
		& \simeq \fib\big( p_1^* \bbT\an_X \to \bbT\an_X\boxtimes\pi_* \cO_{\oC_\tau}  \big) .
	\end{align*}
	The conclusion follows.
	\end{proof}

\begin{proof}[Proof of \cref{prop:mapping_to_a_point}]
	Notation as in diagram \eqref{eq:mapping_to_a_point}, we observe that
	\[ \begin{tikzcd}
		X \times \oM_\tau \arrow[hook]{r}{j} \arrow{d}[swap]{p_2} & \R \oM(X, \tau, 0) \arrow{dl}{\st} \\
		\oM_\tau
	\end{tikzcd} \]
	commutes.
	Let $\cF_i \in \Perf(X)$, $i \in T_\tau$ be representatives for the classes $a_i \in K_0(X)$.
	\cref{lem:mapping_to_a_point} implies the following equality in $K_0(X \times \oM_\tau)$:
	\[ \lambda_{-1}\big( (\bbT\an_ X \boxtimes \rR^1 \pi_* \cO_{\oC_\tau})^\vee \big) = \big[ \Sym_{\cO_{\oM_\tau \times X}} \big( \anL_j[-1] \big) \big] . \]
	It is therefore enough to check that
	\[ \Big[ p_{2*} \Big(p_1^* \big( {\textstyle \bigotimes_i} \cF_i \big)  \otimes \Sym_{\cO_{\oM_\tau \times X}}\big( \anL_j[-1] \big)\Big) \Big] = \big[ \st_*\big( {\textstyle \bigotimes_i} \ev_i^* \cF_i \big) \big] \]
	in $K_0(\oM_\tau)$.
		Let $\mathfrak M(X,\tau,0)$ be the $\bP^1_k$-deformation to the normal bundle associated to $j \colon X \times \oM_\tau \to \R\oM(X,\tau,0)$ we constructed in \cref{sec:deformation_analytic}.
	The functoriality of this construction allows to deform also the maps $\ev_i \colon \R\oM(X, \tau, 0) \to X$ and $\st \colon \R\oM(X, \tau, 0) \to \oM_\tau$.
	Since the targets of these maps are underived, their associated deformations are constant.
	Thus we obtain maps
	\[ \widetilde{\ev}_i \colon \mathfrak M(X, \tau, 0) \longrightarrow X\times \bP^1_k , \qquad \widetilde{\st} \colon \mathfrak M(X, \tau, 0) \longrightarrow \oM_\tau \times \bP^1_k . \]
	\cref{prop:properties_algebraic_deformation}(\ref{prop:properties_algebraic_deformation:trivialization}) implies that the fibers $\widetilde{\ev}_i |_1$ and $\widetilde{\st} |_1$ at $1 \in \bP^1_k$ of $\widetilde{\ev}_i$ and $\widetilde{\st}$ coincide with the maps $\ev_i$ and $\st$ we started from.
	On the other hand, \cref{prop:properties_algebraic_deformation}(\ref{prop:properties_algebraic_deformation:central_fiber}) shows that the fibers $\widetilde{\ev}_i |_0$ and $\widetilde{\st}|_0$ at $0 \in \bP^1_k$ of $\widetilde{\ev}_i$ and $\widetilde{\st}$ fit in the following commutative diagram:
	\begin{equation} \label{eq:mapping_to_a_point_central_fiber}
		\begin{tikzcd}
			{} & \Spec_{\oM_\tau \times X}( \Sym_{\cO_{\oM_\tau \times X}}( \anL_j[-1]) ) \arrow{d}{r} \arrow{dl}[swap]{\widetilde{\st}|_0} \arrow{dr}{\widetilde{\ev}_i |_0} \\
			\oM_\tau & X \times \oM_\tau \arrow{l}[swap]{p_2} \arrow{r}{p_1} & X ,
		\end{tikzcd}
	\end{equation}
	where the vertical map is the natural projection.
	Let $q \colon X \times \bP^1_k \to \bP^1_k$ be the projection and consider the perfect complex
	\[ \cG \coloneqq \widetilde{\st}_*\big( {\textstyle \bigotimes_i} \widetilde{\ev}_i^*( q^*(\cF_i) ) \big) \in \Perf(\oM_\tau \times \bP^1_k) . \]
	Since $\oM_\tau \times \bP^1_k$ is algebraic and proper, the GAGA theorem (\cite[Theorem 1.3]{Porta_Yu_Higher_analytic_stacks}) implies that $\cG$ is an algebraic perfect complex.
	In particular, its restriction to $\oM_\tau \times \bA^1_k$ is also algebraic; and therefore the $\mathbb A^1$-invariance for $K$-theory (see \cite[Theorem 6.13]{Weibel_The_K-book}) implies that
	\[ \big[ \cG |_{\oM_\tau \times 0} \big] = \big[ \cG |_{\oM_\tau \times 1} \big] \]
	in $K_0( \oM_\tau )$.
		By proper base change (\cite[Theorem 1.5]{Porta_Yu_Derived_Hom_spaces}), we have
	\[ \cG |_{\oM_\tau \times 1} \simeq \st_*\big( {\textstyle \bigotimes_i} \ev_i^*(\cF_i) \big) . \]
	On the other hand, diagram \eqref{eq:mapping_to_a_point_central_fiber} together with proper base change implies that
	\begin{align*}
		\cG |_{\oM_\tau \times 0} & \simeq p_{2*} r_* r^* p_1^* \big( {\textstyle \bigotimes_i} \cF_i \big) \\
		& \simeq p_{2*} \Big( p_1^* \big( {\textstyle \bigotimes_i} \cF_i \big) \otimes \Sym_{\cO_{\oM_\tau \times X}} \big( \anL_j[-1] \big)  \Big).
	\end{align*}
	Combining the isomorphisms and the equality above, we achieve the proof.
\end{proof}

\subsection{Products}

\begin{proposition} \label{prop:products}
	Let $(\tau_1, \beta_1)$ and $(\tau_2, \beta_2)$ be two A-graphs.
	Let $(\tau_1 \sqcup \tau_2, \beta_1 \sqcup \beta_2)$ be their disjoint union.
	For any $a_i\in K_0(X), i\in T_{\tau_1}$ and $b_j\in K_0(X), j\in T_{\tau_2}$, we have
	\[K^X_{\tau_1\sqcup\tau_2,\beta_1\sqcup\beta_2}\big(({\textstyle\bigotimes_i} a_i)\otimes({\textstyle\bigotimes_j} b_j)\big)=K^X_{\tau_1,\beta_1}\big({\textstyle\bigotimes_i} a_i\big)\boxtimes K^X_{\tau_2,\beta_2}\big({\textstyle\bigotimes_j} b_j\big).\]
\end{proposition}
\begin{proof}
	For $i\in T_{\tau_1}$, $j\in T_{\tau_2}$, we denote the evaluation maps by
	\begin{align*}
		\ev^1_i&\colon\R\oM(X,\tau_1,\beta_1)\longrightarrow X,\\
		\ev^2_j&\colon\R\oM(X,\tau_2,\beta_2)\longrightarrow X,\\
		\ev^{12}_i, \ev^{12}_j &\colon\R\oM(X,\tau_1\sqcup\tau_2,\beta_1\sqcup\beta_2)\longrightarrow X.
	\end{align*}
	By \cref{thm:products}, we have the following commutative diagram
	\[\begin{tikzcd}
		\R\oM(X,\tau_1,\beta_1)\times\R\oM(X,\tau_2,\beta_2) \rar{\sim}\dar{\st_1\times\st_2} & \R\oM(X,\tau_1\sqcup\tau_2,\beta_1\sqcup\beta_2) \rar{\ev^{12}_i,\ev^{12}_j} \dar{\st_{12}} & X\\
		\oM_{\tau_1}\times\oM_{\tau_2} \rar{\sim} & \oM_{\tau_1\sqcup\tau_2}.
	\end{tikzcd}\]
	Let
	\begin{align*}
		& p_1\colon\R\oM(X,\tau_1\sqcup\tau_2,\beta_1\sqcup\beta_2)\longrightarrow\R\oM(X,\tau_1,\beta_1),\\
		& p_2\colon\R\oM(X,\tau_1\sqcup\tau_2,\beta_1\sqcup\beta_2)\longrightarrow\R\oM(X,\tau_2,\beta_2)
	\end{align*}
	denote the projection maps.
	For any $i\in T_{\tau_1}$, $j\in T_{\tau_2}$, we have $\ev^{12}_i=\ev^1_i\circ p_1$, $\ev^{12}_j=\ev^2_j\circ p_2$.
	Hence we obtain
	\begin{align*}
		K^X_{\tau_1\sqcup\tau_2,\beta_1\sqcup\beta_2}\big(({\textstyle\bigotimes_i} a_i)\otimes({\textstyle\bigotimes_j} b_j)\big) &= \st_{12*}\big(({\textstyle\bigotimes_i} \ev^{12*}_i a_i)\otimes({\textstyle\bigotimes_j} \ev^{12*}_j b_j)\big) \\
		&= (\st_1\times\st_2)_*\big(({\textstyle\bigotimes_i} \ev^{1*}_i a_i)\boxtimes({\textstyle\bigotimes_j} \ev^{2*}_j b_j)\big)\\
		&= \st_{1*}\big({\textstyle\bigotimes_i} \ev^{1*}_i a_i\big)\boxtimes\st_{2*}\big({\textstyle\bigotimes_j} \ev^{2*}_j b_j\big)\\
		&= K^X_{\tau_1,\beta_1}\big({\textstyle\bigotimes_i} a_i\big)\boxtimes K^X_{\tau_2,\beta_2}\big({\textstyle\bigotimes_j} b_j\big),
	\end{align*}
	completing the proof.
\end{proof}

\subsection{Cutting edges}

\begin{proposition} \label{prop:cutting_edges}
	Let $(\sigma,\beta)$ be an A-graph obtained from $(\tau,\beta)$ by cutting an edge $e$ of $\tau$.
	Let $v,w$ be the two tails of $\sigma$ created by the cut.
	Consider the following commutative diagram
	\[ \begin{tikzcd}
	\oM_\tau \rar{d} & \oM_\sigma\\
	\R \oM(X, \tau, \beta) \uar{\st_\tau} \arrow{r}{c} \arrow{d}{\ev_e} & \R \oM(X, \sigma, \beta) \uar{\st_\sigma} \arrow{d}{\ev_v \times \ev_w} \\
	X \arrow{r}{\Delta} & X \times_S X.
	\end{tikzcd} \]
	For any $a_i\in K_0(X), i\in T_\tau$, we have
	\[d_* K^X_{\tau,\beta}\big({\textstyle\bigotimes_i} a_i\big)=K^X_{\sigma,\beta}\big(({\textstyle\bigotimes_i} a_i)\otimes\Delta_*\cO_X\big).\]
	\end{proposition}

\begin{proof}
	For each $i\in T_\tau$, we denote the evaluation maps by
	\begin{align*}
	\ev^\tau_i&\colon\R\oM(X,\tau,\beta)\longrightarrow X,\\
	\ev^\sigma_i&\colon\R\oM(X,\sigma,\beta)\longrightarrow X.
	\end{align*}
	We have $\ev^\tau_i=\ev^\sigma_i\circ c$.
	By \cref{thm:cutting_edges_derived_pullback}, the lower square of the commutative diagram in the statement is a derived pullback.
	Hence we obtain
	\begin{align*}
		d_* K^X_{\tau,\beta}\big({\textstyle\bigotimes_i} a_i\big) &= d_*\st_{\tau*} \big({\textstyle\bigotimes_i}\ev^{\tau*}_i a_i\big)\\
		&= d_*\st_{\tau*} \big(({\textstyle\bigotimes_i}\ev^{\tau*}_i a_i)\otimes\ev^*_e\cO_X\big)\\
		&= \st_{\sigma*} c_* \big(c^*({\textstyle\bigotimes_i}\ev^{\sigma*}_i a_i)\otimes\ev^*_e\cO_X\big)\\
		&= \st_{\sigma*} \big(({\textstyle\bigotimes_i}\ev^{\sigma*}_i a_i)\otimes c_*\ev^*_e\cO_X\big)\\
		&= \st_{\sigma*} \big(({\textstyle\bigotimes_i}\ev^{\sigma*}_i a_i)\otimes (\ev_v\times\ev_w)^*\Delta_*\cO_X\big)\\
		&= K^X_{\sigma,\beta}\big(({\textstyle\bigotimes_i} a_i)\otimes\Delta_*\cO_X\big),
	\end{align*}
	where the fourth equality follows from the projection formula \cite[Theorem 1.4]{Porta_Yu_Derived_Hom_spaces}, and the fifth equality follows from proper base change \cite[Theorem 1.5]{Porta_Yu_Derived_Hom_spaces}.
\end{proof}

\subsection{Forgetting tails}

\begin{proposition} \label{prop:forgetting_tails}
	Let $(\sigma,\beta)$ be an A-graph obtained from $(\tau,\beta)$ by forgetting a tail.
	Let $\Phi\colon\oM_\tau\to\oM_\sigma$ be the forgetful map from $\tau$-marked stable curves to $\sigma$-marked stable curves.
	For any $a_i\in K_0(X), i\in T_\sigma$, we have
	\[K^X_{\tau,\beta}\big(({\textstyle\bigotimes_i} a_i)\otimes \cO_X\big)=\Phi^*K^X_{\sigma,\beta}\big({\textstyle\bigotimes_i} a_i\big).\]
\end{proposition}
\begin{proof}
	Consider the commutative diagram as in the proof of \cref{thm:forgetting_tails}:
	\[\begin{tikzcd}
	\R \oM(X, \tau, \beta) \arrow{r}{\Psi} \arrow{d}{\lambda} \arrow[bend left=15]{rr}{\pi} \arrow{rdd}[near end, swap]{\st_\tau} & \oM_\tau \times_{\oM_\sigma}  \R \oM(X, \sigma, \beta) \arrow{r}{\eta} \arrow{d}{\mu} & \R \oM(X, \sigma, \beta) \arrow{d}{\kappa} \arrow[bend left=40]{dd}{\st_\sigma} \\
	\oCpre_w \arrow[crossing over]{r}[near start]{\rho} & \oM_\tau \times_{\oM_\sigma} \oMpre_\sigma \arrow{d}{\xi} \arrow{r} & \oMpre_\sigma \arrow{d}{\theta} \\
	{} & \oM_\tau \arrow{r}{\Phi} & \oM_\sigma .
	\end{tikzcd} \]
	For each $i\in T_\sigma$, we denote the evaluation maps by
	\begin{align*}
	\ev^\tau_i&\colon\R\oM(X,\tau,\beta)\longrightarrow X,\\
	\ev^\sigma_i&\colon\R\oM(X,\sigma,\beta)\longrightarrow X.
	\end{align*}
	We have $\ev^\tau_i=\ev^\sigma_i\circ\pi$.
	By \cref{thm:forgetting_tails}, all the rectangles in the commutative diagram above are pullbacks; moreover, we have an equivalence
	\begin{equation} \label{eq:forgetting_tail_equivalence}
		\cO\alg_{\oM_\tau \times_{\oM_\sigma}  \R \oM(X, \sigma, \beta)}\xrightarrow{\ \sim\ }\Psi_*\big(\cO\alg_{\R\oM(X,\tau,\beta)}\big).
	\end{equation}
	Let $t\in T_\tau$ denote the forgotten tail.
	We obtain
	\begin{align*}
		K^X_{\tau,\beta}\big(({\textstyle\bigotimes_i} a_i)\otimes \cO_X\big) 
		&= \st_{\tau*} \big(({\textstyle\bigotimes_i}\ev^{\tau*}_i a_i)\otimes\ev_t^*\cO_X\big)\\
		&= \xi_*\mu_*\Psi_*\big(\Psi^*\eta^*({\textstyle\bigotimes_i}\ev^{\sigma*}_i a_i)\otimes\cO\alg_{\R\oM(X,\tau,\beta)}\big)\\
		&= \xi_*\mu_*\big(\eta^*({\textstyle\bigotimes_i}\ev^{\sigma*}_i a_i)\otimes\Psi_*\big(\cO\alg_{\R\oM(X,\tau,\beta)}\big)\big)\\
		&= \xi_*\mu_*\eta^*\big({\textstyle\bigotimes_i}\ev^{\sigma*}_i a_i\big)\\
		&= \Phi^*\theta_*\kappa_*\big({\textstyle\bigotimes_i}\ev^{\sigma*}_i a_i\big)\\
		&= \Phi^*\st_{\sigma*}\big({\textstyle\bigotimes_i}\ev^{\sigma*}_i a_i\big)\\
		&= \Phi^*K^X_{\sigma,\beta}\big({\textstyle\bigotimes_i} a_i\big),
	\end{align*}
		where the third equality follows from the projection formula \cite[Theorem 1.4]{Porta_Yu_Derived_Hom_spaces}, the fourth equality follows from \eqref{eq:forgetting_tail_equivalence}, and the fifth equality follows from proper base change \cite[Theorem 1.5]{Porta_Yu_Derived_Hom_spaces}.
\end{proof}

\begin{remark} \label{rem:forgetting_tails}
	This property is called ``fundamental class'' in \cite[\S 3.5, \S 4.3]{Lee_Quantum_K-theory_I} because the name originates from \cite[\S 2.2.3]{Kontsevich_Gromov-Witten_classes}.
	We chose to call it ``forgetting tails'' because it corresponds to the operation of forgetting a tail.
\end{remark}

\subsection{Contracting edges}

\begin{proposition} \label{prop:contracting_edges}
	Let $(\sigma,\beta)$ be an A-graph where $\sigma$ is obtained from a modular graph $\tau$ by contracting an edge (possibly a loop) $e$.
	We follow the notations of \cref{sec:contracting_edges}.
	Let $\Phi\colon\oM_\tau\to\oM_\sigma$, $\Psi_l^i\colon\oM_{\tau_l^i}\to\oM_\tau$ and $\Omega\colon\oM_{\tsigma}\to\oM_\sigma$ be the induced maps on the moduli stacks of stable curves.
	For any $a_v\in K_0(X), v\in T_{\tsigma}$, we have
	\[\Phi^*\Omega_* K^X_{\tsigma,\beta}\big({\textstyle\bigotimes_v} a_v\big) = \sum_l (-1)^{l+1} \sum_{i,j} \Psi_{l*}^i K^X_{\tau^i_l,\beta^{i_j}_l}\big({\textstyle\bigotimes_v} a_v\big).\]
\end{proposition}
\begin{proof}
	For every $l,i,j$ consider the following commutative diagram
	\[\begin{tikzcd}[row sep=small]
		\R\oM(X,\tau_l^i,\beta_l^{i_j}) \arrow{dd}{\st_l^{i_j}} \arrow[bend left=25]{rr}[swap]{q_l^{i_j}} \rar{p_l^{i_j}} & \R\oM(X,\tsigma,\beta)\times_{\oM_\sigma}\oM_\tau \rar{\lambda} \arrow{dd}{\mu} & \R\oM(X,\tsigma,\beta) \dar{\st_{\tsigma}} \\
		& & \oM_{\tsigma} \dar{\Omega} \\
		\oM_{\tau_l^i} \rar{\Psi_l^i} & \oM_\tau \rar{\Phi} & \oM_\sigma.
	\end{tikzcd}\]
	For each $v\in T_{\tsigma}$, we denote by $v$ again the corresponding tail of $\tau_l^i$, and we denote the evaluation maps by
	\begin{align*}
	\ev^{l,i_j}_v&\colon\R\oM(X,\tau_l^i,\beta_l^{i_j})\longrightarrow X,\\
	\ev^{\tsigma}_v&\colon\R\oM(X,\tsigma,\beta)\longrightarrow X.
	\end{align*}
	We have $\ev^{l,i_j}_v=\ev^{\tsigma}_v\circ q_l^{i_j}$.
	By \cref{thm:contracting_edges}, we have
	\begin{equation} \label{eq:contracting_edges_colimit}
		\sum_l (-1)^{l+1} \sum_{i,j} p_{l*}^{i_j}\big[\cO_{\R\oM(X,\tau^i_l,\beta^{i_j}_l)}\big]=\big[\cO_{\oM_\tau\times_{\oM_\sigma}\R\oM(X,\tsigma,\beta)}\big]
	\end{equation}
	in $K_0\big(\Cohb\big( \oM_\tau\times_{\oM_\sigma}\R\oM(X,\tsigma,\beta) \big) \big)$.
	We obtain
	\begin{align*}
		\Phi^*\Omega_* K^X_{\tsigma,\beta}\big({\textstyle\bigotimes_v} a_v\big)
		&= \Phi^*\Omega_* \st_{\tsigma*}\big({\textstyle\bigotimes_v} \ev_v^{\tsigma*} a_v\big) \\
		&= \mu_* \lambda^*\big({\textstyle\bigotimes_v} \ev_v^{\tsigma*} a_v\big) \\
		&= \mu_*\Big(\lambda^*\big({\textstyle\bigotimes_v} \ev_v^{\tsigma*} a_v\big) \otimes \sum_l (-1)^{l+1} \sum_{i,j} p_{l*}^{i_j}\big[\cO_{\R\oM(X,\tau^i_l,\beta^{i_j}_l)}\big]\Big) \\
		&= \sum_l(-1)^{l+1} \sum_{i,j}\mu_*\Big(\lambda^*\big({\textstyle\bigotimes_v} \ev_v^{\tsigma*} a_v\big) \otimes p_{l*}^{i_j}\big[\cO_{\R\oM(X,\tau^i_l,\beta^{i_j}_l)}\big]\Big) \\
		&= \sum_l(-1)^{l+1} \sum_{i,j}\mu_*p_{l*}^{i_j}\Big(p_l^{i_j*}\lambda^*\big({\textstyle\bigotimes_v} \ev_v^{\tsigma*} a_v\big) \otimes \big[\cO_{\R\oM(X,\tau^i_l,\beta^{i_j}_l)}\big]\Big) \\
		&= \sum_l(-1)^{l+1} \sum_{i,j}\Psi_{l*}^i\st_{l*}^{i_j}q_l^{i_j*}\big({\textstyle\bigotimes_v} \ev_v^{\tsigma*} a_v\big) \\
		&= \sum_l(-1)^{l+1} \sum_{i,j}\Psi_{l*}^i\st_{l*}^{i_j}\big({\textstyle\bigotimes_v} \ev_v^{l,i_j*} a_v\big) \\
		&=\sum_l(-1)^{l+1} \sum_{i,j}\Psi_{l*}^i K^X_{\tau^i_l,\beta^{i_j}_l}\big({\textstyle\bigotimes_v} a_v\big),
	\end{align*}
	where the second equality follows from proper base change \cite[Theorem 1.5]{Porta_Yu_Derived_Hom_spaces}, the third equality follows from \eqref{eq:contracting_edges_colimit}, and the fifth equality follows from the projection formula \cite[Theorem 1.4]{Porta_Yu_Derived_Hom_spaces}.
\end{proof}

\section{Quantum K-invariants with multiplicities} \label{sec:multiplicities}

The flexibility of our derived approach to quantum K-invariants allows us to impose not only simple incidence conditions for marked points as in the last section, but also incidence conditions with multiplicities.
This leads to a new set of enumerative invariants, which is not yet considered in the literature, even in algebraic geometry, to the best of our knowledge.

Fix a proper smooth \kanal space $X$ and an $A$-graph $(\tau,\beta)$.
Let $i\in T_\tau$ be a tail vertex and $m_i$ a positive integer.
Let $s_i\colon\oMpre_\tau\to\oCpre_\tau$ be the $i$-th section of the universal curve, $\cI_i$ the ideal sheaf on $\oCpre_\tau$ cutting out the image of $s_i$, and $(\oCpre_\tau)_{(s_i^{m_i})}$ the closed substack given by the $m_i$-th power of the ideal $\cI_i$.
Let $X_{i,\tau}^{m_i}$ denote the mapping stack
\[\bfMap_{\oMpre_\tau}\Bigl((\oCpre_\tau)_{(s_i^{m_i})}, X\times \oMpre_\tau\Bigr).\]
Let $\ev_i^{m_i}$ denote the composition of the inclusion
\[\R\oM(X,\tau,\beta) \longrightarrow \bfMap_{\oMpre_\tau}\bigl(\oCpre_\tau, X\times\oMpre_\tau\bigr)\]
and the restriction
\[\bfMap_{\oMpre_\tau}\bigl(\oCpre_\tau, X\times\oMpre_\tau\bigr)\longrightarrow X_{i,\tau}^{m_i}.\]
We call
\[\ev_i^{m_i}\colon\R\oM(X,\tau,\beta)\longrightarrow X_{i,\tau}^{m_i}\]
the evaluation map of the $i$-th marked point of order $m_i$.

Note that when $m_i=1$, we have
\[ (\oCpre_\tau)_{(s_i^{m_i})}\simeq\oMpre_\tau \quad\text{and}\quad X_{i,\tau}^{m_i}\simeq X\times\oMpre_\tau, \]
hence $\ev^1_i \simeq \ev_i\times\dom$, where $\dom\colon\R\oM(X,\tau,\beta)\to\oMpre_\tau$ denotes the map taking domains of stable maps.

Given any closed analytic subspace $Z_i\subset X$ such that the inclusion is lci, let
\[Z_{i,\tau}^{m_i}\coloneqq\bfMap_{\oMpre_\tau}\Bigl((\oCpre_\tau)_{(s_i^{m_i})}, Z_i\times \oMpre_\tau\Bigr).\]

\begin{lemma} \label{lem:Zi}
	The derived \kanal stack $X_{i,\tau}^{m_i}$ is smooth over $\oMpre_\tau$, in particular underived.
	The map $\iota\colon Z_{i,\tau}^{m_i}\to X_{i,\tau}^{m_i}$ is derived lci.
	The pushforward $\iota_* \cO_{Z_{i,\tau}^{m_i}}$ is a perfect complex on $X_{i,\tau}^{m_i}$.
\end{lemma}
\begin{proof}
	Note that all the statements are local over $\oMpre_\tau$.
	Let $J^{m_i}\coloneqq\Sp k\langle t\rangle/(t^{m_i})$.
	Locally over $\oMpre_\tau$, $(\oCpre_\tau)_{(s_i^{m_i})}$ is of the form $J^{m_i}\times S$ for some $S$ smooth over $\oMpre_\tau$.
	Then the map $\iota\colon Z_{i,\tau}^{m_i}\to X_{i,\tau}^{m_i}$ is of the form
	\[\bfMap(J^{m_i}, Z_i)\times S \longrightarrow \bfMap(J^{m_i}, X)\times S.\]
	Consider the commutative diagram
	\[ \begin{tikzcd}
		\bfMap(J^{m_i}, Z_i) \arrow{d}{\theta} & J^{m_i} \times \bfMap(J^{m_i}, Z_i) \arrow{r}{e'} \arrow{l}[swap]{q'} \arrow{d} & Z_i \arrow{d} \\
		\bfMap(J^{m_i}, X) & J^{m_i} \times \bfMap(J^{m_i}, X) \arrow{r}{e} \arrow{l}[swap]{q} & X,
	\end{tikzcd} \]
	where $e$ and $e'$ are the evaluation maps.
	It follows from \cite[Lemma 8.4]{Porta_Yu_Derived_Hom_spaces} that
	\[\bbT\an_{\bfMap(J^{m_i}, X)} \simeq q_* e^* \bbT\an_X.\]
	Since $X$ is smooth, $\bbT\an_X$ is perfect and has tor-amplitude 0, and the same goes for $e^*\bbT\an_X$.
	Note that $q$ has relative dimension $0$.
	Furthermore, $q$ is flat by \cref{cor:stability_flatness}.
	Therefore, the functor $q_*$ has cohomological dimension $0$ and preserves the tor-amplitude.
	As a result, $q_* e^* \bbT\an_X$ is perfect and has tor-amplitude 0.
	Since $X$ is smooth, the truncation $\trunc\bfMap(J^{m_i}, X)$ is also smooth.
	It follows from \cite[Proposition 5.50]{Porta_Yu_Representability_theorem} that $\bfMap(J^{m_i}, X)$ is smooth.
	Therefore, $X_{i,\tau}^{m_i}$ is smooth over $\oMpre_\tau$, in particular underived.
	
	Similarly we have
	\[ \bbT\an_\theta \simeq q'_\ast e'^\ast \bbT\an_{Z_i / X} . \]
	Since $Z_i \to X$ is lci, $\mathbb T\an_{Z_i/X}$ is perfect and has tor-amplitude $[-1,0]$, and thus the same goes for $e^* \mathbb T\an_{Z_i/X}$.
	Similar to $q$, the map $q'$ has cohomological dimension $0$ and preserves the tor-amplitude.
	We conclude that $\mathbb T\an_\theta$ is perfect and has tor-amplitude $[-1, 0]$, and therefore $\theta$ is derived lci.
	Since being derived lci is stable under products, we deduce that the map $\iota\colon Z_{i,\tau}^{m_i}\to X_{i,\tau}^{m_i}$ is derived lci.
	Combining \cref{cor:derived_lci_finite_tor-amplitude} and \cite[Proposition 7.8]{Porta_Yu_Derived_Hom_spaces}, we conclude that $\iota_* \cO_{Z_{i,\tau}^{m_i}}$ is a perfect complex on $X_{i,\tau}^{m_i}$.
\end{proof}

\begin{definition} \label{def:quantum_K-invariants_with_multiplicities}
	Given any A-graph $(\tau,\beta)$, $\mathbf m_i=(m_i)_{i\in T_\tau}$ with $m_i\in\bbN_{>0}$, and $\bZ=(Z_i)_{i\in T_\tau}$ with $Z_i\subset X$ closed lci, we define the associated quantum K-invariants of $X$ with multiplicities
	\[
	K^X_{\tau,\beta,\mathbf m}(\bZ)\coloneqq \st_*\bigg(\bigotimes_{i\in T_\tau}(\ev_i^{m_i})^*\cO_{Z_{i,\tau}^{m_i}}\bigg)\in K_0(\oM_\tau),\]
	where the pullback $(\ev_i^{m_i})^*$ is well-defined by \cref{lem:Zi}, and \Cref{prop:boundedness} ensures that the image of $\st_*$ lies in $\Cohb(\oM_\tau)$, thus gives an element in $K_0(\oM_\tau)$. 
\end{definition}

The quantum K-invariants with multiplicities satisfy the following list of natural geometric properties exactly parallel to the previous section.
As the proofs are also parallel to the previous section, we do not repeat them here.

\subsection{Mapping to a point}

\begin{proposition} \label{prop:mapping_to_a_point_with_multiplicity}
	Let $(\tau,0)$ be any A-graph where $\beta$ is 0.
	Let $p_1, p_2$ be projections
	\[\begin{tikzcd}
	X\times \oM_\tau \rar{p_1} \dar{p_2} & X\\
	\oM_\tau & \phantom{X}.
	\end{tikzcd}\]
	Let $q$ be the composition
	\[X\times\oM_\tau\hookrightarrow X\times\oMpre_\tau\hookrightarrow\bfMap_{\oMpre_\tau}\Bigl((\oCpre_\tau)_{(s_i^{m_i})}, X\times \oMpre_\tau\Bigr)\simeq X_{i,\tau}^{m_i}.\]
	Given any $\mathbf m=(m_i)_{i\in T_\tau}$ and $\bZ=(Z_i)_{i\in T_\tau}$ as in \cref{def:quantum_K-invariants_with_multiplicities}, for each $i\in T_\tau$, we have
	\[K^X_{\tau,\beta,\mathbf m}(\bZ) = p_{2*}\Big(\big({\textstyle\bigotimes _{i\in T_\tau}} q^*\cO_{Z_{i,\tau}^{m_i}}\big)\otimes p_1^*\lambda_{-1}\big((\bbT\an_X \boxtimes R^1\pi_*\cO_{\oC_\tau})^\vee\big) \Big),\]
	where $\lambda_{-1}(F)\coloneqq\sum_{i}(-1)^i\wedge^i F$.
\end{proposition}

\subsection{Products}

\begin{proposition}
	Let $(\tau_1, \beta_1)$ and $(\tau_2, \beta_2)$ be two A-graphs.
	Let $(\tau_1 \sqcup \tau_2, \beta_1 \sqcup \beta_2)$ be their disjoint union.
	Given any $\mathbf m=(m_i)_{i\in T_{\tau_1}}$, $\bZ=(Z_i)_{i\in T_{\tau_1}}$, $\mathbf n=(n_j)_{j\in T_{\tau_2}}$, $\bW=(W_j)_{j\in T_{\tau_2}}$ as in \cref{def:quantum_K-invariants_with_multiplicities}, we have
	\[K^X_{\tau_1\sqcup\tau_2,\beta_1\sqcup\beta_2,\mathbf m\sqcup\mathbf n}(\bZ\sqcup\bW)=K^X_{\tau_1,\beta_1,\mathbf m}(\bZ)\boxtimes K^X_{\tau_2,\beta_2,\mathbf n}(\bW).\]
\end{proposition}

\subsection{Cutting edges}

\begin{proposition}
	Let $(\sigma,\beta)$ be an A-graph obtained from $(\tau,\beta)$ by cutting an edge $e$ of $\tau$.
	Let $v,w$ be the two tails of $\sigma$ created by the cut.
	Consider the following commutative diagram
	\[ \begin{tikzcd}
	\oM_\tau \rar{d} & \oM_\sigma\\
	\R \oM(X, \tau, \beta) \uar{\st_\tau} \arrow{r}{c} \arrow{d}{\ev_e} & \R \oM(X, \sigma, \beta) \uar{\st_\sigma} \arrow{d}{\ev_v \times \ev_w} \\
	X \arrow{r}{\Delta} & X \times_S X.
	\end{tikzcd} \]
	Given any $\mathbf m=(m_i)_{i\in T_\tau}$ and $\bZ=(Z_i)_{i\in T_\tau}$ as in \cref{def:quantum_K-invariants_with_multiplicities}, let
	\[\mathbf m'\coloneqq\mathbf m\sqcup(m_v=1, m_w=1).\]
	We have
	\[d_*K^X_{\tau,\beta,\mathbf m}(\bZ)=K^X_{\sigma,\beta,\mathbf m'}(\bZ\otimes\Delta_*\cO_X),\]
	where the right hand side means
	\[\st_*\bigg(\Big({\textstyle\bigotimes}_{i\in T_\tau}(\ev_i^{m_i})^*\cO_{Z_{i,\sigma}^{m_i}}\Big)\otimes(\ev_v\times\ev_w)^*\Delta_*\cO_X\bigg)\in K_0(\oM_\sigma).\]
\end{proposition}

\subsection{Forgetting tails}

\begin{proposition}
	Let $(\sigma,\beta)$ be an A-graph obtained from $(\tau,\beta)$ by forgetting a tail $t$.
	Let $\Phi\colon\oM_\tau\to\oM_\sigma$ be the forgetful map from $\tau$-marked stable curves to $\sigma$-marked stable curves.
	Given any $\mathbf m=(m_i)_{i\in T_\sigma}$ and $\bZ=(Z_i)_{i\in T_\sigma}$ as in \cref{def:quantum_K-invariants_with_multiplicities}, and any positive integer $m_t$, let
	\[\mathbf m'\coloneqq\mathbf m\sqcup(m_t).\]
	We have	
	\[K^X_{\tau,\beta,\mathbf m'}(\bZ\sqcup(X))=\Phi^*K^X_{\sigma,\beta,\mathbf m}(\bZ).\]
	\end{proposition}

\subsection{Contracting edges}

\begin{proposition}
	Let $(\sigma,\beta)$ be an A-graph where $\sigma$ is obtained from a modular graph $\tau$ by contracting an edge (possibly a loop) $e$.
	We follow the notations of \cref{sec:contracting_edges}.
	Let $\Phi\colon\oM_\tau\to\oM_\sigma$, $\Psi_l^i\colon\oM_{\tau_l^i}\to\oM_\tau$ and $\Omega\colon\oM_{\tsigma}\to\oM_\sigma$ be the induced maps on the moduli stacks of stable curves.
	Given any $\mathbf m=(m_v)_{v\in T_{\tsigma}}$ and $\bZ=(Z_v)_{v\in T_{\tsigma}}$ as in \cref{def:quantum_K-invariants_with_multiplicities}, we have
	\[\Phi^*\Omega_* K^X_{\tsigma,\beta,\mathbf m}(\bZ) = \sum_l (-1)^{l+1} \sum_{i,j} \Psi_{l*}^i K^X_{\tau^i_l,\beta^{i_j}_l,\mathbf m}(\bZ).\]
\end{proposition}

\section{Appendix: Flatness in derived \texorpdfstring{$k$}{k}-analytic geometry} \label{sec:flatness}

\begin{definition}
	Let $f \colon X \to Y$ be a morphism of derived \kanal spaces and let $\cF \in \Coh^+(X)$.
	We say that $\cF$ is \emph{flat relative to $Y$} if $\cF$ is flat as $f\inv \cO_Y\alg$-module.
\end{definition}

\begin{lemma} \label{lem:equivalent_formulation_flatness}
	Let $f \colon X \to Y$ be a morphism of derived \kanal spaces.
	Let $\cF \in \Coh^+(X)$.
	The following are equivalent:
	\begin{enumerate}
		\item $\cF$ is flat relative to $Y$.
		\item $\cF$ has tor-amplitude $[0,0]$ relative to $Y$ in the sense of \cite[Definition 7.1]{Porta_Yu_Derived_Hom_spaces}.
		\item For every $\cG \in \Cohh(Y)$, $\cF \otimes_{\cO_X} f^*(\cG)$ belongs to $\Cohh(X)$.
		\item Consider the pullback square
		\[ \begin{tikzcd}
		X_0 \arrow{r}{i} \arrow{d}{f_0} & X \arrow{d}{f} \\
		\trunc(Y) \arrow[hook]{r}{j} & Y ,
		\end{tikzcd} \]
		where $j$ is the inclusion of the truncation.
		Then $i^*(\cF)$ is flat relative to $\trunc(Y)$.
	\end{enumerate}
\end{lemma}

\begin{proof}
	We start by proving the equivalence (1) $\Leftrightarrow$ (4).
	Write $X = (\cX, \cO_X)$ and $Y = (\cY, \cO_Y)$.
	Since $j$ is a closed immersion, \cite[Propositions 3.17(iii) and 6.2(iv)]{Porta_Yu_Derived_non-archimedean_analytic_spaces} imply that
	\[ \begin{tikzcd}
	f\inv \cO_Y\alg \arrow{r} \arrow{d} & \cO_X\alg \arrow{d} \\
	f\inv( \pi_0( \cO_Y\alg ) ) \arrow{r} & \cO_{X_0}\alg
	\end{tikzcd} \]
	is a pushout square in $\CAlg_k(\cX)$.
	Then $\cF$ is flat as $f\inv \cO_Y\alg$-module if and only if the base change $\cF \otimes_{f\inv \cO_Y\alg} f\inv( \pi_0( \cO_Y\alg ) )$ is flat as $f\inv( \pi_0(\cO_Y\alg) )$-module.
	The derived base change implies that
	\[ \cF \otimes_{f\inv \cO_Y\alg} f\inv( \pi_0( \cO_Y\alg ) ) \simeq \cF \otimes_{\cO_X\alg} \cO_{X_0}\alg . \]
	Therefore (1) is equivalent to (4).
	
	Next, we prove the equivalence (4) $\Leftrightarrow$ (2).	
	Since (4) is local on both $X$ and $Y$, we may assume that $X$ and $Y$ are derived $k$-affinoid spaces.
	Write
	\[ A \coloneqq \Gamma(X, \cO_X\alg) , \quad A_0 \coloneqq \Gamma(X_0, \cO_{X_0}\alg), \quad B \coloneqq \Gamma(Y, \cO_Y\alg) , \quad M \coloneqq \Gamma(X, \cF) . \]
	Note that $i^*(\cF)$ is flat relative to $\trunc(Y)$ if and only if $i^*(\cF) \in \Cohh(X_0) \simeq \Cohh(\trunc(X))$ and as a coherent sheaf on $\trunc(X)$ it is flat relative to $\trunc(Y)$ in the underived sense.
	Using \cite[Theorem 3.4]{Porta_Yu_Derived_Hom_spaces}, we conclude that $i^*(\cF)$ is flat relative to $\trunc(Y)$ if and only if the module
	\[ \Gamma(X_0, i^*(\cF)) \simeq M \otimes_A A_0 \]
	is flat as $\pi_0(B)$-module.
	By \cite[Proposition 4.2]{Porta_Yu_Derived_Hom_spaces}, the diagram
	\[ \begin{tikzcd}
	B \arrow{d} \arrow{r}{\phi} & A \arrow{d} \\
	\pi_0(B) \arrow{r} & A_0
	\end{tikzcd} \]
	is a pushout diagram in $\CAlg_k$.
	This implies that
	\[ M \otimes_A A_0 \simeq M \otimes_B \pi_0(B) . \]
	If (4) is satisfied, then $M \otimes_B \pi_0(B)$ is flat as $\pi_0(B)$-module.
	Hence, it has tor-amplitude $[0,0]$ as $\pi_0(B)$-module, which implies that $M$ has tor-amplitude $[0,0]$ as $B$-module, and hence $M$ is flat as $B$-module.
	Thus (2) holds.
	Conversely, if (2) is satisfied, then $M$ is flat as $B$-module, hence $M \otimes_A A_0$ is flat as $\pi_0(B)$-module, so (4) holds.
	
	We finally prove the equivalence (2) $\Leftrightarrow$ (3).
	Once again, we can assume $X$ and $Y$ to be derived $k$-affinoid spaces.
	Keeping the same notations as above, let $\cG \in \Cohh(Y)$ and let $N \coloneqq \Gamma(Y, \cG)$.
	By \cite[Theorem 3.1]{Porta_Yu_Derived_Hom_spaces}, $\cF \otimes_{\cO_X} f^*( \cG )$ is discrete if and only if $M \otimes_A \phi^*(N)$ is discrete.
	On the other hand, $M \otimes_A \phi^*(N)$ is discrete if and only if
	\[ \phi_*(M \otimes_A \phi^*(N)) \simeq \phi_*(M) \otimes_B N \]
	is discrete.
	This completes the proof.
\end{proof}

\begin{proposition} \label{prop:stability_flatness}
	Let
	\[ \begin{tikzcd}
	X' \arrow{d}{q} \arrow{r}{g} & X \arrow{d}{p} \\
	Y' \arrow{r}{f} & Y
	\end{tikzcd} \]
	be a pullback square of derived \kanal spaces.
	Let $\cF \in \Coh^+(X)$.
	If $\cF$ is flat relative to $Y$, then $g^*(\cF)$ is flat relative to $Y'$.
\end{proposition}

\begin{proof}
	Consider the extended diagram
	\[ \begin{tikzcd}
	X'_0 \arrow{d}{q_0} \arrow{r}{i} & X' \arrow{d}{q} \arrow{r}{g} & X \arrow{d}{p} \\
	\trunc(Y') \arrow{r}{j} & Y' \arrow{r}{f} & Y .
	\end{tikzcd} \]
	By \cref{lem:equivalent_formulation_flatness}, it is enough to check that $i^*(\cF)$ is flat.
	We factor $\trunc(Y') \to Y$ through $\trunc(Y)$, which yields the following diagram
	\[ \begin{tikzcd}
	X'_0 \arrow{d}{q_0} \arrow{r}{i'} & X_0 \arrow{r}{i''} \arrow{d}{p_0} & X \arrow{d}{p} \\
	\trunc(Y') \arrow{r} & \trunc(Y) \arrow{r} & Y .
	\end{tikzcd} \]
	Since $\cF$ is flat relative to $Y$, \cref{lem:equivalent_formulation_flatness} shows that $(i'')^*(\cF)$ is flat relative to $\trunc(Y)$.
	We are therefore left to prove the proposition assuming that $X$, $Y$ and $Y'$ are underived.
	Using the theory of formal models (see \cite{Bosch_Formal_and_rigid_geometry_II}), we see that $\pi_0(g^*(\cF))$ is flat relative to $Y'$.
	Therefore, we are left to check that $g^*(\cF) \in \Cohh(X')$.
	Working locally, we can assume that $Y$ and $Y'$ are affinoid and that the map $Y' \to Y$ factors as
	\[ \begin{tikzcd}
	Y' \arrow[hook]{r}{i} & \bD^n_Y \arrow{r}{\pi} & Y ,
	\end{tikzcd} \]
	where $i$ is a closed immersion and $\pi$ is the canonical projection.
	We can therefore decompose the pullback square as
	\[ \begin{tikzcd}
	X' \arrow{r} \arrow{d}{q} & \bD^n_X \arrow{r}{\pi_X} \arrow{d}{p'} & X \arrow{d}{p} \\
	Y' \arrow[hook]{r}{i} & \bD^n_Y \arrow{r}{\pi_Y} & Y .
	\end{tikzcd} \]
	As $i$ is a closed immersion, \cite[Proposition 3.17(iii) and 6.2(iv)]{Porta_Yu_Derived_non-archimedean_analytic_spaces} imply that it is enough to check that $\pi_X^*(\cF)$ is flat relative to $\bD^n_Y$.
	For this, it is enough to check that $\pi_X^*(\cF)$ is discrete.
	By \cite[Proposition 4.10]{Porta_Yu_Derived_Hom_spaces}, we see that $\pi_X$ is flat, and therefore $\pi_X^*$ is $t$-exact.
	The conclusion follows.
\end{proof}

\begin{corollary} \label{cor:stability_flatness}
	Let
	\[ \begin{tikzcd}
	X' \arrow{r} \arrow{d}{q} & X \arrow{d}{p} \\
	Y' \arrow{r} & Y
	\end{tikzcd} \]
	be a pullback square of derived \kanal stacks.
	If $p$ is flat, then so is $q$.
\end{corollary}

\bibliographystyle{plain}
\bibliography{dahema}

\end{document}